\documentclass[11pt ,twoside]{amsart}  %Normally report
\usepackage[left=3.1cm,right=3.1cm,top=3.1cm,bottom=3.1cm]{geometry}
\usepackage{graphicx}
\usepackage{multirow}
\usepackage{hyperref}
\usepackage{bbm}
\usepackage{amsmath,amssymb,amsfonts, amsthm,epsfig}
\usepackage{verbatim}
\usepackage{enumerate}
\numberwithin{equation}{section}
\usepackage{mathtools}
\mathtoolsset{showonlyrefs}

\newtheorem{thma}{Theorem}[section]
\newtheorem{lemma}[thma]{Lemma}

\newtheorem{defi}[thma]{Definition}

\renewcommand{\l}{{\lambda}}

\begin{document}

\title{Coupling Functions for Domino tilings of  Aztec diamonds}

\begin{abstract}
The inverse Kasteleyn matrix of a bipartite graph holds much information about the perfect matchings of the system such as local statistics which can be used to compute local and global asymptotics.  In this paper, we consider three different weightings of domino tilings of the Aztec diamond and show using recurrence relations, we can compute the inverse Kasteleyn matrix.  These weights are the one-periodic weighting where the horizontal edges have one weight and the vertical edges have another weight, the $q^{\mathrm{vol}}$ weighting which corresponds to multiplying the product of tile weights by $q$ if we add a `box' to the height function and the two-periodic weighting which exhibits a flat region with defects in the center.

\end{abstract}

\author{Sunil Chhita}
\thanks{Department of Mathematics, Royal Institute of Technology (KTH), Stockholm, Sweden. E-mail: chhita@kth.se.  The support of the Knut and Alice Wallenberg Foundation grant KAW 2010:0063 is gratefully acknowledged.}

\author{Benjamin Young}
\thanks{Department of Mathematics, University of Oregon, Eugene, USA. Email: bjy@uoregon.edu. }

\maketitle
\tableofcontents

\section{Introduction}

\subsection{Terminology}
Domino tilings of bounded lattice regions have been extensively researched during the last twenty years.   These tilings are the same as \emph{perfect matchings} of a bounded portion $G$ of the dual square lattice,  in the following way: a \emph{matched edge} corresponds to a domino; the fact that the dominos do not overlap means that no two matched edges share a vertex, and the fact that the dominos cover the entire region means that each vertex in the region is covered by a matched edge.  In the statistical mechanics literature, one speaks of \emph{dimer covers} rather than \emph{perfect matchings}, and \emph{dimers} rather than \emph{matched edges}.

The most well-studied example of such a model is domino tilings of the \emph{Aztec diamond} which was introduced in \cite{EKLP:92}.  Here, one tiles the region $\{(x,y): |x|+|y|\leq n+1\}$ with 2 by 1 rectangles where $n$ is the size of the Aztec diamond.  There are other examples of the theory, but they involve replacing the graph $G$ with a different one, such as the regular square-octagon lattice (giving the so-called \emph{diabolo tilings}) or the hexagonal mesh (giving \emph{lozenge tilings}).

By giving each edge a multiplicative weight, we can consider \emph{random} dimer coverings: the probability of each covering is proportional to the product of the edge weights of the dimer covering.  The corresponding discrete probability space is called the \emph{dimer model}.  
If the graph $G$ is bipartite (as it shall be for the rest of this paper) then each dimer covering can be encoded by a three dimensional discrete surface, where the third coordinate is derived from the specific dimer covering and is called the \emph{height function} \cite{Thu:90}.  For random tilings, \cite{CKP:01, KO:07} showed that with probability tending to one, the height function of a  randomly tiled large bounded region tends to a deterministic \emph{limit shape}.  This shape is not smooth over the entire region: typically, there are macroscopic regions wherein the tiling is ``frozen'' (i.e. exhibits deterministic correlations), called \emph{facets}; as such the measure is often said to be in a \emph{solid state} here~\cite{KOS:06}.  
Outside of the facets, the correlations between pairs of dimers are only mesoscopic, tending to zero as the dimers move farther apart. If the decay is polynomial, the measure is said to be \emph{liquid}; if it is exponential it is said to be \emph{gaseous}~\cite{KOS:06}.  Not all tilings possess a gaseous region; however, all but the most degenerate have liquid regions.  The limiting height function is smooth in these regions.  Figure~\ref{fig:pretty pictures of Aztec diamonds1} shows two random tilings of relatively large Aztec diamonds. 

\subsection{Local asymptotics for nice regions}

For particular bounded regions, one approach to study these models uses an interlaced particle system which can be derived from the underlying tiling \cite{Joh:05, BF:08}.  Using the Lindstr\"{o}m-Gessel-Viennot theorem \cite{Joh:06,Ste:90} combined with the Eynard-Mehta theorem (e.g. see \cite{Bor:11}), it is often possible to find the correlation kernel for a determinantal process and compute finer statistics for the model\cite{Joh:05, BF:08}.  Using these statistics, one can study the fluctuations between the interface of the solid and liquid regions when the system size gets large.  Amazingly, these fluctuations have the same distributions arising from the study of eigenvalues of random matrices (see for example \cite{Joh:05, JN:06, FF:11}).  Furthermore, one can even change the boundary conditions of the underlying tiling problem to find more degenerate kernels which also appear in the random matrix literature, for example, see \cite{AJvM:11}.

For bipartite graphs, the \emph{Kasteleyn matrix} is a signed weighted adjacency matrix indexed by the white and black vertices of the graph \cite{Kas:61}.  The inverse of the Kasteleyn matrix, known as the \emph{inverse Kasteleyn matrix}, for bipartite graphs provides much information about the model -- by \cite{Ken:97} the edges form a determinantal process with the kernel given by the inverse Kasteleyn matrix.  Hence, by knowing the inverse Kasteleyn matrix for a bipartite graph one can compute all finite, local and global asymptotics of the edges in the dimer model.   For lozenge tilings, the interlaced particle system kernel can be used to compute the inverse Kasteleyn matrix as the particle system kernel and the inverse Kasteleyn matrix are in bijection \cite{Pet:12}.   However, for domino tilings on the Aztec diamond, the most natural kernel from the interlaced particle system contains different information to the inverse Kasteleyn matrix.  Although, one can derive the particle system correlation 
kernel from the inverse Kasteleyn matrix, the particle system correlation kernel gives a better description of the interface between the unfrozen and frozen regions, see \cite{CJY:12}.  By knowing both the interlaced particle system and the inverse Kasteleyn matrix, we believe that one understands the full asymptotic picture of the system. 

\subsection{Purpose}
The aim of this paper to highlight an elementary procedure which allows one to compute the correlation kernel of the determinantal process associated to the edges of tilings of Aztec diamonds.  For pedagogical reasons, we do this first for the most well-understood instance of the dimer model (thereby recovering the work of~\cite{Hel:00}), followed by two, substantially more complicated new settings: the so-called $q^{\mathrm{vol}}$ weighting (Section~\ref{section q-vol for the Aztec Diamond}, and the two-periodic weighting (Section~\ref{section Diablo Tilings on the Fortress}) which includes as a special case the uniform measure on diabolo tilings on a \emph{fortress graph}~\cite{Pro:03}.  The two-periodic weighting has a $2$ by $2$ fundamental domain which is defined in \cite{KOS:06}.
%Although this paper does not study any specific asymptotics, the asymptotics of the kernels from the $q^{\mathrm{vol}}$ weighting and two-periodic weighting of the Aztec diamond are both current works in progress.

For the $q^{\mathrm{vol}}$ weighting, large random tilings of the Aztec diamond possess a limit shape  when $q \to 1^-$ as the  system size tends to infinity.  The existence of a limit shape is due to  the results from  \cite{KO:07}.  However, when  $q\to 1^-$ and the $a \to \infty$, the results from \cite{KO:07} will no longer apply but simulations seem to suggest that there may be a limit shape and possibly interesting local and global asymptotic behavior.  Figure~\ref{fig:pretty pictures of Aztec diamonds2} shows relatively large tilings with $q^{\mathrm{vol}}$ weighting.

Large random tilings of the two-periodic weighting of the Aztec diamond feature all three phases where the limit shape is described by an `octic'  curve.  An explicit formula for this curve, at the special case corresponding to the uniform measure of tiling diabolos on the fortress, is given in \cite{Pro:03} and is derived in \cite{KO:07} using general machinery. Figure~\ref{fig:pretty pictures of Aztec diamonds3} shows a relatively large random tiling of a two-periodic weighting of the Aztec diamond.  By having an expression for the inverse Kasteleyn matrix for the two-periodic weighted Aztec diamond, it may be possible to study this model on all the phase interfaces.  Computing the interlaced particle system kernel for the two-periodic weighting of the Aztec diamond using the Lindstr\"{o}m-Gessel-Viennot theorem \cite{Joh:06,Ste:90} combined with the Eynard-Mehta theorem (e.g. see \cite{Bor:11}) seems somewhat complicated -- one can either proceed by inverting either a block LGV matrix or a block Toeplitz matrix.

\subsection{Explicit inversion of Kasteleyn matrices}
We turn now to a discussion of our methods.  It is possible to derive the inverse Kasteleyn matrix for domino tilings of the Aztec diamond with weight 1 for horizontal tiles and weight $a$ for vertical tiles (one-periodic weighting).  We also show that it is possible to extend part of the method to the $q$-analog of domino tilings of the Aztec diamond ($q^{\mathrm{vol}}$ weighting).  The height function of a domino tiling can be viewed as the surface of a pile of \emph{Levitov blocks }\cite{Lev:89,Lev:90}; one can modify the edge weights so that the weight of each tiling is proportional to $q^{\#\{\text{Levitov blocks}\}}$ which is $q^{\mathrm{vol}}$ weighting.  Finally, we show that it is possible to compute the inverse Kasteleyn matrix with a two-periodic weighting of the Aztec diamond and do so for an Aztec diamond of size $4m$ for $m\in \mathbb{N}$. 

It is often quite difficult to invert a matrix whose entries have parameters; indeed, the typical methods in the literature involve first orthoganalizing, so that the matrix to be inverted is diagonal~\cite{BR:05}.  However, in the particular case of the Aztec diamond, one can make some progress by computing a multivariate generating function for the entries of $K^{-1}$.  This is possible essentially using graphical transformations similar to those used in~\cite{EKLP:92} which is the same procedure used in the generalized domino shuffle~\cite{Pro:03} to compute the weights in the transitional steps of the shuffling algorithm. 

We use three relations among the entries of $K^{-1}$.  The first two are the matrix identities $K \cdot K^{-1} = \mathbbm{I}$, $K^{-1} \cdot K = \mathbbm{I}$; the third is a recurrence relation in $n$, the order of the Aztec diamond, generated from the generalized domino shuffling algorithm which determines only those elements of $K^{-1}$ which correspond to two dimers on the \emph{boundary} of the region.  We solve this recurrence by computing the ordinary power-series generating function for its coefficients, which we call the \emph{boundary generating function}.

Having done this, the entries of $K^{-1}$ are slightly overdetermined.  As such, we treat the equations $K \cdot K^{-1} = \mathbbm{I}$, $K^{-1} \cdot K = \mathbbm{I}$ as \emph{linear recurrence relations}, for which the \emph{boundary generating function} serves as a boundary condition.

\subsection{Details}

The heuristic in the preceding sections drastically oversimplifies the intricacy of the combinatorics involved in these calculations, which consume the bulk of this paper.  We also found that the computation of the boundary generating function of the two-periodic weighting was intractable  without the use of computer algebra.  Here is a description of the calculation in somewhat greater detail.  
   The recurrences $K . K^{-1} = K^{-1}.K = I$ allow us to write each entry of the inverse Kasteleyn matrix as a linear combination of entries of $K^{-1}(\mathtt{w}, \mathtt{b})$ where $\mathtt{b}$ and $\mathtt{w}$ are black and white vertices on the boundary.  
To find $K^{-1}(\mathtt{w}, \mathtt{b})$ where $\mathtt{b}$ and $\mathtt{w}$ are black and white vertices on the boundary, we first find $Z(\mathtt{w},\mathtt{b})/Z$ where $Z$ is the partition function (the sum of the weights of all tilings) and  $Z(\mathtt{w},\mathtt{b})$ is the partition function of the tiling with the vertices $\mathtt{b}$ and $\mathtt{w}$ removed from the Aztec diamond.  

As removing two vertices $\mathtt{w}$ and $\mathtt{b}$ from the boundary does not change the so-called \emph{Kasteleyn orientation}, we are able 
to recover $K^{-1}(\mathtt{w},\mathtt{b})$ from $Z(\mathtt{w},\mathtt{b})/Z$.  The boundary recurrence relation gives a relation for $Z(\mathtt{w},\mathtt{b})/Z$ as a linear combination of $Z(\mathtt{w}_1,\mathtt{b}_1)/Z$ for a smaller Aztec diamond where $w$, $w_1$ and $b$, $b_1$ are white and black vertices on the boundaries of their corresponding Aztec diamonds.  To solve this recurrence, we use a generating function approach and hence find $K^{-1}(\mathtt{w},\mathtt{b})$ for $\mathtt{b}$ and $\mathtt{w}$ are black and white vertices on the boundary.  % However, for the $q^{\mathrm{vol}}$ weighting, it is simpler to solve the boundary recurrence relation directly using double contour integral which we obtained from its boundary generating function.   

For the one-periodic and $q^{\mathrm{vol}}$ weighting, we can use the three sets of recurrence relations mentioned above to find  formulas for their inverse Kasteleyn matrices.  For the one-periodic weighting, we derive the formula as a generating function.  For the $q^{\mathrm{vol}}$ weighting, we guessed a double contour integral based on our methods and verify that this guess is correct.  Indeed, sending $q \to 1$, we find a double contour integral formula of the inverse Kasteleyn matrix in the one-periodic case, see \cite{CJY:12} for a comparison.  

 For the two-periodic weighting of the Aztec diamond, we must use a further two recurrence relations; $K^* \cdot K  \cdot K^{-1} = K^* \cdot \mathbbm{I}$ and $K^{-1} \cdot K \cdot K^* = \mathbbm{I}\cdot K^*$ due to the complexity of the model; of course, these are completely equivalent and in principle carry no new information; however, they are simpler in a certain sense as they involve the \emph{discrete Laplacian}.  
Additionally, the boundary recurrence relation for the two-periodic weighting of the Aztec diamond has order 4 with the size of the Aztec diamond and is dependent on the parity of the white and black vertices on the boundary.  This leads to a matrix equation for the recurrence relations describing the boundary generating function.  We only find the solution of the matrix recurrence relation when the size of the Aztec diamond is equal to $4m$, though in principle the main obstacles to handling the other sizes of the Aztec diamonds are the difficulty and length of the computations.  This leads to a formula for the inverse Kasteleyn matrix when the size of the Aztec diamond is equal to $4m$.
\begin{figure}
 \caption{Random Simulations of the Aztec diamond of size 100 with one-periodic weight for $a=1$ (left) and $a=1/2$ (right)(see Section~\ref{section One periodic case} for a description of edge weights).}
\includegraphics[height=3in]{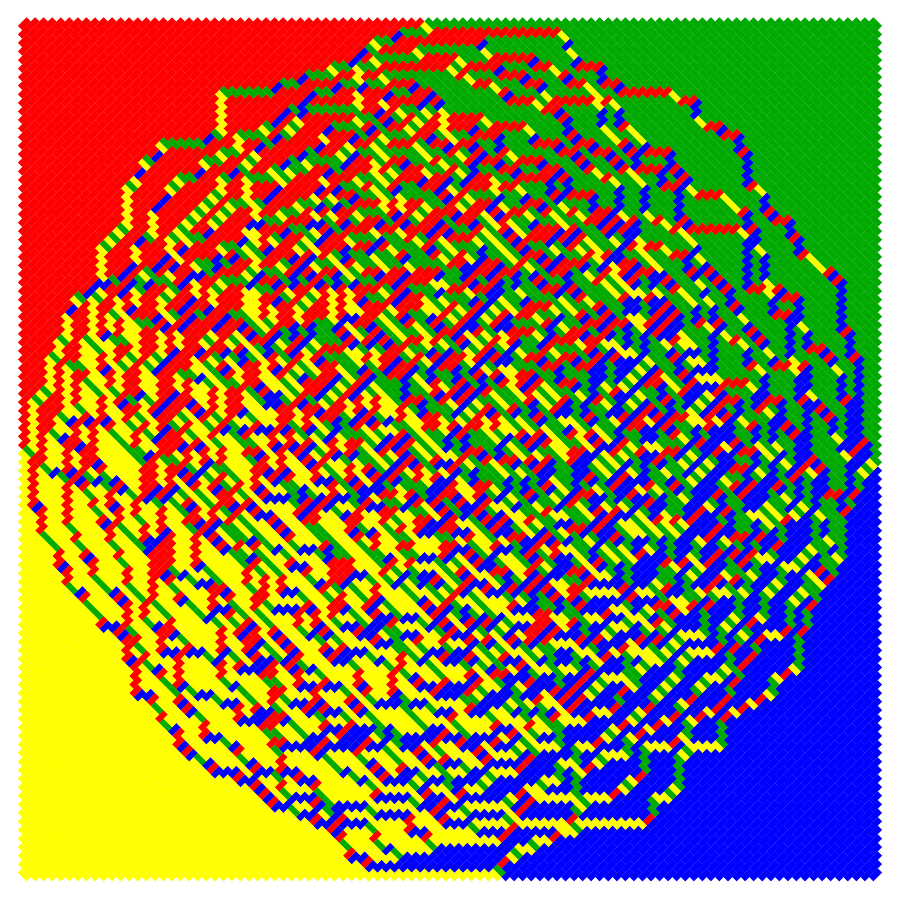}
\includegraphics[height=3in]{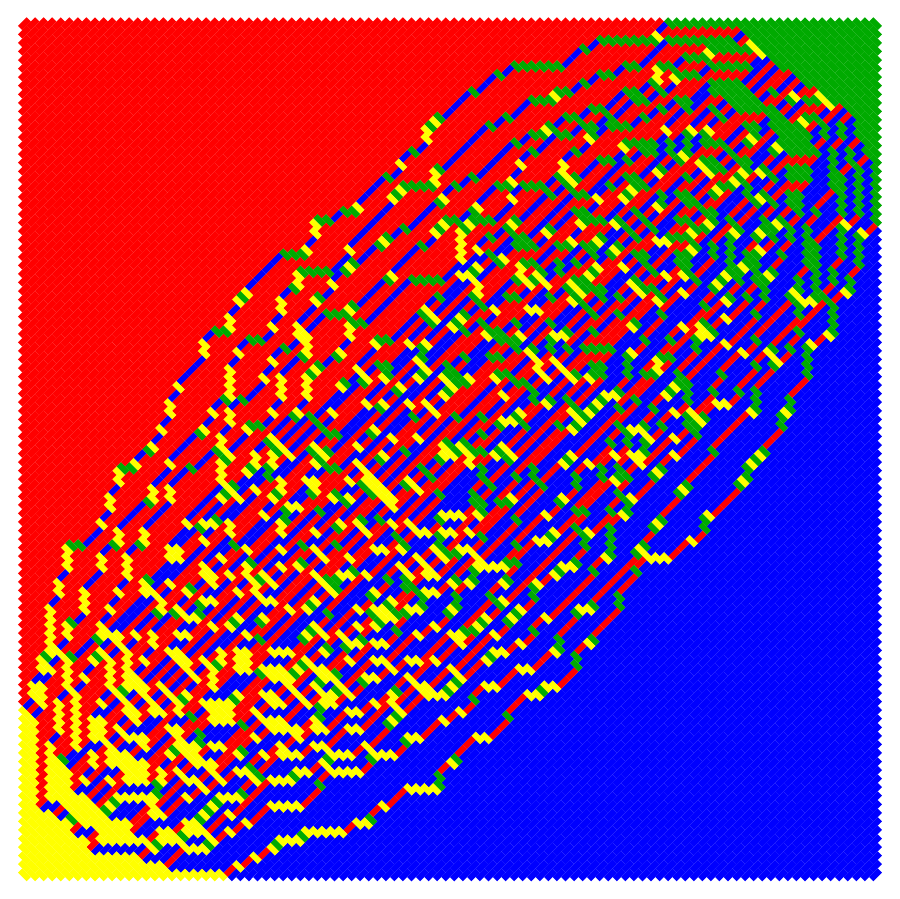}
\label{fig:pretty pictures of Aztec diamonds1}

\end{figure}

\begin{figure}
\caption{Random simulations of the Aztec diamond of size 200 with $q^{\mathrm{vol}}$ weighting (see Section~\ref{section q-vol for the Aztec Diamond}).  The top picture has weights $q=0.99$ and $a=1$.  The picture below has $q=0.98$ and $a=10$.  }
\includegraphics[height=4in]{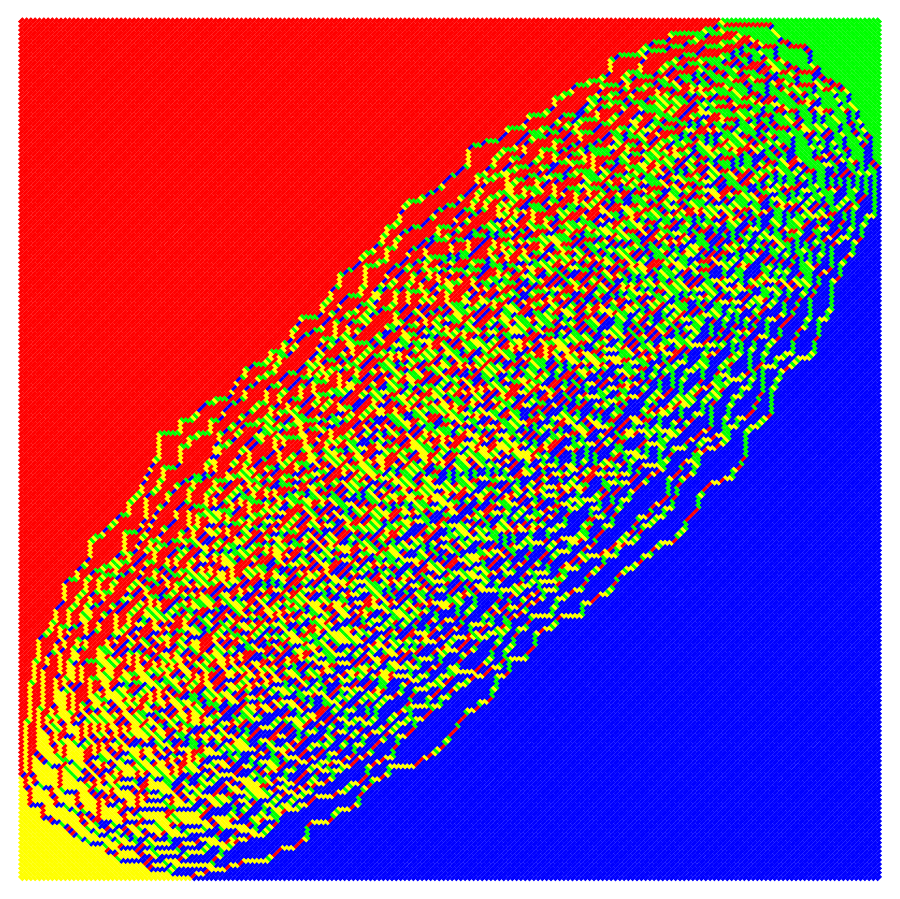}
\includegraphics[height=4in]{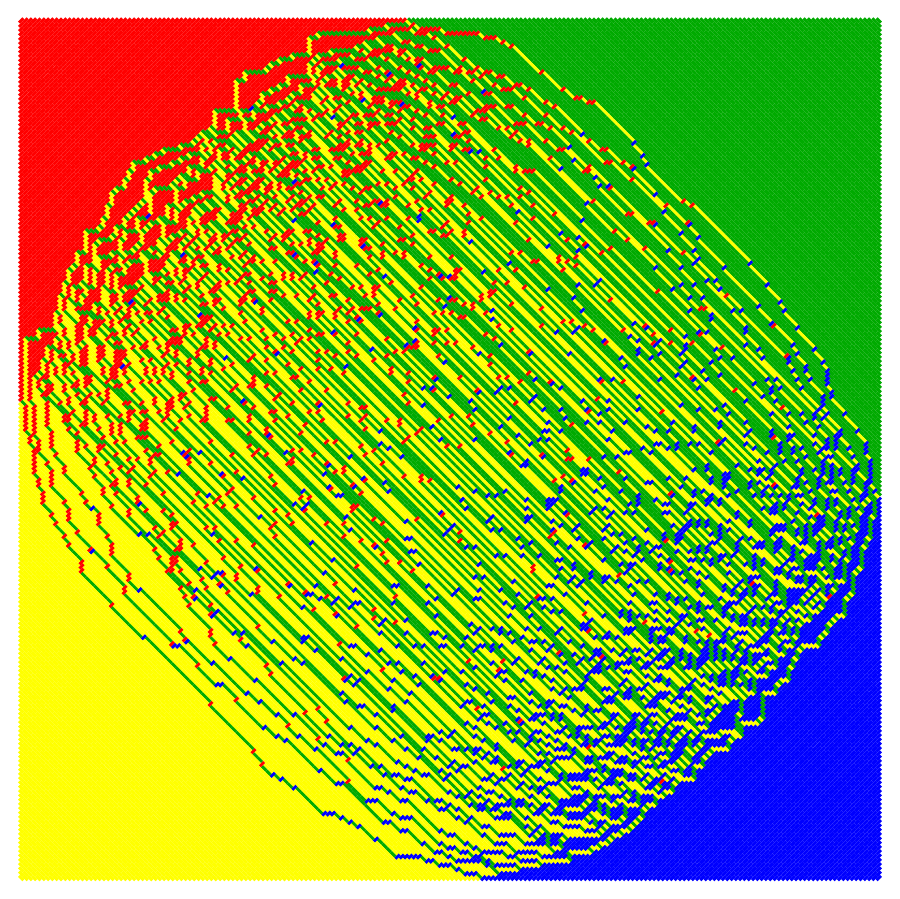}
\label{fig:pretty pictures of Aztec diamonds2}
\end{figure}

\begin{figure}
 \caption{Random simulations of the Aztec diamond of size 200 with two-periodic weights (see Section~\ref{section Diablo Tilings on the Fortress}.  The first picture has $a=1/2$ and $b=1$ with 8 colors. The second picture is the same tiling as the first but contains four colors to highlight the three phases for black and white printing. }
\includegraphics[height=4in]{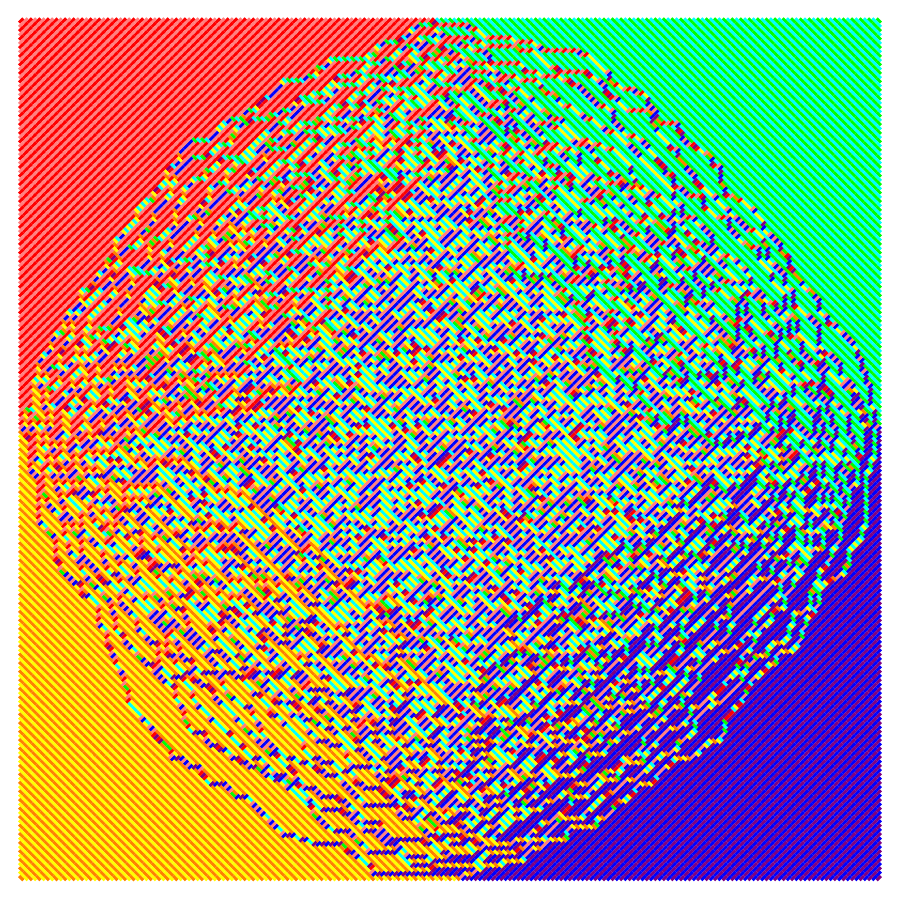}
\includegraphics[height=4in]{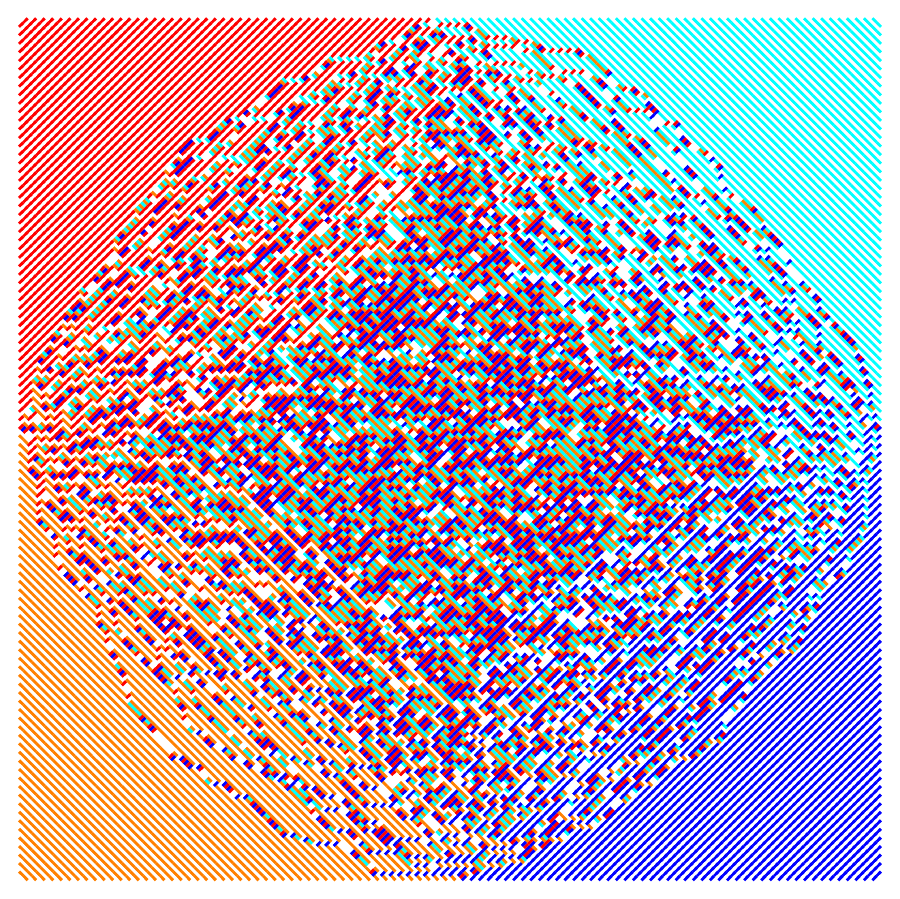}
\label{fig:pretty pictures of Aztec diamonds3}
  \end{figure}

\subsection{Remarks about Asymptotics}

As remarked above, this paper might provide a gateway to computing fine asymptotics of the $q^{\mathrm{vol}}$ and two-periodic weightings of the Aztec diamond. That is, one will hopefully be able to compute global correlations and local correlations at the phase boundaries using the formulas found in this paper.

From simulations, the $q^{\mathrm{vol}}$ weighting seems to exhibit interesting limiting behavior  when $q \to 1^-$  and also when $q \to 1^-$ and $a \to \infty$ as $n \to \infty$ simultaneously. We believe that a possible parametrization would be to set $a=e^{c/2}$ and $q=e^{-c/n}$ where $n$ tends to infinity.  Although it is possible to derive to the limit shape using \cite{KO:07}, we think  our formula could be used to find the height fluctuations in the unfrozen region when $c>0$. When $c=0$, the model is equivalent to the one-periodic weighting of the Aztec diamond and so the height fluctuations are governed by the so-called \emph{Gaussian Free field} (details of this process can be found in  \cite{she:05}) as shown in \cite{CJY:12}. When $c$ tends to infinity, we suspect that the unfrozen region is a \emph{flat square} given by alternating (diagonal) columns of east and west dominos.  That is, for a rescaled Aztec diamond with corners given by $(0,0),(0,1),(1,0)$ and $(1,1)$, the unfrozen region is given by $(1/2,0),(0,1/2),(1/2,1)$ and $(1,1/2)$. From initial computations, the asymptotic analysis to find these height fluctuations is encouraging and is current work in progress.

 It may also be possible to derive the $q^{\mathrm{vol}}$ correlation (particle) kernel using established techniques (e.g. \cite{Joh:05,BF:08}). This works quite cleanly, since the process in question is \emph{Schur}~\cite{OR:03}, although, we have not tried this computation.

For the two-periodic weighting, the process does not appear to be Schur.  As mentioned above, under the right choice of parameters the model exhibits a third phase which has been named gaseous in which the height function correlations decay exponentially  \cite{KOS:06}.      Other models that might possess similar phenomenon are the three periodic lozenge tiling in a large hexagon and the six vertex model with domain wall boundary conditions away from the so-called free fermion line \cite{CP:10}.  Indeed, the six vertex model on the free fermion line with domain wall boundary conditions can be recovered from the one periodic Aztec diamond \cite{FS:06}.

The main motivation behind this paper was to find the correlation kernel of the two-periodic weighting of the Aztec diamond, so that one can find the local correlations of dominos at the gaseous-liquid boundary.
As mentioned above, this boundary represents the transition from the correlations of dominos having power law decay to the correlations of dominos having exponential decay.  For tiling models, the solid-liquid boundary  (with no cusps)  has been well studied (see \cite{Pet:12} for the most general results); the interlaced particle system associated with the tiling has fluctuations of size $n^{1/3}$ and the distributions of particles are normally given by the so-called \emph{Airy process}, a natural distribution originally formulated in  Random matrix theory \cite{AGZ:10}. 
As far as we are aware, the gaseous-liquid boundary of a tiling model has not been previously studied in any such probabilistic model.

The formula for the inverse Kasteleyn matrix for two-periodic weightings of the Aztec diamond, is given as a four variable generating function and it is not immediate how to analyze asymptotically.   However, Kurt Johansson, using our formula,  has been able  to  derive a double contour integral formula.  From this double contour integral formula, it should be possible to use a saddle point analysis. 
  Indeed, early computations using this approach show that the limiting octic curve can be recovered which agrees with limiting curve computed using the techniques from \cite{KO:07}.   We believe that this approach will lead to finding the correlations of the dominos at the gaseous-liquid.  This is current work in progress which will hopefully appear elsewhere \cite{CJ:13}.

\subsection{Overview of the Paper}

The paper is organized as follows: Section~\ref{section Prerequisites} we give some of the prerequisites and notation for understanding the proofs of the rest of the paper.  In Section~\ref{section Uniform case}, we compute the generating function of the inverse Kasteleyn matrix for a uniformly weighted one-periodic Aztec diamond which provides a blueprint computation.   We extend this result to biased tilings in Section~\ref{section One periodic case} as well as formulate a general boundary recurrence relation which makes computations simpler in the following sections.
In Section~\ref{section q-vol for the Aztec Diamond}, we give a contour integral formula for the inverse Kasteleyn matrix for $q^{\mathrm{vol}}$ weighting of an Aztec diamond.  Finally in Section~\ref{section Diablo Tilings on the Fortress} we derive the generating function for two-periodic weightings of the Aztec diamond for size $4m$.

\subsection*{Acknowledgements}
We would very much like to thank James Propp for the original question of computing $K^{-1}$ for diabolo tilings which prompted this research and the very useful discussions which followed,  Kurt Johansson for valuable discussions and improvements to this paper.
 We would also like to thank Alexei Borodin, C\'{e}dric Boutillier, Maurice Duits, Harald Helfgott, Richard Kenyon, Anthony Metcalfe and Andrei Okounkov for some useful insights of tiling models. We would particularly like to thank MSRI Berkeley, where this work was initiated, and the Knut and Alice Wallenberg foundation grant KAW2010.0063, which partially supported the authors' work during its completion.

\section{Notation and Background Information} \label{section Prerequisites}

Let $\mathtt{W}=\{(x_1,x_2): x_1 \mod 2=1, x_2 \mod 2=0, 1\leq x_1 \leq 2n-1, 0 \leq x_2 \leq 2n \}$ and let $\mathtt{B}=\{(x_1,x_2): x_1 \mod 2=0, x_2\mod 2=1,  0 \leq x_1 \leq 2n, 1 \leq x_2 \leq 2n-1 \}$. The  set $\mathtt{W} \cup \mathtt{B}$ denotes the vertex set of the dual graph of the Aztec diamond (rotated by $-\pi/4$ and translated), where $\mathtt{W}$ denotes the set of white vertices and $\mathtt{B}$ denotes the set of black vertices.  
We call the above coordinate system (of the dual graph) of the  Aztec Diamond, the \emph{Kasteleyn coordinates} - see Figure \ref{fig:coordinates} for an example. To avoid any confusion, we only consider dimer coverings of the dual graph of the Aztec diamond and so we refer to the Aztec diamond as the graph which has vertices given by $\mathtt{W} \cup \mathtt{B}$.  We also set $e_1=(1,1)$ and $e_2=(-1,1)$. 

\begin{figure}
 \caption{The Kasteleyn coordinates for an Aztec diamond of order 3.}
\includegraphics[height=2in]{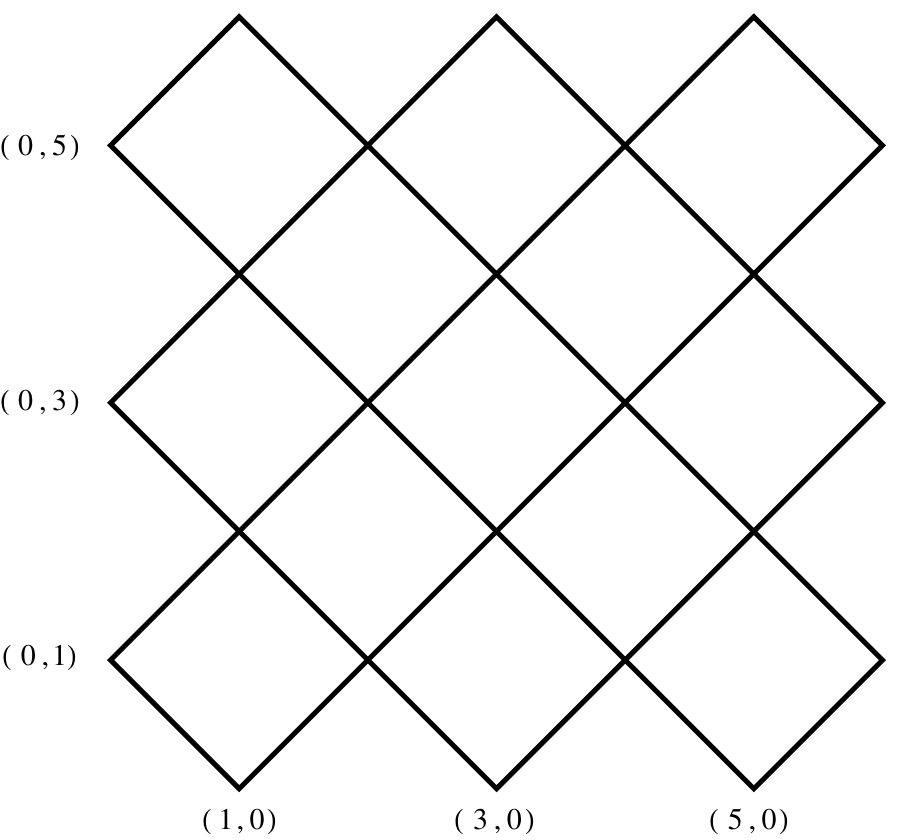}
\label{fig:coordinates}

\end{figure}

For a planar bipartite finite graph $G=(V,E)$, the \emph{Kasteleyn orientation} is a choice of edge weights so that the product of edge weights around each face is odd.  In this paper, $G$ is the graph formed from (the dual graph of) the Aztec diamond.  We consider the edge weights given by   positive numbers for edges parallel to $e_1$ and   positive numbers multiplied by $\mathtt{i}=\sqrt{-1}$ for  edges parallel to $e_2$.  As $G$ is a bipartite graph, we consider the \emph{Kasteleyn matrix}, whose rows are indexed by black vertices and columns indexed by white vertices, with entries given by
\begin{equation}
    K(\mathtt{b} , \mathtt{w}) = \left\{
    \begin{array}{ll}
    wt(\mathtt{e}) \mathtt{i}^{j-1} & \mbox{for } \mathtt{e}=(\mathtt{b},\mathtt{w}) \mbox{ and } \mathtt{b} - \mathtt{w} = e_j, j \in\{1,2\} \\
    0 & \mbox{otherwise}
    \end{array} \right.
\end{equation}
where $wt(\mathtt{e})$ is edge weight of edge $\mathtt{e}$.  As $G$ is a  bipartite finite graph,  $|\det K|$ is equal to the number of weighted dimer covers of $G$.  This was first proved by Kasteleyn~\cite{Kas:61} but in a more general setting. 
The explicit formula for the inverse Kasteleyn matrix and correlation kernel of the dominos is given as follows: suppose that $e_1=(b_1,w_1), \dots, e_n=(b_m,w_m)$ then the probability of seeing a perfect matching with the edges $e_1,\dots, e_m$ is given by \cite{Ken:97}
\begin{equation} \label{localstatistics}
      \det \left(  K(b_i,w_i)K^{-1}(w_i,b_j) \right)_{i,j=1}^m. 
\end{equation}
That is, the edges form a determinantal point process with the correlation kernel given by $L(e_i,e_j)=K(b_i,w_i) K^{-1}(w_i,b_j)$~\cite{Ken:97}.  In the case of the two-periodic weighting of the Aztec diamond, this point process is a block determinantal point process. With an explicit formula for the inverse Kasteleyn matrix, equation~\eqref{localstatistics} means that we can compute the joint probabilities of any subset of edges appearing in the matching.  For instance, we can compute edge placement probabilities when $m=1$.

We now summarize the graph theory techniques used in this paper.

Two dimer models are said to be \emph{gauge equivalent} if their partition function differs by a constant multiple.  The dimer model obtained from multiplying all the edge weights surrounding one specific vertex by the same constant is called a \emph{gauge transformation}.  As a slight abuse of terminology, ``multiplying the vertex $v$ by $a$'' means applying a gauge transformation where the weights of all the edges incident to $v$ are multiplied by $a$. Note that to keep the partition function the same under this operation, we divide the partition function of the new graph by $a$.  That is, gauge transformed dimer models are gauge equivalent.

Other than gauge transformations, we make use of three other graph transformations.  All three of these alter the graph itself, but leave the partition function invariant up to a gauge transformation.

\begin{enumerate}

\item Suppose we have a large square with edge weights $a$, $b$, $c$ and $d$ where the labelling is clockwise, contained in some graph $H$.  Suppose we deform this large square to a smaller square with edge weights $A,B,C$ and $D$ (same labelling convention as the large square) and also include an edge, with edge weight 1, between each vertex of the smaller square and its corresponding original vertex --- see Figure~\ref{fig:urbanrenewal}. We call this new graph $H^*$.
 We set $A=c/(ac+bd), B=d/(ac+bd),C=a/(ac+bd)$ and $D=b/(ac+bd)$ so that the local configurations and the weights of the matchings of $H$ are preserved under the transformation to $H^*$.  For example, a dimer covering the edge with weight $a$ and no dimer covering the edge with weight $c$ in  $H$  corresponds to dimers covering the edge with weight $C$ and the two diagonal edges incident to the edge with weight $A$ in $H^*$.
  This graphical transformation is called \emph{urban renewal}~\cite{Pro:03} and we have that 
$$ Z_{H}= (a c+b d) Z_{H^*}$$
where $Z_{H}$ and $Z_{H^*}$ are the partition functions of $H$ and $H^*$ respectively.

\begin{figure}
 \caption{The Urban renewal step which maps the large square on the left the smaller square on the right and multiplies the partition function by $a c+bd$.  The diagonal edges have weight 1} 
\includegraphics[height=1in]{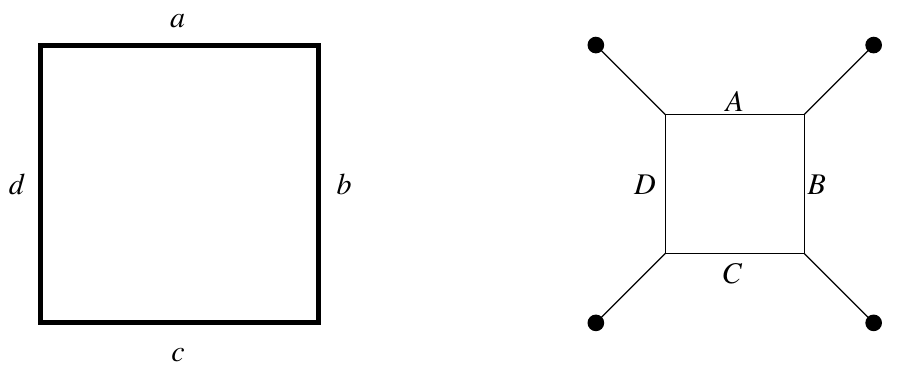}
\label{fig:urbanrenewal}

\end{figure}

\item If a vertex, $v$, is incident to two edges, each having weight 1, we can contract  the two incident edges and vertices of $v$, to $v$.  The new edge set of $v$ is the union of the edge set of the two contracted vertices omitting the two edges previously incident to $v$, all with the same edge weight before the contraction.  This procedure is called \emph{edge contraction} and has no effect on the partition function of the dimer covering.

\item If a vertex $v$ is incident to one edge $e = (v,v')$, \emph{i.e.} it is a \emph{pendant vertex}, then $v$, $v'$ and \emph{all} edges incident to $v'$ can be removed, since $e$ is forced to be present in \emph{every} perfect matching.  This procedure is called \emph{removal of pendant edges}. The partition function of the new graph formed by this procedure is equal to the original partition function divided by the weight of $e$. 
\end{enumerate}

The above three procedures can be used to compute the partition function \cite{EKLP:92}, the edge probabilities \cite{CEP:96, Pro:03} and the inverse Kasteleyn matrix for one-periodic, two-periodic and $q^{\mathrm{vol}}$ Aztec diamonds.  For general weightings of the Aztec diamond and also other stepped square lattices, although it may be theoretically possible to find the inverse Kasteleyn matrix, we were not able to solve the recurrence relations.

%\begin{figure}
%\caption{Shuffling algorithm: Applying urban }
%\includegraphics[height=3in]{shuffling1.pdf}
%\includegraphics[height=3in]{shuffling3.pdf}
%\includegraphics[height=2in]{shuffling3.pdf}
%\includegraphics[height=2in]{shuffling4.pdf}
%\label{fig:shuffling}
%\end{figure}

\section{Uniform measure case} \label{section Uniform case}

In this section, we derive the inverse Kasteleyn matrix for uniform dimer coverings as it provides a simplest example for our new method.    The inverse Kasteleyn matrix for uniform tilings of the Aztec diamond was originally computed in \cite{Hel:00}.  

The new results here are:

\begin{itemize}
\item the \emph{boundary generating function} for the Aztec diamonds, defined below, and
\item the observation that the generating function for the entries of $K^{-1}$, for all Aztec diamonds together, is a rational function in four variables: two variables marking the row of $K^{-1}$, two marking the column of $K^{-1}$.  There is also one parameter, $n$, which is the order of the Aztec diamond.
\end{itemize}

In the next section, we will give vertical edges weight $a>0$ and horizontal edges weight one; this section covers the special case $a=1$.    

The Kasteleyn matrix for the uniform Aztec diamond is given by
\begin{equation} \label{bgf:K}
      K(x,y)=\left\{ \begin{array}{ll}
              1 & \mbox{if } x-y=\pm e_1 \\ 
                \mathtt{i} &\mbox{if } x-y=\pm e_2\\
              0 &\mbox{otherwise} 
             \end{array} \right.
\end{equation}
for $x \in \mathtt{B}$ and $y \in \mathtt{W}$.  Unless stated otherwise, we shall always assume that $x=(x_1,x_2)$ and $y=(y_1,y_2)$.

Let $f_n(t)=(1-t^n)/(1-t)$ denote the sum of the geometric series $1+t+\dots+ t^{n-1}$, and let
\begin{equation}  \nonumber
	F_n(w,b)=-  \mathtt{i}/2 f_n \left( \frac{(1 + b  \mathtt{i}) (1 + w \mathtt{i} )}{2} \right).
\end{equation}

Further, for $\mathtt{w}=(w_1,w_2)$ and $\mathtt{b}=(b_1,b_2)$ set 
\begin{equation} \label{unif_bfga:h00}
	F^{0,0}_n( \mathtt{w},\mathtt{b})= F_n(w_1^2,b_2^2) w_1 b_2,
\end{equation}
\begin{equation}\label{unif_bfga:h02n}
	F^{0,1}_n(\mathtt{w},\mathtt{b})= F_n(-1/w_1^2,-b_2^2) w_1^{2n-1} b_1^{2n} b_2 \mathtt{i}, 
\end{equation}
\begin{equation}\label{unif_bfga:h2n0}
	F^{1,0}_n(\mathtt{w},\mathtt{b})= F_n(-w_1^2,-1/b_2^2) w_1 w_2^{2n} b_2^{2n-1} \mathtt{i},
\end{equation}
and
\begin{equation} \label{unif_bfga:h2n2n}
	F^{1,1}_n(\mathtt{w},\mathtt{b})= F_n(1/w_1^2,1/b_2^2) w_1^{2n-1} w_2^{2n} b_1^{2n} b_2^{2n-1}. 
\end{equation}

Let 
\begin{equation}
	G_n(\mathtt{w},\mathtt{b})=\sum_{x \in \mathtt{W}, y \in \mathtt{B}} \mathtt{w}^x \mathtt{b}^y K^{-1}(x,y)
\end{equation}
where $\mathtt{w}^x=w_1^{x_1} w_2^{x_2}$ and $\mathtt{b}^y=b_1^{y_1} b_2^{y_2}$ denote the generating function of the inverse Kasteleyn matrix of the one-periodic domino tilings of the Aztec diamond with Kasteleyn orientation given by multiplying all vertical edges (the vector $e_2$ is vertical) by $\mathtt{i}$.  

The following theorem gives the entries of $K^{-1}$ as the coefficients of a generating function.

\begin{thma} \label{thm:uniform}
\begin{equation} 
\begin{split}
&G_n(\mathtt{w},\mathtt{b})= \frac{w_1   w_2^2 b_2 f_{n+1}(w_1^2 b_1^2) f_n(w_2^2 b_2^2)}{C(w_1,w_2)} \\
&+(1+\mathtt{i} w_1)\frac{ (1+ \mathtt{i} b_2^2)F^{0,0}_n+b_1^2\left(b_2w_1f_n(b_1^2 w_1^2)+(\mathtt{i}  +b_2^2) F^{0,1}_n \right)}{C(w_1,w_2) C(b_1,b_2)}  \\
&+(\mathtt{i} +w_1^2)w_2^2 \frac{ b_1^2 b_2^{2n+1} w_1 w_2^{2n} f_n(b_1^2 w_1^2) + (1+ \mathtt{i} b_2^2)F_n^{1,0} +b_1^2(\mathtt{i} +b_2^2) F_n^{1,1}}{C(w_1,w_2) C(b_1,b_2)}
\end{split}
\end{equation}
where $C(r_1,r_2)=1+r_1^2 r_2^2+ \mathtt{i} (r_1^2+r_2^2)$ and $F^{i,j}_n=F^{i,j}_n(\mathtt{w},\mathtt{b})$ for $i,j \in\{0,1\}$.

\end{thma}

Our proof (like the proofs of the harder theorems in subsequent sections) breaks down into three steps:

\begin{itemize}
\item Computing the boundary generating function (ignoring the Kasteleyn orientation),
\item Moving the white vertices and black vertices to the boundaries,
%\item Moving the black vertices to the boundaries,
\item Computing the sign of the boundary generating function for each boundary.
\end{itemize}

\subsection{Boundary Generating function}

Our method involves first finding a generating function which will eventually be used to compute $K^{-1} (w,b)$, where $w$ and $b$ correspond to vertices on the boundary of the Aztec diamond. This generating function is called the \emph{boundary generating function}.  This will be a generating function in three variables, which encode (respectively) the position of a white boundary vertex, a black boundary vertex, and the size of the Aztec diamond.  

Consider an Aztec diamond with all edge weights equal to 1.  Let $Z_n$ denote the number of dimer coverings of an Aztec diamond of size $n$. From~\cite{EKLP:92}, we have that $Z_n=2^{n(n+1)/2}$.

We will also consider an Aztec diamond of size $n$ with the vertices $(2i+1,0)$ and $(0,2j+1)$ removed for fixed $0\leq i,j \leq n-1$. We shall call this graph $A_n(i,j)$.  Let $Z(i,j,r,n)$ denote the number of dimer coverings of $A_n(i,j)$, with edge weights all equal to  $r$ and the constraint that $Z(i,j,r,n)=0$ if $i$ or $j$ (or both) are not in $\{0,1,\dots, n-1\}$.

 It is important to notice that the Kasteleyn orientation on the original Aztec diamond, restricted to $A_n(i,j)$, remains a Kasteleyn orientation, so the corresponding entry of $K^{-1}$ for the Aztec diamond can be computed from the partition function on this graph and a relevant sign; that is, we can write
\begin{equation}
	|K^{-1}((2i+1,0),(0,2j+1))|= \frac{Z(i,j,1,n)}{Z_n}.
\end{equation}
We will compute the sign of $K^{-1}((2i+1,0),(0,2j+1))$ in Lemma~\ref{uniform:lem:sign}.

\begin{lemma} \label{lem:oneperiodicrecurrence}

\begin{equation}
\label{oneperiodic:recurrence}	\frac{Z(i,j,1,n)}{Z_n}= \frac{1}{2}\sum_{\substack{k,l \in \{0,1\} \\ (i-k,j-l) \not = (-1,-1)} } \frac{Z(i-k,j-l,1,n-1)}{Z_{n-1}} + \frac{1}{2}\mathbb{I}_{(i,j)=(0,0),n \geq1}
\end{equation}

\end{lemma}

\begin{proof}
%This proof refines the ideas in~\cite{Pro:03} which compute the edge placement probabilities for the Aztec diamond.
It was shown in~\cite{EKLP:92} that one can compute a recurrence for the partition function of an Aztec diamond using the graph transformations given in Section~\ref{section Prerequisites}.  In particular, they showed that $Z_n=2^n Z_{n-1}$. We first review this computation because showing the recurrence for $Z(i,j,1,n)/Z_n$ is an extension of this computation. 

Consider an Aztec diamond graph of size $n$ with Kasteleyn coordinates.  To the faces with centers given by the coordinates $(2i+1,2j+1)$ for all $0 \leq i,j \leq n-1$ we apply the urban renewal transformation --- see Figure~\ref{fig:add edge urban renewal} for an example but ignore the dashed edges.  Having applied urban renewal $n^2$ times, we see that $Z_n$ is equal to the partition function of the new graph multiplied by $2^{n^2}$. For this new graph all diagonal edges have weight  $1/2$ while the remaining edges have weight $1$; there are pendant vertices at $(2i+1,0)$,$(2i+1,2n)$, $(0,2j+1)$ and $(2n,2j+1)$ for all $0\leq i,j \leq n-1$ which are all  removed with their pendant edges (all edges have weight 1) with no effect to the partition function; and  there are vertices that are incident to two edges (both weight 1) at coordinates $(2i+1,2j)$ for all $0 \leq i \leq n-1$ and $1 \leq j \leq n-1$ and $(2i,2j+1)$ for all $1\leq i \leq n-1$ and $0 \leq j \leq n-1$.  The edges incident to these vertices are contracted with no effect to the partition function. The effect of the above three graph transformations on an Aztec diamond of size $n$ with edge weights 1 gives an Aztec diamond of size $n-1$ with edges weights $1/2$ and which means  
$$ Z_n = 2^{n^2} \tilde{Z}_{n-1}$$
where $\tilde{Z}_{n-1}$ denotes an Aztec diamond of size $n-1$ with all edge weights equal to $1/2$. 
By multiplying every white vertex by 2, we are applying a gauge transform   so that all edges of the Aztec diamond have edge weight 1.  As there are $n(n-1)$ white vertices of an Aztec diamond of size $n-1$, the effect of the gauge transformation is to divide the partition function by $2^{n(n-1)}$.  We conclude that $Z_n=2^n Z_{n-1}$. We now compute a recurrence for $Z(i,j,1,n)$.

Notice that $Z(i,j,1,n)$ is in fact equal to the number of matchings of an Aztec diamond with two extra vertices $v_0$ and $v_1$ added with an edge (edge weight 1) between $v_0$ and $(2i+1,0)$ and an edge (edge weight 1) between $v_1$ and $(0,2j+1)$.  We shall call $v_0$ and $v_1$ \emph{auxiliary vertices} and their incident edges will be called \emph{auxiliary edges}. 
To this graph, we will apply the same sequence of graph transformations, keeping track of the auxiliary edges.  The remainder of this proof is a careful accounting of the effects of these graph transformations on the partition function.

Explicitly, the computation is as follows. 
   We first apply urban renewal on the faces with centers $(2k+1,2l+1)$ for all $0 \leq k,l \leq n-1$, i.e. we apply urban renewal $n^2$ times, see Figure~\ref{fig:add edge urban renewal} for an example. 
\begin{figure}
\caption{For the partition function computation, ignore the dashed edges.  Otherwise, the dashed edges on the left are the two extra edges added as described in the proof of Lemma~\ref{lem:oneperiodicrecurrence} for an Aztec diamond of size 4.  The figure on the right is obtained from the figure on the left by applying urban renewal 16 times.}
 \includegraphics[height=2.5in,angle=90]{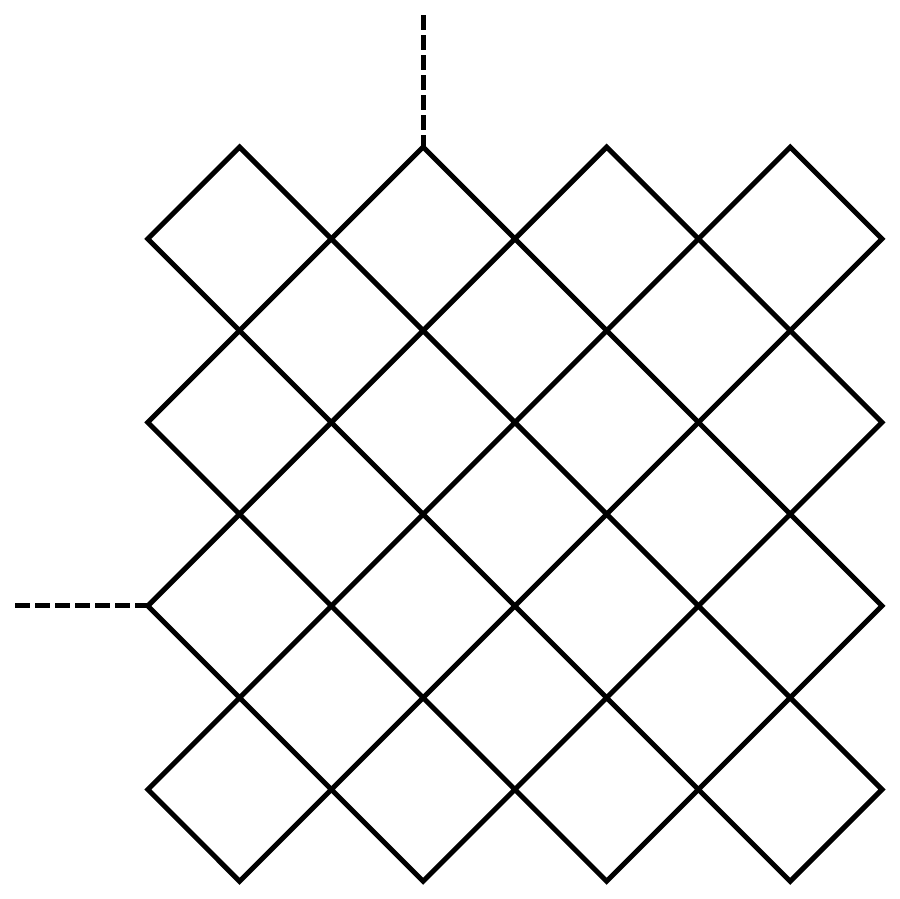}\hspace{10mm}
  \includegraphics[height=2.5in,angle=90]{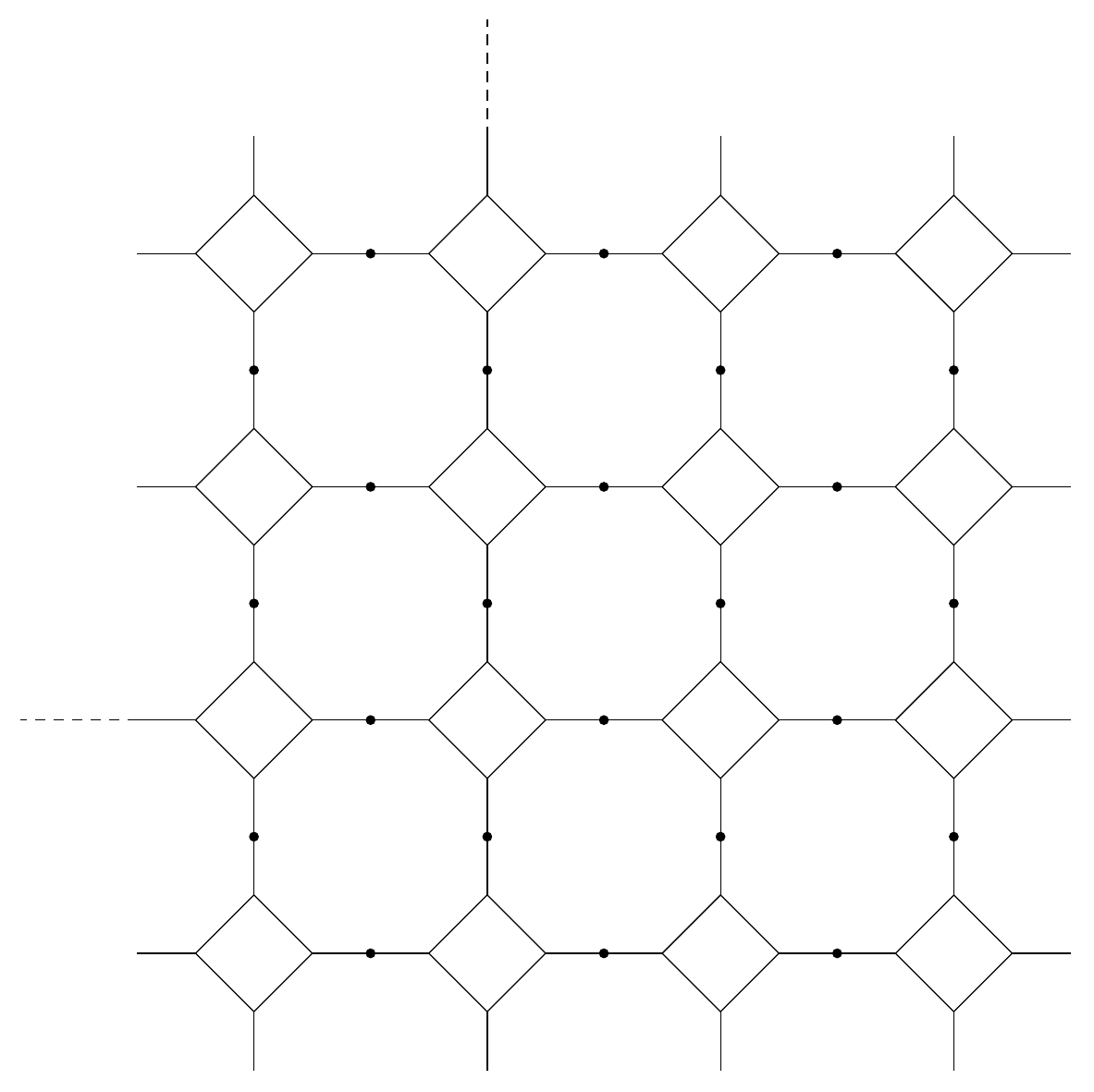} 
 %\includegraphics[height=3in]{shuffle2i.pdf}
 %\caption{The figure on the left represents}
 \label{fig:add edge urban renewal}
\end{figure}
From applying the urban renewal, every diagonal edge has weight $1/2$ and the remaining edges have weight 1. Similar to the computation of the partition function recurrence, we apply edge contractions to all the edges incident to vertices located at $(2k+1,2l)$ for $0 \leq k \leq n-1$ and  $1 \leq l \leq n-1$ and located at $(2k,2l+1)$ for $1\leq k \leq n-1$ and $0 \leq l \leq n-1$.  Some care has to be applied removing the pendant edges; we remove the pendant edges incident to the following pendant vertices:
\begin{itemize} 
\item $(2k+1,0)$ for  $0 \leq k \leq n-1$ but $k \not =i$,
\item $(0,2l+1)$ for  $0 \leq l\leq n-1$ but $l \not = j$, 
\item$(2k+1,2n)$ for  $0 \leq k \leq n-1$ and 
\item$(2n,2l+1)$ for  $0\leq l \leq n-1$.   
\end{itemize}
As the auxiliary edges are also pendant, we remove these edges too (or, equivalently, we contract them). 

The resulting shape is a modified Aztec diamond of size $n-1$: because our initial graph contained auxiliary edges, we have an additional vertex (call it  $v_2)$ and its incident edges on the bottom boundary  and an additional vertex (call it $v_3$) and its incident edges on the left-hand boundary  of the Aztec diamond of size $n-1$. Note that the edges incident to $v_2$ (and similarly to $v_3$) have weight $1/2$. 
 We shall describe the incident edges to $v_2$ and   similar descriptions hold for $v_3$ which are based on our initial choice of $v_0$ and $v_1$.  If $(v_0,(2i+1,0))$ was our initial choice of the auxiliary edge on the bottom boundary for $1 \leq i \leq n-2$,  then the edges incident to $v_2$ are  $(v_2,(2i-1,0))$ and $(v_2,(2i+1,0))$ -- see Figure~\ref{fig:after edge contraction} for an example. If $(v_0, (2n-1,0))$ was our initial choice of the auxiliary edge on the bottom boundary, then the edge incident to $v_2$ is $(v_2,(2n-3,0))$.  If $(v_0,(1,0))$ was our initial choice of the auxiliary edge  on the bottom boundary, in order to determine the edges incident $v_2$, we must also keep track of the initial choice of edges incident to $v_1$.   If we initially chose the auxiliary edges $(v_0,(1,0))$  and $(v_1,(0,2j+1))$ for $j \not=0$, then the edge incident to $v_2$  is  $(v_2,(1,0))$.  If we initially chose the auxiliary edges $(v_0,(1,0))$ and $(v_1,(0,1))$ then the edges incident to $v_2$ (and $v_3$) are $(v_2,v_3)$ and  $(v_2,(1,0))$ (and $(v_3,(0,1))$). 

For $(i,j) \not = (0,0)$, in a dimer covering the edge $v_2$ is matched to a vertex $(2k+1,0)$ where $0\leq k \leq n-2$ with $k=i$ or $k=i-1$ which has the same effect of removing the vertex $(2k+1,0)$ from an Aztec diamond of size $n-1$. A similar statement is true for $v_3$. We conclude that for $(i,j) \not =(0,0)$, the dimer covering consists of edges $(v_2,(2k+1,0))$, $(v_3,(0,2l+1))$ and an Aztec diamond of size $n-1$ with vertices $(2k+1,0)$ and $(0,2l+1)$ removed (i.e. a covering of $A_{n-1}(k,l)$) for $0\leq k \leq n-2$ with $k=i$ or $k=i-1$ and $0 \leq l \leq n-2$ with $l=j$ or $l=j-1$. 
  For $(i,j) =(0,0)$, we either match  $(v_2,(1,0))$, $(v_3,(0,1))$ and an Aztec diamond of size $n-1$ with vertices $(1,0)$ and $(0,1)$ removed (i.e. matching $A_{n-1}(0,0)$)  or we match $(v_2,v_3)$ in which case the remaining graph is an Aztec diamond of size $n-1$ with edge weights $1/2$.    The above formulation means that we have built a recurrence of $Z(i,j,1,n)$ in terms of $Z(k,l,1/2,n-1)$, namely
\begin{figure}
 \caption{The result of edge contraction from Figure~\ref{fig:add edge urban renewal}. The four dashed edges represent the result of the edge contraction on the two dashed edges in Figure~\ref{fig:add edge urban renewal}.}
 \includegraphics[height=2.5in,angle=90]{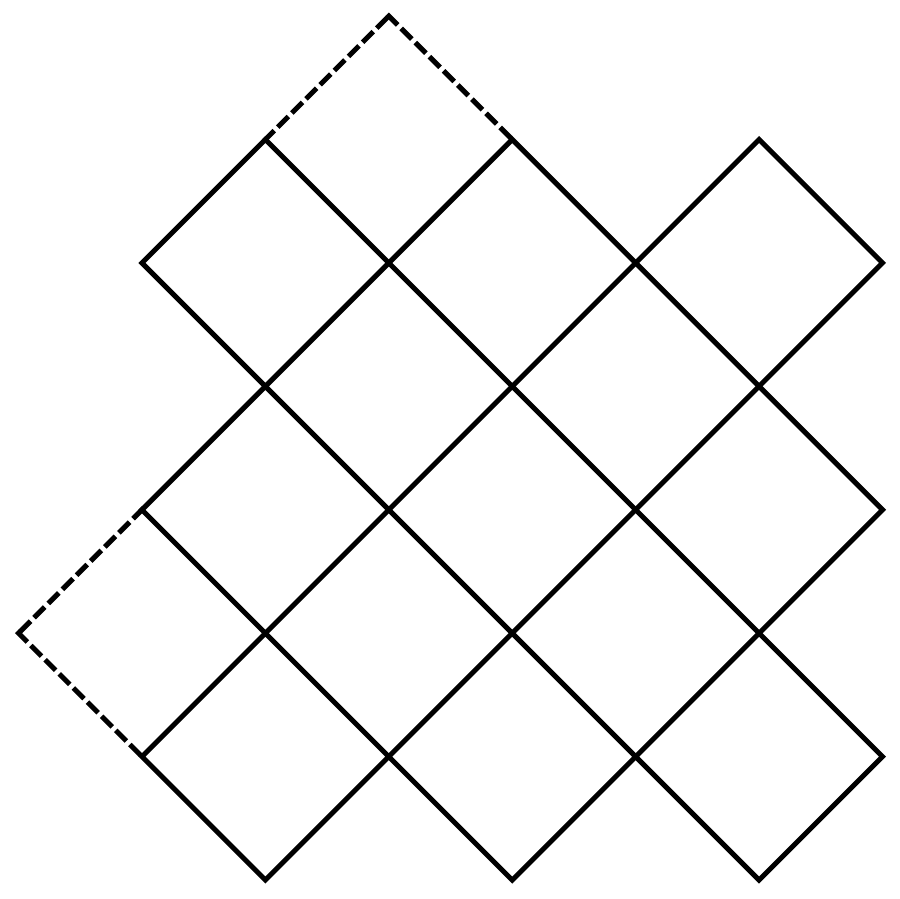}
\label{fig:after edge contraction}
 \end{figure}
\begin{equation} \label{lemproof:recurrencefirststep}
	Z(i,j,1,n)= \sum_{\substack{ k \in \{i-1,i\} \\ j \in \{j-1,j\} \\}} \frac{1}{4} Z(k,l,1/2,n-1) 2^{n^2} \mathbb{I}_{0\leq k\leq n-2} \mathbb{I}_{0\leq l\leq n-2} + \frac{1}{2}    \mathbb{I}_{(i,j)=(0,0)} \tilde{ Z}_{n-1} 2^{n^2}
\end{equation}
where the factor $1/4$ is explained by the fact that the edges incident to $v_2$ and $v_3$ have edge weight $1/2$, $\tilde{Z}_{n-1}$ is an Aztec diamond of size $n-1$ with edge weights equal to $1/2$ and the factor $2^{n^2}$ is explained by the urban renewal steps.  By a change of summation index and setting $Z(n-1,s,1/2,n-1)=Z(r,n-1,1/2,n-1)=Z(-1,s,1/2,n-1)=Z(r,-1,1/2,n-1)=0$ for all $0\leq r,s \leq n-2$, the above equation can be rewritten as 
\begin{equation}
	Z(i,j,1,n)= \sum_{\substack{k,l \in \{0,1\} \\ (i-k,j-l) \not = (-1,-1)} } \frac{1}{4} Z(i-k,j-l,1/2,n-1)2^{n^2}+\frac{1}{2}  \mathbb{I}_{(i,j)=(0,0)} \tilde{ Z}_{n-1} 2^{n^2}.
\end{equation}

We apply the gauge transformation to write $\tilde{Z}_{n-1}=2^{-n^2+n} Z_{n-1}$ which we described above in the partition function recurrence computation.   To rewrite $Z(i-k,j-l,1/2,n)$ in terms of $Z(i-k,j-l,1,n)$ for $k,l \in \{0,1\}$ and $(i-k,j-l) \not = (-1,-1)$, we  apply a gauge transformation   which multiplies all the edge weights incident to the white vertices by $2$.  As there are $n(n-1)-1$ white vertices in $Z(i-k,j-l,1/2,n-1)$ for $(i-k,j-l) \not = (-1,-1)$, we have to divide $Z(i-k,j-l,1,n-1)$ by $2^{n(n-1)-1}$.  These operations give
\begin{equation}
	Z(i,j,1,n)= \sum_{\substack{k,l \in \{0,1\} \\  (i-k,j-l) \not = (-1,-1)} } Z(i-k,j-l,1,n-1)2^{n-1}+ \mathbb{I}_{(i,j)=(0,0)} Z_{n-1} 2^{n-1}.
\end{equation}

Dividing the above equation by $Z_n$ and noting that $Z_n= 2^n Z_{n-1}$ from the partition function recurrence computation gives the lemma.

\end{proof}

\begin{defi}
The \emph{boundary generating function} is
\begin{equation}
	Z_{\partial}(w,b,1,z)=\sum_{n=0}^\infty \sum_{i=0}^{n-1} \sum_{j=0}^{n-1} \frac{Z(i,j,1,n)}{Z_n}w^i b^j z^n.
\end{equation}
\end{defi}

We now compute the boundary generating function.

\begin{lemma} \label{lem:oneperiodicbgf}
\begin{equation}
	Z_{\partial}(w,b,1,z)=\frac{ z}{(1-z)(2-z(1+b)(1+w))}
\end{equation}
\end{lemma}

\begin{proof}
We multiply~\eqref{oneperiodic:recurrence} by $w^ib^j z^n$ and sum over $0\leq i\leq n-1$, $0\leq j \leq n-1$ and $n\geq 0$  which gives
\begin{equation}
	Z_{\partial}(w,b,1,z) = \frac{1}{2}(1+b) (1+w)z Z_{\partial}(w,b,1,z)+\frac{1}{2} \frac{z}{1-z}
\end{equation}

Rearranging the above equation gives the result.

\end{proof}

\subsection{Moving the white vertices and black vertices to the boundary} \label{subsection Moving vertices}

In this section, we derive recurrences for $K^{-1}$ from each of the matrix equations $K\cdot K^{-1}=\mathbbm{I}$ and $K\cdot K^{-1}=\mathbbm{I}$.  Using these relations, we find that it is possible to write $G_n(\mathtt{w},\mathtt{b})$ as a function of $K^{-1}(w,b)$ where $w$ and $b$ are white and black vertices (respectively) on the boundary of the Aztec diamond. By this, we mean that $w$ is either $(2i+1,0)$ or $(2i+1,2n)$ for $0\leq i \leq n-1$  and $b$ is either $(0,2j+1)$ or $(2n,2j+1)$ for $0 \leq j \leq n-1$.  We define some additional generating functions: let
\begin{equation}
	G_n^0(w_1,b_1,b_2)= \sum_{\substack{(x_1,0) \in \mathtt{W} \\ y=(y_1,y_2) \in \mathtt{B}}} K^{-1}((x_1,0),y) w_1^{x_1} b_1^{y_1} b_2^{y_2},
\end{equation}
\begin{equation}
	G_n^n(w_1,b_1,b_2)= \sum_{\substack{(x_1,2n) \in \mathtt{W} \\ y=(y_1,y_2) \in \mathtt{B}}} K^{-1}((x_1,2n),y) w_1^{x_1} b_1^{y_1} b_2^{y_2},
\end{equation}
and
\begin{equation} \label{bgfa:def:Hn}
	H_n^{i,j}(\mathtt{w},\mathtt{b})= \sum_{\substack{1\leq x_1 \leq 2n-1,x_1\mod 2=1 \\ 1 \leq y_2 \leq 2n-1,y_2\mod 2=1}} K^{-1} (( x_1,2ni),(2nj, y_2)) w_1^{x_1} w_2^{2ni} b_1^{2nj} b_2^{y_2} 
\end{equation}
where $i,j \in \{0,1 \}$. From the above definition and for each $i,j \in \{0,1\}$, $H_n^{i,j}$ is a boundary generating function (with Kasteleyn orientation) where $i$ and $j$ determin which two boundaries the removed vertices lie on.

We now write $G_n(\mathtt{w},\mathtt{b})$ in terms of the boundary generating functions with the Kasteleyn orientation.
\begin{lemma} \label{uniform:lem:movevertices}
\begin{equation} 
\begin{split}
&G_n(\mathtt{w},\mathtt{b})= \frac{w_1 w_2^2 b_2  f_{n+1}(w_1^2 b_1^2) f_n(w_2^2 b_2^2)}{C(w_1,w_2)} \\
&+(1+\mathtt{i} w_1)\frac{ (1+ \mathtt{i} b_2^2)H^{0,0}_n+b_1^2\left(b_2w_1f_n(b_1^2 w_1^2)+(\mathtt{i}  +b_2^2) H^{0,1}_n \right)}{C(w_1,w_2) C(b_1,b_2)}  \\
&+(\mathtt{i} +w_1^2)w_2^2 \frac{ b_1^2 b_2^{2n+1} w_1 w_2^{2n} f_n(b_1^2 w_1^2) + (1+ \mathtt{i} b_2^2)H_n^{1,0} +b_1^2(\mathtt{i} +b_2^2) H_n^{1,1}}{C(w_1,w_2) C(b_1,b_2)}
\end{split}
\end{equation}
where $C(r_1,r_2)=1+r_1^2r_2^2+ \mathtt{i}(r_1^2+r_2^2)$ and $H_n^{i,j}=H_n^{i,j}(\mathtt{w},\mathtt{b})$.
\end{lemma}

\begin{proof}
%The Kasteleyn matrix $K$, defines the \emph{Kasteleyn operator} which maps $\mathbb{C}^\mathtt{B} \mapsto \mathbb{C}^\mathtt{W}$ with 
%\begin{equation}
%      K f(v)=\sum_{w\sim v} K(v,w)f(v)
%\end{equation}
%for $f:\mathtt{B} \mapsto \mathbb{R}$ and $v\in \mathtt{B}$.

From equation~\eqref{bgf:K}, we know that $K$ is a sparse matrix: each row has at most four nonzero entries, one entry for each neighbor of the vertex indexing the row. As such, we expand the matrix product $K\cdot K^{-1}$ in terms of the unknown matrix entries of $K^{-1}$ and compare them to the entries of $\mathbbm{I}$, this gives
\begin{equation}
	\sum_{v \sim x, v \in \mathtt{W}} K(x,v) K^{-1}(v,y)= \delta_{x=y}
\end{equation}
where $x,y \in \mathtt{B}$ and $v\sim x, v \in \mathtt{W}$ means that $v$ is a white vertex in the Aztec diamond which is nearest neighbored to $x$ (i.e. $v$ is a vertex of the form $x\pm e_1$, $x\pm e_2$ provided that these vertices are inside the Aztec diamond).

In the above equation, we evaluate the entries of $K$ using equation~\eqref{bgf:K} and so we obtain a five term relation
\begin{equation} \label{internalrecurrence1}	
K^{-1} ( x+e_1,y) \delta_{x_1 < 2n} + \mathtt{i} K^{-1} (x+e_2,y)\delta_{x_1>0} + \mathtt{i} K^{-1} (x-e_2,y) \delta_{x_1<2n} + K^{-1} (x-e_1,y) \delta_{x_1>0} = \delta_{x=y} 
\end{equation}
where $x, y \in \mathtt{B}$, $x=(x_1,x_2)$, $y=(y_1,y_2)$ and
\begin{equation}
	\delta_{x_1>0} = \left\{ \begin{array}{ll}
			1 & \mbox{if } x_1>0\\
			0 & \mbox{otherwise} \end{array}\right.	
\end{equation}	
For the left-hand side of the above equation, the delta functions account for the vertices on the boundary. For example, the black vertex $(0,2k+1)$ has two neighboring white vertices $(1,2k)$ and $(1,2k+2)$.  

Similarly, we expand the matrix product $K^{-1}\cdot K$ entry wise and compare with the identity matrix.  We obtain another five term relation which is given by
\begin{equation} \label{internalrecurrence2}
	K^{-1} (x,y+e_1) \delta_{y_2<2n} + \mathtt{i} K^{-1} (x,y+e_2) \delta_{y_2 <2n}  + K^{-1} (x,y-e_1) \delta_{y_2>0} + \mathtt{i} K^{-1}(x,y-e_2) \delta_{y_2>0}  = \delta_{x=y},
\end{equation}
where $x,y \in \mathtt{W}$.  We view equations~\eqref{internalrecurrence1} and~\eqref{internalrecurrence2} as \emph{recurrence relations} with the initial conditions coming from evaluating $K^{-1}$ at the boundary vertices.  We remark that~\eqref{internalrecurrence1} is a relation for the white vertices of $K^{-1}$ which keeps the black vertex of $K^{-1}$ fixed.  On the other hand ~\eqref{internalrecurrence2} is a relation for the black vertices of $K^{-1}$ keeping the white vertex of $K^{-1}$ fixed.

We proceed to use the above recurrence relations to find the generating function for $K^{-1}$.   We first multiply~\eqref{internalrecurrence1} by $\mathtt{w}^x \mathtt{b}^y=w_1^{x_1}w_2^{x_2} b_1^{y_1}b_2^{y_2}$ for $x=(x_1,x_2),y=(y_1,y_2) \in \mathtt{B}$ and sum both quantities over $\mathtt{B}$.  Each term on the left-hand side of~\eqref{internalrecurrence1} can be written using $G_n(\mathtt{w}, \mathtt{b})$ and either $G_n^0(w_1,b_1,b_2)$ or $G_n^n(w_1,b_1,b_2)$
 by taking a sum change of variables so that we sum over $x\in \mathtt{W}$ (as opposed to $x \in \mathtt{B}$).  We list these computations for each term on the left-hand side of~\eqref{internalrecurrence1} and then give the outcome of~\eqref{internalrecurrence1} under these computations: the first term gives
\begin{equation} \label{bgfa:thm:expand:eqn1}
\begin{split} 
	\sum_{x,y \in \mathtt{B}}  K^{-1} (x+e_1,y) \delta_{x_1<2n} \mathtt{w}^x \mathtt{b}^y &=\frac{1}{w_1 w_2} \sum_{x \in \mathtt{W}, x_2\not = 0, y \in \mathtt{B}} K^{-1}(x,y) \mathtt{w}^x \mathtt{b}^y \\
&=\frac{G_n(\mathtt{w},\mathtt{b}) - \sum_{\substack{1\leq x_1 \leq 2n\\ (x_1,0)\in\mathtt{W},y\in\mathtt{B}}} K^{-1}((x_1,0),y) w_1^{x_1} \mathtt{b}^y}{w_1 w_2}\\
&=\frac{G_n(\mathtt{w},\mathtt{b}) -G_n^0(w_1,b_1,b_2)}{w_1 w_2},
\end{split}
\end{equation}
the second term gives
\begin{equation}  \label{bgfa:thm:expand:eqn2}
\begin{split}
\mathtt{i} \sum_{x,y \in\mathtt{B}} K^{-1}(x+e_2,y) \delta_{x_1>0} \mathtt{w}^x \mathtt{b}^y& = \mathtt{i}\frac{w_1}{w_2} \sum_{x \in \mathtt{W}, x_2 \not = 0, y\in \mathtt{B}} K^{-1}(x,y) \mathtt{w}^x \mathtt{b}^y\\
&=\mathtt{i} \frac{w_1}{w_2} \left(G_n(\mathtt{w},\mathtt{b}) -  \sum_{\substack{1\leq x_1 \leq 2n\\ (x_1,0)\in\mathtt{W},y\in\mathtt{B}}} K^{-1}((x_1,0),y) w_1^{x_1} \mathtt{b}^y \right)\\
&=\mathtt{i} \frac{w_1}{w_2} \left(G_n(\mathtt{w},\mathtt{b}) - G_n^0(w_1,b_1,b_2) \right),
\end{split}
\end{equation}
the third term gives 
\begin{equation}\label{bgfa:thm:expand:eqn3}
\begin{split}
\mathtt{i} \sum_{x,y \in\mathtt{B}} K^{-1}(x-e_2,y) \delta_{x_1<2n} \mathtt{w}^x \mathtt{b}^y& = \mathtt{i}\frac{w_2}{w_1} \sum_{x \in \mathtt{W}, x_2 \not = 2n, y\in \mathtt{B}} K^{-1}(x,y) \mathtt{w}^x \mathtt{b}^y\\
&=\mathtt{i} \frac{w_2}{w_1} \left(G_n(\mathtt{w},\mathtt{b}) -w_2^{2n}  \sum_{\substack{1\leq x_1 \leq 2n\\ (x_1,2n)\in\mathtt{W},y\in\mathtt{B}}} K^{-1}((x_1,2n),y) w_1^{x_1} \mathtt{b}^y \right)\\
&=\mathtt{i} \frac{w_2}{w_1} \left(G_n(\mathtt{w},\mathtt{b}) -w_2^{2n}G_n^n(w_1,b_1,b_2) \right),
\end{split}
\end{equation}
the fourth term gives
\begin{equation}\label{bgfa:thm:expand:eqn4}
\begin{split}
\sum_{x,y \in\mathtt{B}} K^{-1}(x-e_1,y) \delta_{x_1>0} \mathtt{w}^x \mathtt{b}^y& = w_1 w_2 \sum_{x \in \mathtt{W}, x_2 \not = 2n, y\in \mathtt{B}} K^{-1}(x,y) \mathtt{w}^x \mathtt{b}^y\\
&=w_1 w_2 \left(G_n(\mathtt{w},\mathtt{b}) -w_2^{2n}  \sum_{\substack{1\leq x_1 \leq 2n\\ (x_1,2n)\in\mathtt{W},y\in\mathtt{B}}} K^{-1}((x_1,2n),y) w_1^{x_1} \mathtt{b}^y \right)\\
&=w_1 w_2 \left(G_n(\mathtt{w},\mathtt{b}) -w_2^{2n}G_n^n(w_1,b_1,b_2)\right),
\end{split}
\end{equation}
and using~\eqref{bgfa:thm:expand:eqn1},~\eqref{bgfa:thm:expand:eqn2},~\eqref{bgfa:thm:expand:eqn3} and~\eqref{bgfa:thm:expand:eqn4},~\eqref{internalrecurrence1} becomes
\begin{equation}
\begin{split}  \label{internalrecurrencefirstincomplete}
&\left( \frac{1}{w_1 w_2} + \frac{w_1}{w_2} \mathtt{i} + \frac{w_2}{w_1} \mathtt{i} + w_1 w_2 \right) G_n(\mathtt{w},\mathtt{b}) -\left(\frac{1}{w_1 w_2} + \frac{w_1}{w_2} \mathtt{i} \right) G_n^0(w_1,b_1,b_2) \\
&-\left(w_1 w_2 + \frac{w_2}{w_1} \mathtt{i}  \right) G_n^n (w_1,b_1,b_2) w_2^{2n} = \sum_{\substack{x \in \mathtt{B} \\ y \in \mathtt{B}}} \delta_{x=y} \mathtt{w}^x \mathtt{b}^y.
\end{split}
\end{equation}
We evaluate the right-hand side of the above equation by computing explicitly the sum, namely
\begin{equation}
 \sum_{\substack{x \in \mathtt{B} \\ y \in \mathtt{B}}} \delta_{x=y} \mathtt{w}^x \mathtt{b}^y=\sum_{\substack{0\leq x_1 \leq 2n,x_1 \mod 2=0\\ 1 \leq x_2 \leq 2n-1, x_2 \mod 2=1}} w_1^{x_1} w_2^{x_2} b_1^{x_1} b_2^{x_2} = \left( \frac{1-(w_1 b_1)^{2n+2}}{1-w_1^2 b_1^2} \right) \left( \frac{ w_2 b_2 (1-(w_2 b_2)^{2n})}{1-w_2^2 b_2^2} \right)
\end{equation}
We use the above equation,  the definitions of $f_n$ and $C(w_1,w_2)$ to rewrite~\eqref{internalrecurrencefirstincomplete}.  After multiplying both sides by $w_1 w_2$ and rearranging,~\eqref{internalrecurrencefirstincomplete} becomes
\begin{equation}
  \label{internalrecurrencefirst}
	\begin{split}
C(w_1,w_2) G_n(\mathtt{w},\mathtt{b})& = w_1 w_2^2 b_2 f_{n+1}(w_1^2 b_1^2)f_{n}(w_2^2 b_2^2)+(1+w_1^2 \mathtt{i}) G_n^0(w_1,b_1,b_2)\\
&+(w_1^2+\mathtt{i})w_2^2 G_n^n(w_1,b_1,b_2) w_2^{2n} 
\end{split}
\end{equation}

We now find expressions for $G_n^0(w_1,b_1,b_2)$ and $G_n^n(w_1,b_1,b_2)$ in terms of $H_n^{i,j}(\mathtt{w},\mathtt{b})$ for $i,j \in \{0,1\}$.  
We first compute $G_n^0(w_1,b_1,b_2)$ using the recurrence given in~\eqref{internalrecurrence2}: we set $x=(x_1,0)$ fixed and  we mulitply~\eqref{internalrecurrence2} by $\mathtt{b}^y$ with $y \in \mathtt{W}$ and we sum over all white vertices $y \in \mathtt{W}$ and the white vertices $x=(x_1,0) \in \mathtt{W}$.  Each term on the left-hand side of~\eqref{internalrecurrence2} can  be written in terms of $G_n^0(w_1,b_1,b_2)$ and either  $H_n^{0,0}(\mathtt{w},\mathtt{b})$ or $H_n^{0,1}(\mathtt{w},\mathtt{b})$ by taking a sum change of variables.  We list these computations for the first four terms of~\eqref{internalrecurrence2} and then give~\eqref{internalrecurrence2} under this computation: the first term is
\begin{equation}
\begin{split}
\sum_{\substack{ (x_1,0) \in \mathtt{W} \\ y \in \mathtt{W}}} K^{-1} ((x_1,0),y+e_1) \delta_{y_2 <2n} w_1^{x_1} \mathtt{b}^y &= \frac{1}{b_1b_2} \sum_{\substack{ (x_1,0) \in \mathtt{W} \\ y_1 \not =0, y \in \mathtt{B}}} K^{-1}((x_1,0),(y_1,y_2))w_1^{x_1} b_1^{y_1} b_2^{y_2} \\
&=\frac{1}{b_1 b_2} \left( G_n^0(w_1,b_1,b_2) - H_n^{0,0}(\mathtt{w},\mathtt{b})\right),
\end{split}
\end{equation}
the second term is
\begin{equation}
\begin{split}
\mathtt{i} \sum_{\substack{ (x_1,0) \in \mathtt{W} \\ y \in \mathtt{W}}} K^{-1} ((x_1,0),y+e_2) \delta_{y_2 <2n} w_1^{x_1} \mathtt{b}^y &= \mathtt{i} \frac{b_1}{b_2} \sum_{\substack{ (x_1,0) \in \mathtt{W} \\ y_1 \not =2n, y \in \mathtt{B}}} K^{-1}((x_1,0),(y_1,y_2))w_1^{x_1} b_1^{y_1} b_2^{y_2} \\
&=\mathtt{i}\frac{b_1}{ b_2} \left( G_n^0(w_1,b_1,b_2) - H_n^{0,1}(\mathtt{w},\mathtt{b}) \right),
\end{split}
\end{equation}
the third term is
\begin{equation}
\begin{split}
 \sum_{\substack{ (x_1,0) \in \mathtt{W} \\ y \in \mathtt{W}}} K^{-1} ((x_1,0),y-e_1) \delta_{y_2 >0} w_1^{x_1} \mathtt{b}^y &=  b_1 b_2 \sum_{\substack{ (x_1,0) \in \mathtt{W} \\ y_1 \not =2n, y \in \mathtt{B}}} K^{-1}((x_1,0),(y_1,y_2))w_1^{x_1} b_1^{y_1} b_2^{y_2} \\
&= b_1 b_2 \left( G_n^0(w_1,b_1,b_2) - H_n^{0,1}(\mathtt{w},\mathtt{b})\right),
\end{split}
\end{equation}
the fourth term is
\begin{equation}
\begin{split}
\mathtt{i}\sum_{\substack{ (x_1,0) \in \mathtt{W} \\ y \in \mathtt{W}}} K^{-1} ((x_1,0),y-e_2) \delta_{y_2 >0} w_1^{x_1} \mathtt{b}^y &=\mathtt{i} \frac{b_2}{b_1} \sum_{\substack{ (x_1,0) \in \mathtt{W} \\ y_1 \not =0, y \in \mathtt{B}}} K^{-1}((x_1,0),(y_1,y_2))w_1^{x_1} b_1^{y_1} b_2^{y_2} \\
&=\mathtt{i}\frac{b_2}{b_1} \left( G_n^0(w_1,b_1,b_2) - H_n^{0,0}(\mathtt{w},\mathtt{b})\right),
\end{split}
\end{equation}
and~\eqref{internalrecurrence2} becomes
\begin{equation} \label{internalrecurrence3pre}
	\begin{split}
	&\left(\frac{1}{b_1b_2} + \frac{b_1}{b_2} \mathtt{i} + b_1 b_2 + \frac{b_2}{b_1} \mathtt{i} \right)G_n^0 (w_1,b_1,b_2) - \left(\frac{1}{b_1b_2} + \frac{b_2}{b_1} \mathtt{i} \right)H^{0,0}_n (\mathtt{w},\mathtt{b}) \\
&-\left( \frac{b_1}{b_2} \mathtt{i} +b_1b_2 \right) H^{0,1}_n(\mathtt{w},\mathtt{b}) = \frac{w_1 b_1( 1- (w_1 b_1)^{2n})}{1-(w_1 b_1)^2}
\end{split}
\end{equation}
where the right-hand side of the above equation follows from
\begin{equation}
\sum_{(x_1,0) \in \mathtt{W}, (y_1,y_2) \in \mathtt{W}} \delta_{(x_1,0)=(y_1,y_2)} w_1^{x_1} b_1^{y_1}b_2^{y_2}=\frac{ w_1 b_1(1-w_1^{2n} b_1^{2n})}{1-w_1^2 b_1^2}
\end{equation}
By multiplying~\eqref{internalrecurrence3pre}  by $b_1 b_2$ and using the definition of $f_n$ and $C(b_1,b_2)$, we obtain
\begin{equation} \label{internalrecurrence3}
\begin{split}
C(b_1,b_2) G_n^0(w_1,b_1,b_2) &= w_1 b_1 f_n(w_1^2 b_1^2) b_1b_2+(1+b_2^2 \mathtt{i}) H_n^{0,0} (\mathtt{w},\mathtt{b}) \\
&+ b_1^2(b_2^2+\mathtt{i} ) H_n^{0,1}(\mathtt{w},\mathtt{b}).
\end{split}
\end{equation}

To find an expression for $G_n^n(w_1,b_1,b_2)$, we mirror the computation used for finding an expression for $G_n^0(w_1,b_1,b_2)$.  From doing this computation, we obtain
\begin{equation}
\begin{split}
\label{internalrecurrence4}
 C(b_1,b_2) G_n^n(w_1,b_1,b_2)w_2^{2n} &= w_2^{2n} b_2^{2n} w_1 b_1 b_1 b_2 f_n(w_1^2 b_1^2)+(1+b_2^2 \mathtt{i} ) H_n^{1,0}(\mathtt{w},\mathtt{b})\\
& + b_1^2 (b_2^2+\mathtt{i}) H_{n}^{1,1} (\mathtt{w},\mathtt{b}).
\end{split} 
\end{equation}

We now substitute the expressions for $G_n^0(w_1,b_1,b_2)$ and $G_n^{2n}(w_1,b_1,b_2)w_2^{2n}$ from~\eqref{internalrecurrence3} and~\eqref{internalrecurrence4} into~\eqref{internalrecurrencefirst}, which gives the formula given in the lemma.

\end{proof}

\subsection{Computing the sign of the boundary generating function for each boundary} \label{subsection Computing sign}

The determinant of a Kasteleyn matrix computes the partition function only up to an overall sign.  It is important to compute this sign whenever two such determinants appear in the same formula (as they do throughout our work).  To do this, it suffices to compute the contribution of any one perfect matching to $\det K$.

In this subsection, we complete the proof of Theorem~\ref{thm:uniform} by computing the sign of the boundary generating function for each boundary.  We prove the following lemma

\begin{lemma} \label{uniform:lem:sign}
For $i,j \in \{0,1\}$, we have
\begin{equation}
H_n^{i,j}(\mathtt{w},\mathtt{b})=F_n^{i,j}(\mathtt{w},\mathtt{b})
\end{equation}	
where $H_n^{i,j}(\mathtt{w},\mathtt{b})$ is defined in~\eqref{bgfa:def:Hn} and $F_n^{0,0}(\mathtt{w},\mathtt{b}),F_n^{0,1}(\mathtt{w},\mathtt{b}), F_n^{1,0}(\mathtt{w},\mathtt{b})$ and $F_n^{1,1}(\mathtt{w},\mathtt{b})$ are given in equations~\eqref{unif_bfga:h00},~\eqref{unif_bfga:h02n},\eqref{unif_bfga:h2n0} and~\eqref{unif_bfga:h2n2n} respectively.
\end{lemma}

\begin{proof}

We first remark that $H_n^{i,j}(\mathtt{w},\mathtt{b})$ denotes a boundary generating functions of an Aztec diamond of size $n$ with the Kasteleyn orientation for each $i,j\in \{0,1\}$.
In Lemma~\ref{lem:oneperiodicbgf}, we computed the boundary generating function $Z_{\partial}(w,b,1,z)$ where the parameter $z$ marked the size of the Aztec diamond.  To find the boundary generating function of an Aztec diamond of size $n$, we extract out the $n^{th}$ coefficient of $z$ in $Z_{\partial}(w,b,1,z)$. This is given by
\begin{equation} \label{bgfa:thmproof:signs}
	\frac{1}{2} \frac{ 1-\left(\frac{(1+b)(1+w)}{2} \right)^n}{ 1-\left(\frac{(1+b)(1+w)}{2} \right)}
\end{equation}
Recall that we have $|K^{-1}((2i+1,0),(0,2j+1))|= Z (i,j,1,n)/Z_n$ for $0\leq i,j\leq n-1$.  This follows from the fact that the Kasteleyn orientation remains the same if we remove two vertices on the same face and so each signed count of perfect matchings on the graph with the two removed vertices has the same sign.  We first compute the sign of $K^{-1}((1,0),(0,1))$.   Consider the Aztec diamond with only vertical dimers (parallel to $e_2$). This has weight $\mathtt{i}^{n(n+1)}$.  After removing the edge $((1,0),(1,0))$, the product of edge weights is given by $\mathtt{i}^{n(n+1)-1}$.  This means that $K^{-1}((1,0),(0,1))$ has sign $-\mathtt{i}$.  We proceed in computing $K^{-1}((2i+1,0),(0,2j+1))$ from $K^{-1}((2i+1,0),(0,2j-1))$ by induction: we add a vertex at $(-1,2j)$ with edges $((-1,2j),(0,2j-1))$ and $((-1,2j),(0,2j+1))$  with the edges given the Kasteleyn orientation (i.e. edges parallel to $e_2$ have weight $\mathtt{i})$.  It follows that for $K^{-1}((2i+1,0),(0,2j-1))$ we have the edge $((-1,2j),(0,2j-1))$ matched.  To obtain $K^{-1}((1,0),(0,2j+1))$ we rotate the edges around the face $(0,2j)$ which contributes a factor of $(-1) \mathtt{i}^{-1}$ where the $(-1)$ is from applying the rotation and $\mathtt{i}^{-1}$ is from having one less vertical edge.  Therefore, we have that the sign of $K^{-1}((2i+1,0),(0,2j+1))$ is given by the sign of $K^{-1}((2i+1,0),(0,2j-1))$ multiplied by $\mathtt{i}$.  By reflection we can apply the same argument for computing the sign of $K^{-1}((2i+1,0),(0,2j+1))$ from $K^{-1}((2i-1,0),(0,2j+1))$ but we now add a vertex at $(2i,-1)$ instead.  It follows from the above argument that the sign of $K^{-1}((2i+1,0),(0,2j+1))$ is equal to $-\mathtt{i}^{i+j+1}$.
 We can substitute this sign back into~\eqref{bgfa:thmproof:signs} noting that the coefficient of $w^{i} b^j$ in~\eqref{bgfa:thmproof:signs} corresponds to removing the vertices $(2i+1,0)$ and $(0,2j+1)$.  Under this operation,~\eqref{bgfa:thmproof:signs} is exactly $F_n(w,b)$.  We conclude that $F_n(w,b)$ is the boundary generating function of the uniform Aztec diamond of size $n$ with the Kasteleyn orientation, that is
\begin{equation}
	F_n(w,b)=\sum_{\substack{0 \leq i \leq n-1 \\ 0 \leq j \leq n-1}} K^{-1}((2i+1,0),(0,2j+1)) w^i b^j
\end{equation}
We compare the above equation with $H^{0,0}_n(\mathtt{w},\mathtt{b})$ and we find
\begin{equation} 
	H^{0,0}_n(\mathtt{w},\mathtt{b})= F_n(w_1^2,b_2^2) w_1 b_2,
\end{equation}
because the coefficient of  $w^i b^j$ in $F_n(w,b)$ is $K^{-1}((2i+1,0),(0,2j+1))$.  From~\eqref{bfga:h00}, the right-hand side of the above equation is equal to $F^{0,0}_n(\mathtt{w},\mathtt{b})$ and so we have shown $H^{0,0}_n(\mathtt{w},\mathtt{b})=F^{0,0}_n(\mathtt{w},\mathtt{b})$.
We now have to consider the boundary generating functions on the different boundaries to compute the remaining terms $H^{0,1}_n(\mathtt{w},\mathtt{b})$, $H^{1,0}_n(\mathtt{w},\mathtt{b})$ and $H^{1,1}_n(\mathtt{w},\mathtt{b})$.

To find $H^{0,1}_n(\mathtt{w},\mathtt{b})$, we first consider the boundary generating function (with Kasteleyn orientation) where the horizontal edges have weight $\mathtt{i}$ and vertical edges have weight $1$. By the symmetry of the Aztec diamond, we use the boundary generating function (with no orientation) for a uniform weighted  Aztec diamond of size $n$ given in~\eqref{bgfa:thmproof:signs} and compute the sign from the orientation using the argument given above for attaching a sign to the boundary generating function. For this particular orientation, we find that the coefficient of $w^i b^j$ in~\eqref{bgfa:thmproof:signs} is  given by $\mathtt{i}^{3i+3j}$.  This means that the boundary generating function (with Kasteleyn orientation) with horizontal edges given weight $\mathtt{i}$ and vertical edges with weight $1$ is given by $\mathtt{i} F_n(-w,-b)$. 
 The coefficient of $w^i b^j$ in $\mathtt{i}F_n(-w,-b)$ is exactly equal to $K^{-1}((2n-1-2i,0),(2n,2j+1))$ and so we have
\begin{equation}
 \mathtt{i} F_n(-w,-b)=\sum_{\substack{0\leq i \leq n-1 \\ 0 \leq j \leq n-1}} K^{-1} ((2i+1,0),(2n,2j+1)) w^{n-1-i} b^j.
\end{equation}
We compare the above equation with $H_n^{0,1}(\mathtt{w},\mathtt{b})$ and we find
\begin{equation}
	H^{0,1}_n(\mathtt{w},\mathtt{b})= F_n(-1/w_1^2,-b_2^2) w_1^{2n-1} b_1^{2n} b_2 \mathtt{i}, 
\end{equation}
and from~\eqref{bfga:h02n}, the right-hand side of the above equation is equal to $F^{0,1}_n(\mathtt{w},\mathtt{b})$. We conclude that $H^{0,1}_n (\mathtt{w}, \mathtt{b})= F^{0,1}_n(\mathtt{w},\mathtt{b})$. 

To find $H^{1,0}_n(\mathtt{w},\mathtt{b})$, we first consider the boundary generating function (with Kasteleyn orientation) where the horizontal edges have weight $\mathtt{i}$ and vertical edges have weight $1$ which we computed in the previous paragraph and is given by $\mathtt{i} F_n(-w,-b)$.  The coefficient of $w^i b^j$ in $\mathtt{i}F_n(-w,-b)$ is exactly equal to $K^{-1}((2i+1,2n),(0,2n-2j-1))$ and so we have
\begin{equation}
\mathtt{i}F_n(-w,-b)= \sum_{\substack{0 \leq i \leq n-1 \\ 0 \leq j \leq n-1}} K^{-1}((2i+1,2n),(0,2j+1)) w^i b^{n-1-j} 
\end{equation}
We compare the above equation with $H_n^{1,0}(\mathtt{w},\mathtt{b})$ and we find
\begin{equation}
	H^{1,0}_n(\mathtt{w},\mathtt{b})= F_n(-w_1^2,-1/b_2^2) w_1 w_2^{2n} b_2^{2n-1} \mathtt{i}.
\end{equation}
From~\eqref{bfga:h2n0}, the right-hand side of the above equation is equal to $F^{1,0}_n(\mathtt{w},\mathtt{b})$.  We  conclude that $H^{1,0}_n(\mathtt{w},\mathtt{b}) =F^{1,0}_n(\mathtt{w},\mathtt{b})$. 
 
To find $H^{1,1}_n(\mathtt{w},\mathtt{b})$, we  use the above argument for computing the sign when the horizontal edges have weight $1$ and the vertical edges have weight $\mathtt{i}$.
By following the computation given for $H^{0,0}_n(\mathtt{w},\mathtt{b})$,  the sign of $K^{-1}((2n-2i-1,2n),(2n,2n-2j-1))$ is given by $-\mathtt{i}^{i+j+1}$.
By the symmetry of the Aztec diamond, it follows that $K^{-1}((2i+1,2n),(2n,2j+1))$ is equal to the coefficient of $w^{n-1-i} b^{n-1-j}$ in $F_n(w,b)$ which means that
\begin{equation}
	F_n(w,b) = \sum_{ \substack{0 \leq i \leq n-1 \\ 0\leq j \leq n-1}} K^{-1}((2i+1,2n),(2n,2j+1)) w^{n-1-i} b^{n-1-j}.
\end{equation}
We compare the above equation with $H_n^{1,1}(\mathtt{w},\mathtt{b})$ and we find
\begin{equation} \label{bfg:h2n2n}
	H^{1,1}_n(\mathtt{w},\mathtt{b})= F_n(1/w_1^2,1/b_2^2) w_1^{2n-1} w_2^{2n} b_1^{2n} b_2^{2n-1}. 
\end{equation}
From~\eqref{bfga:h2n2n}, the right-hand side of the above equation is equal to $F^{1,1}_n(\mathtt{w},\mathtt{b})$. We conclude that $H^{1,1}_n(\mathtt{w},\mathtt{b})= F^{1,1}_n(\mathtt{w},\mathtt{b})$. 

\end{proof}

The computation for the sign of $K^{-1}((2i+1,0),(0,2j+1))$ (and other boundary values of the inverse Kasteleyn matrix) is independent of the edge weights and only depends on which  Kasteleyn orientation we choose. As all of the Aztec diamonds considered in this paper have the same Kasteleyn orientation,  from the proof of Lemma~\ref{uniform:lem:sign} we find
\begin{equation}
\begin{split}
&\mathrm{sgn} (K^{-1}((2i+1,0),(0,2j+1)))=\mathrm{sgn}(K^{-1}((2n-2i-1,2n),(2n,2n-2j-1))=-\mathtt{i}^{i+j+1}\\
&\mathrm{sgn} (K^{-1}((2n-2i-1,0),(2n,2j+1)))=\mathrm{sgn}(K^{-1}((2i+1,2n),(0,2n-2j-1))=\mathtt{i}^{3i+3j}
\end{split}
\end{equation}

We now prove Theorem~\ref{thm:uniform}.

\begin{proof}[Proof of Theorem~\ref{thm:uniform}]

In Lemma~\ref{uniform:lem:sign}, we found the expressions for $H^{i,j}_n(\mathtt{w},\mathtt{b})$ for all $i,j \in \{0,1\}$.  We substitute these expressions for $H^{i,j}_n(\mathtt{w},\mathtt{b})$ into the formula given in Lemma~\ref{uniform:lem:movevertices} which gives the generating function for $K^{-1}$.  

\end{proof}

\section{Biased One-Periodic Case} \label{section One periodic case}

The inverse Kasteleyn matrix for biased, one-periodic tilings was first computed in \cite{CJY:12}.  The formula was guessed based on the one-to-two lifting from the interlaced particle system \cite{Joh:05}.   We are now able to do this in a much more systematic way, following the framework laid out in Section~\ref{section Uniform case}.  As before, we compute the inverse Kasteleyn matrix by computing a multivariate generating function for its entries.  The formula strictly generalizes that of Theorem~\ref{thm:uniform}, adding one new parameter: $a$, the bias.

The Kasteleyn matrix for the one-periodic Aztec diamond is given by
\begin{equation} \label{bgfa:K}
      K_a(x,y)=\left\{ \begin{array}{ll}
              1 & \mbox{if } x-y=\pm e_1 \\ 
              a  \mathtt{i} &\mbox{if } x-y=\pm e_2\\
              0 &\mbox{otherwise} 
             \end{array} \right.
\end{equation}
for $x \in \mathtt{B}$ and $y \in \mathtt{W}$.  Unless stated otherwise, we shall always assume that $x=(x_1,x_2)$ and $y=(y_1,y_2)$

Let $f_n(t)=(1-t^n)/(1-t)$ denote the sum of the geometric series $1+t+\cdots+t^{n-1}$.  Let
\begin{equation}  \nonumber
	F_n(w,b,a)=-  \mathtt{i}/(1 + a^2) a  f_n \left( \frac{(1 + b a  \mathtt{i}) (1 + w a \mathtt{i} )}{1 + a^2} \right).
\end{equation}
Further, for $\mathtt{w}=(w_1,w_2)$ and $\mathtt{b}=(b_1,b_2)$ set 
\begin{equation} \label{bfga:h00}
	F^{0,0}_n( \mathtt{w},\mathtt{b},a)= F_n(w_1^2,b_2^2,a) w_1 b_2,
\end{equation}
\begin{equation}\label{bfga:h02n}
	F^{0,1}_n(\mathtt{w},\mathtt{b},a)= F_n(-1/w_1^2,-b_2^2,a^{-1}) w_1^{2n-1} b_1^{2n} b_2 \mathtt{i}/a, 
\end{equation}
\begin{equation}\label{bfga:h2n0}
	F^{1,0}_n(\mathtt{w},\mathtt{b},a)= F_n(-w_1^2,-1/b_2^2,a^{-1}) w_1 w_2^{2n} b_2^{2n-1} \mathtt{i}/a
\end{equation}
and
\begin{equation} \label{bfga:h2n2n}
	F^{1,1}_n(\mathtt{w},\mathtt{b},a)= F_n(1/w_1^2,1/b_2^2,a) w_1^{2n-1} w_2^{2n} b_1^{2n} b_2^{2n-1}. 
\end{equation}

Let 
\begin{equation}
	G_n(\mathtt{w},\mathtt{b})=\sum_{x \in \mathtt{W}, y \in \mathtt{B}} \mathtt{w}^x \mathtt{b}^y K_a^{-1}(x,y)
\end{equation}
where $\mathtt{w}^x=w_1^{x_1} w_2^{x_2}$ and $\mathtt{b}^y=b_1^{y_1} b_2^{y_2}$ denote the generating function of the inverse Kasteleyn matrix of the one-periodic weighting of the Aztec diamond with Kasteleyn orientation given by multiplying all vertical edges (the vector $e_2$ is vertical) by $\mathtt{i}$.

\begin{thma} \label{thm:oneperiodic}
\begin{equation} 
\begin{split}
&G_n(\mathtt{w},\mathtt{b})= \frac{w_1 w_2^2 b_2  f_{n+1}(w_1^2 b_1^2) f_n(w_2^2 b_2^2)}{C(w_1,w_2)} \\
&+(1+\mathtt{i} a w_1)\frac{ (1+ \mathtt{i}a b_2^2)F^{0,0}_n+b_1^2\left(b_2w_1f_n(b_1^2 w_1^2)+(\mathtt{i} a +b_2^2) F^{0,1}_n \right)}{C(w_1,w_2) C(b_1,b_2)}  \\
&+(\mathtt{i} a +w_1^2)w_2^2 \frac{ b_1^2 b_2^{2n+1} w_1 w_2^{2n} f_n(b_1^2 w_1^2) + (1+ \mathtt{i} a b_2^2)F_n^{1,0} +b_1^2(\mathtt{i} a+b_2^2) F_n^{1,1}}{C(w_1,w_2) C(b_1,b_2)}
\end{split}
\end{equation}
where $C(r_1,r_2)=1+r_1^2 r_2^2+ \mathtt{i} a(r_1^2+r_2^2)$ and $F^{i,j}_n=F^{i,j}_n(\mathtt{w},\mathtt{b},a)$ for $i,j\in \{0,1\}$.
\end{thma}

The proof of Theorem~\ref{thm:oneperiodic} is given in the next but one subsection.  In the next subsection, we introduce a general boundary recurrence relation which is a generalization of the methods used in Lemma~\ref{lem:oneperiodicrecurrence}.  This general boundary recurrence is used in the proof of Theorem~\ref{thm:oneperiodic} and the subsequent proofs of the other weightings.

\subsection{General boundary recurrence} \label{section:general}
In this subsection, we introduce a  general boundary recurrence relation.  Essentially, this is built by a  generalization of part of the proof of Lemma~\ref{lem:oneperiodicrecurrence} by using an Aztec diamond graph with arbitrary edge weights.
Although the notation for this subsection is self-contained, we will refer to it later in the paper. 

Let $w_{0,0}(k,l)$,$w_{0,1}(k,l)$,$w_{1,0}(k,l)$ and $w_{1,1}(k,l)$ be the edge weights of the four edges surrounding the face whose center has coordinates $(2k+1,2l+1)$ for $ 0 \leq k,l \leq n-1$ for an Aztec diamond of size $n$.  That is, the weights of the edges $((2k,2l+1),(2k+1,2l+2))$, $((2k+1,2l+2),(2k+2,2l+1))$,$((2k,2l+1),(2k+1,2l))$ and $((2k+1,2l),(2k+2,2l+1))$ are given by $w_{0,0}(k,l)$,$w_{0,1}(k,l)$,$w_{1,0}(k,l)$ and $w_{1,1}(k,l)$ respectively for $0 \leq k,l \leq n-1$. We denote the urban renewal factor of the face whose center has coordinates $(2k+1,2l+1)$ by $\Delta(k,l)$ for $0 \leq k,l \leq n-1$. Explicitly, we have
$$\Delta(k,l)=w_{0,0}(k,l) w_{1,1}(k,l)+w_{0,1}(k,l) w_{1,0}(k,l)$$
for $0\leq k,l \leq n-1$.
Let $\mathcal{Z}_n$ be the number of weighted dimer covers of the Aztec diamond of size $n$ (with the above weighting) and let $\mathcal{Z}_n(i,j)$ denote the weighted number of  dimer coverings of the Aztec diamond of size $n$ (with the above weighting) with the vertices $(2i+1,0)$ and $(0,2j+1)$ removed from the graph, for fixed $0 \leq i,j \leq n-1$.
 
To the above Aztec diamond, we apply urban renewal to the faces with centers $(2k+1,2l+1)$ for all  $0 \leq k,l \leq n-1$ (i.e. apply urban renewal $n^2$ times), removal of all pendant edges and edge contraction of each edge which is incident to a vertex incident to exactly two edges which gives an Aztec diamond of size $n-1$.  This deformation is detailed in the proof of Lemma~\ref{lem:oneperiodicrecurrence} --- see computing the partition function recurrence. For this new graph, the edge weights around the face whose center has coordinates $(2k+1,2l+1) $ for $0 \leq k,l \leq n-2$ are given by $w_{0,0}(k,l+1)/\Delta(k,l+1)$, $w_{0,1}(k+1,l+1)/\Delta(k+1,l+1)$, $w_{1,0}(k,l)/\Delta(k,l)$ and $w_{1,1}(k+1,l)/\Delta(k+1,l)$.  That is, the weights of the edges $((2k,2l+1),(2k+1,2l+2))$, $((2k+1,2l+2),(2k+2,2l+1))$,$((2k,2l+1),(2k+1,2l))$ and $((2k+1,2l),(2k+2,2l+1))$ are given by $w_{0,0}(k,l+1)/\Delta(k,l+1)$, $w_{0,1}(k+1,l+1)/\Delta(k+1,l+1)$, $w_{1,0}(k,l)/\Delta(k,l)$ and $w_{1,1}(k+1,l)/\Delta(k+1,l)$ respectively for $0 \leq k,l \leq n-2$. 
This will be shown in the proof of the lemma given below.
For this graph, let $\tilde{\mathcal{Z}}_{n-1}$ denote the number of weighted dimer covers and let $\tilde{\mathcal{Z}}_{n-1}(i,j)$ denote the number (weighted) of dimer covers when removing the vertices $(2i+1,0)$ and $(0,2j+1)$ from this graph for fixed $0 \leq i,j \leq n-2$.

\begin{lemma} \label{general:lem:recurrence}
For $\mathcal{Z}_n$,$\tilde{\mathcal{Z}}_n$,$\mathcal{Z}_{n}(i,j)$ and $\tilde{\mathcal{Z}}_{n}(i,j)$ be given above. We have
\begin{equation}\label{general:eqn:partition}
\mathcal{Z}_n =  \tilde{\mathcal{Z}}_{n-1}\prod_{k,l=0}^{n-1} \Delta (k,l)
\end{equation} 
and for fixed $0 \leq i,j \leq n-1$,
\begin{equation} \label{general:eqn:recurrence}
\begin{split}
\mathcal{Z}_n(i,j)&=\frac{ w_{0,1}(0,0)}{\Delta (0,0)} \tilde{\mathcal{Z}}_{n-1}\prod_{r,s=0}^{n-1} \Delta (r,s) \\
+& \left(\prod_{r,s=0}^{n-1} \Delta (r,s) \right)
\sum_{ \substack{k \in \{i-1,i\} \\ {l \in \{j-1,j\}}}} \frac{w_{0,i-k}(i,0)}{\Delta(i,0)} \frac{w_{1+l-j,1}(0,j)}{\Delta(0,j)}
\tilde{\mathcal{Z}}_{n-1} (k,l) \mathbbm{I}_{0 \leq k \leq n-2} \mathbbm{I}_{0 \leq l \leq n-2}
\end{split}
\end{equation}
\end{lemma}
In the above lemma, we chose to remove vertices  from the left and bottom boundaries of the Aztec diamond.  Analogous results can be obtained for the other boundary pairings. 

\begin{proof}
We first show that the edges weight of $\tilde{\mathcal{Z}}_{n-1}$ are obtained from applying the urban renewals, removal of pendant edges and edge contraction as detailed in Lemma~\ref{lem:oneperiodicrecurrence} and then show~\eqref{general:eqn:partition}.  Finally, we show~\eqref{general:eqn:recurrence}.

For the Aztec diamond of size $n$ which corresponds to $\mathcal{Z}_n$, the edges around the face whose center has coordinates $(2k+1,2l+1)$  have weight $$\{ w_{0,0}(k,l), w_{0,1}(k,l), w_{1,0}(k,l), w_{1,1}(k,l)\}$$ where we use the same labeling procedure as above. The urban renewal factor from this face is exactly equal to $\Delta (k,l)$ and when we apply urban renewal to this face, the edge weights around the small square (same labeling procedure as above) read
\begin{equation} \label{general:eqn:weightchange}
\left\{ \frac{w_{1,1}(k,l)}{\Delta(k,l)},\frac{w_{1,0}(k,l)}{\Delta(k,l)}, \frac{w_{0,1}(k,l)}{\Delta(k,l)},\frac{w_{0,0}(k,l)}{\Delta(k,l)}\right\}.
\end{equation}  
When applying the removal of pendant edges and edge contractions, as given in Lemma~\ref{lem:oneperiodicrecurrence}, we obtain an Aztec diamond of order $n-1$. Moreover, this Aztec diamond of order $n-1$ is contained within the Aztec diamond of size $n$: the Aztec diamond of size $n-1$ consists of the vertices and edges around the faces $(2r+2,2s+2)$ for $0 \leq r,s \leq n-2$ of an Aztec diamond with a weight change explained from applying urban renewal $n^2$ times (i.e. to the faces with coordinates $(2k+1,2l+1)$ for all $0 \leq k,l \leq n-1$).  To find these edge weights of the Aztec diamond of size $n-1$, it is enough to find the edge weights around the face with center $(2r+2,2s+2)$ for $0 \leq r,s \leq n-2$ for an Aztec diamond of size $n$ where the edge weights around the face with center $(2k+1,2l+1)$ are given by~\eqref{general:eqn:weightchange}.  We find that these edge weights  are given by 
$$\left\{ \frac{w_{0,0}(r,s+1)}{\Delta(r,s+1)} , \frac{ w_{0,1}(r+1,s+1)}{\Delta(r+1,s+1)}, \frac{ w_{1,0}(r,s)}{\Delta (r,s)} , \frac{w_{1,1}(r+1,s)}{\Delta(r+1,s)}\right\}. $$
Hence, we have that the edge weights around the face with center $(2k+1,2l+1)$ for an Aztec diamond of size $n-1$ are given by
$$\left\{ \frac{w_{0,0}(k,l+1)}{\Delta(k,l+1)} , \frac{ w_{0,1}(k+1,l+1)}{\Delta(k+1,l+1)}, \frac{ w_{1,0}(k,l)}{\Delta (k,l)} , \frac{w_{1,1}(k+1,l)}{\Delta(k+1,l)} \right\} $$
which are exactly the edge weights encoded by $\tilde{\mathcal{Z}}_{n-1}$.

Equation~\eqref{general:eqn:partition} follows exactly from the construction of $\tilde{\mathcal{Z}}_{n}$ from $\mathcal{Z}_n$ because the product on the right-hand side of~\eqref{general:eqn:partition} is exactly equal to the product of all the urban renewal factors.

For equation~\eqref{general:eqn:recurrence}, we follow the steps given in the proof of Lemma~\ref{lem:oneperiodicrecurrence} which generate a recurrence for $\mathcal{Z}_n(i,j)$ in the case when all the edge weights are set to equal to $1$.  As in the proof of Lemma~\ref{lem:oneperiodicrecurrence}, we add an auxiliary edge incident to $(2i+1,0)$ and another auxiliary edge incident to $(0,2j+1)$.  As the shape of the graph  remains the same after applying the same steps given in Lemma~\ref{lem:oneperiodicrecurrence} (for the computation of $\mathcal{Z}_n(i,j)$), we only need to compute the analog of~\eqref{lemproof:recurrencefirststep} for the above choice of edge weights, which is the effect on $\mathcal{Z}_n(i,j)$ after applying urban renewal ($n^2$ times), removal of pendant edges and edge contractions.   Therefore, we follow the computation to obtain~\eqref{lemproof:recurrencefirststep} in Lemma~\ref{lem:oneperiodicrecurrence} noting the following differences due to the choice of edge weights:
\begin{itemize}
\item The product of the urban renewal factors is given by $
\prod_{r,s=0}^{n-1} \Delta (r,s)$ (as opposed to $2^{n^2}$).
\item For $(i,j) \not = (0,0)$, the edge $(v_2,(2k+1,0))$ for $0 \leq k \leq n-2$ has weight $w_{0,0}(i,0)/\Delta(i,0)$ if $k=i$ and weight $w_{0,1}(i,0)/\Delta(i,0)$ if $k=i-1$.  Note that the vertex $v_2$ has the same definition as given in Lemma~\ref{lem:oneperiodicrecurrence}. 
\item For $(i,j) \not = (0,0)$, the edge $(v_3,(0,2l+1))$ for $0 \leq l \leq n-2$ has weight $w_{1,1}(0,j)/\Delta(0,j)$ if $l=j$ and weight $w_{0,1}(0,j)/\Delta(0,j)$
if $l=j-1$.  The vertex $v_3$ has the same definition as given in Lemma~\ref{lem:oneperiodicrecurrence}.
\item For $(i,j)=(0,0)$, the edges $(v_2,(1,0))$ and $(v_3,(0,1))$  have weights $w_{0,0}(0,0)/\Delta(0,0)$ and $w_{1,1}(0,0)/\Delta(0,0)$ respectively.  The edge $(v_2,v_3)$ has weight $w_{0,1}(0,0)/\Delta (0,0)$ and recall that when this edge is matched, the resulting graph is an Aztec diamond of order $n-1$ whose weighted number of matchings is equal to $\tilde{\mathcal{Z}}_{n-1}$.
\end{itemize}
The above edges $(v_2, (2k+1,0))$ and $(v_3,(0,2l+1))$ are a consequence of the auxiliary edges --- they are the resulting edges after applying urban renewal ($n^2$ times), removal of pendant edges and edge contractions. After matching the edges $(v_2, (2k+1,0))$ and $(v_3,(0,2l+1))$, the number of weighted matchings on the remaining graph is equal to $\tilde{\mathcal{Z}}_{n-1}(k,l)$. 
By following the computation detailed in Lemma~\ref{lem:oneperiodicrecurrence} to find~\eqref{lemproof:recurrencefirststep} with the edge weights found above, we obtain~\eqref{general:eqn:recurrence}.  

\end{proof}

\subsection{Proof of Theorem~\ref{thm:oneperiodic}}

The proof of Theorem~\ref{thm:oneperiodic} follows the same steps as the one for Theorem \ref{thm:uniform}.  Generalizing the proof to arbitrary $a$ requires no new ideas, but rather increases the complexity and amount of bookkeeping required:  
\begin{itemize}
\item we first find the boundary recurrence relation (without the Kasteleyn orientation) which is the analog of Lemma~\ref{lem:oneperiodicrecurrence} for biased tilings using Section~\ref{section:general}.
\item From this boundary recurrence relation, we find the analog of Lemma~\ref{lem:oneperiodicbgf} which is given in Lemma~\ref{bgf:oneperiodicanalog}
\item We next find the analogs of Lemmas~\ref{uniform:lem:movevertices} and~\ref{uniform:lem:sign} which are  given in Lemma~\ref{oneperiodic:lem:movevertices} and~\ref{oneperiodic:lem:sign} 
\item Finally, we use the above lemmas to give the proof of Theorem~\ref{thm:oneperiodic}.
\end{itemize}

Let $Z_n(r,a)$ denote the partition function of an Aztec diamond of size $n$ with weights $r$ and $ ra$ for horizontal and vertical edges respectively.  Let $Z(i,j,r,a,n)$ denote the number of matchings of an Aztec diamond of size $n$ with weights $r$ and $r a$ for horizontal and vertical edges respectively with the vertices $(2i+1,0)$ and $(0,2j+1)$ removed from the graph. 

\begin{lemma} \label{oneperiodicbias:lem:linearrecurrence}
For $i,j \in \{0,1,\dots, n-1\}$ we have
\begin{equation}\label{oneperiodbias:linearrecurrence}
	\frac{Z(i,j,1,a,n)}{Z_n(1,a)}=\sum_{\substack{k,l\in\{0,1\}\\ (i-k,l-k)\not=(-1,-1)}} a^{k+l} \frac{Z(i-k,j-l,1,a,n-1)}{Z_{n-1}(1,a)(1+a^2)}  +\frac{a}{1+a^2} \mathbb{I}_{(i,j)=(0,0),n\geq 1}.
\end{equation}
and $Z(i,j,1,a,n-1)=0$ unless both $i$ and $j$ are in $\{0,1,\dots,n-2\}$. 
\end{lemma}
\begin{proof}

By using Lemma~\ref{general:lem:recurrence} (i.e. set $w_{0,0}(k,l)=w_{1,1}(k,l)=1$ and $w_{0,1}(k,l)=w_{1,0}(k,l)=a$ for all $0\leq k,l \leq n-1$) and in particular from~\eqref{general:eqn:partition}, we have 
\begin{equation}
Z_n(1,a)= (1+a^2)^{n^2} Z_{n-1}\left(\frac{1}{1+a^2},a \right) 
\end{equation}
and by applying a gauge transformation which multiplies all the white vertices by $1+a^2$  we have
\begin{equation}
Z_n(1,a) = (1+a^2)^n Z_{n-1}(1,a).
\end{equation}
where $Z_0(1,a)=1$.
For $Z(i,j,1,a,n)$, from~\eqref{general:eqn:recurrence}, we have that
\begin{equation}
\begin{split}
Z(i,j,1,a,n)&=(1+a^2)^{n^2} \sum_{ \substack{k \in \{i-1,i\} \\ {l \in \{j-1,j\}}}}  \frac{a^{i-k+l-j}}{(1+a^2)^2} Z\left(k,l,\frac{1}{1+a^2},a,n-1 \right) 
\mathbbm{I}_{0 \leq k \leq n-2} \mathbbm{I}_{0 \leq l \leq n-2}\\
&+\frac{a}{1+a^2} (1+a^2)^{n^2} Z_{n-1}\left(\frac{1}{1+a^2},a \right)
\end{split}
\end{equation} 
and we apply gauge transformations to both terms which multiplies to the white vertices by $1+a^2$ giving
\begin{equation}
\begin{split}
Z(i,j,1,a,n)&=(1+a^2)^{n-1} \sum_{ \substack{k \in \{i-1,i\} \\ {l \in \{j-1,j\}}}}  a^{i-k+l-j} Z\left(k,l,1,a,n-1 \right) 
\mathbbm{I}_{0 \leq k \leq n-2} \mathbbm{I}_{0 \leq l \leq n-2}\\
&+\frac{a}{1+a^2} (1+a^2)^{n} Z_{n-1}\left(1,a \right)
\end{split}
\end{equation} 
We divide the recurrence for $Z(i,j,1,a,n)$ by the recurrence for $Z_n(1,a)$ and rearranging the summation, we obtain the lemma.
\end{proof}

Let
\begin{equation}
Z_{\partial}(w,b,a,z)=\sum_{n \geq 0 } \sum_{i=0}^{n-1} \sum_{j=0}^{n-1} \frac{Z (i,j,1,a,n)}{Z_n(1,a)}w^i b^j z^n
\end{equation}
which denotes the boundary generating function without Kasteleyn orientation for the bottom and left boundaries where $w$ and $b$ mark the white and black vertices respectively while $z$ marks the size of the Aztec diamond. 

\begin{lemma} \label{bgf:oneperiodicanalog}
\begin{equation}
Z_{\partial}(w,b,a,z)= \frac{a z}{(1-z)(1+a^2-z(1+wa)(1+ba) )}
\end{equation}
\end{lemma}
\begin{proof}
As given in the proof of Lemma~\ref{lem:oneperiodicbgf}, we multiply~\eqref{oneperiodbias:linearrecurrence} by $w^i b^j z^n$ and sum of $0\leq i,j \leq n-1$ and $n\geq 0$ and we obtain 
\begin{equation}\label{oneperidicbias:gffirst}
Z_{\partial}(w,b,a,z)= \frac{z(1+w a)(1+ba )}{1+a^2} Z_{\partial}(w,b,a,z)+ \frac{z}{1-z} \frac{a}{1+a^2}
\end{equation}
which is the boundary generating function at the bottom and left boundary of the biased one-periodic weighting of the Aztec diamond. Solving~\eqref{oneperidicbias:gffirst} gives the lemma.
\end{proof}

Let
\begin{equation}
\label{oneperiodic:bgfa:def:Hn}
H_n^{i,j,a}(\mathtt{w},\mathtt{b}) \sum_{\substack{1 \leq x_1\leq 2n-1,x_1 \mod2=1\\ 1 \leq y_2 \leq 2n-1,y_2 \mod2=1} } K_a^{-1}((x_1,2ni),(2nj,y_2)) w_1^{x_1} w_2^{2ni} b_1^{2nj} b_2^{y_2}
\end{equation}

\begin{lemma} \label{oneperiodic:lem:movevertices}
\begin{equation} 
\begin{split}
&G_n(\mathtt{w},\mathtt{b})= \frac{w_1 w_2^2 b_2  f_{n+1}(w_1^2 b_1^2) f_n(w_2^2 b_2^2)}{C(w_1,w_2)} \\
&+(1+\mathtt{i}a w_1)\frac{ (1+a \mathtt{i} b_2^2)H^{0,0,a}_n+b_1^2\left(b_2w_1f_n(b_1^2 w_1^2)+(a\mathtt{i}  +b_2^2) H^{0,1,a}_n \right)}{C(w_1,w_2) C(b_1,b_2)}  \\
&+(\mathtt{i}a +w_1^2)w_2^2 \frac{ b_1^2 b_2^{2n+1} w_1 w_2^{2n} f_n(b_1^2 w_1^2) + (1+ \mathtt{i} b_2^2)H_n^{1,0,a} +b_1^2(a\mathtt{i} +b_2^2) H_n^{1,1,a}}{C(w_1,w_2) C(b_1,b_2)}
\end{split}
\end{equation}
where $C(r_1,r_2)=1+r_1^2r_2^2+a \mathtt{i}(r_1^2+r_2^2)$ and $H_n^{i,j,a}=H_n^{i,j,a}(\mathtt{w},\mathtt{b})$ defined in~\eqref{oneperiodic:bgfa:def:Hn}.
\end{lemma}

\begin{proof}
We follow the proof of Lemma~\ref{uniform:lem:movevertices} but replace $K$  by $K_a$ defined in~\eqref{bgfa:K} and replace $K^{-1}$ by $K_a^{-1}$.  This results in setting $\mathtt{i}$ to $\mathtt{i} a$ and $H_n^{i,j}$ to $H_n{i,j,a}$ in the proof of Lemma~\ref{uniform:lem:movevertices}.

\end{proof}

\begin{lemma} \label{oneperiodic:lem:sign}
For $i,j \in \{0,1\}$, we have
\begin{equation}
H_n^{i,j,a}(\mathtt{w},\mathtt{b})=F_n^{i,j}(\mathtt{w},\mathtt{b},a^{1-2(j(1-i)-i(1-j))})
\end{equation}	
where $H_n^{i,j,a}(\mathtt{w},\mathtt{b})$ is defined in~\eqref{oneperiodic:bgfa:def:Hn} and $F_n^{0,0}(\mathtt{w},\mathtt{b},a),F_n^{0,1}(\mathtt{w},\mathtt{b},a^{-1}), F_n^{1,0}(\mathtt{w},\mathtt{b},a^{-1})$ and $F_n^{1,1}(\mathtt{w},\mathtt{b},a)$ are given in equations~\eqref{bfga:h00},~\eqref{bfga:h02n},\eqref{bfga:h2n0} and~\eqref{bfga:h2n2n} respectively.
\end{lemma}

\begin{proof}
The computations for $(i,j)=(0,0)$ and $(i,j)=(1,1)$ are analogs of Lemma~\ref{uniform:lem:sign} using the same Kasteleyn orientation but using the boundary generating function (without Kasteleyn orientation) from Lemma~\ref{bgf:oneperiodicanalog} instead. 

For $(i,j)=(1,0)$, we cannot use the boundary generating function computed in Lemma~\ref{bgf:oneperiodicanalog} directly ---
 we need to interchange the vertical and horizontal edge weights so that we can use our previous computations (this interchange accounts for the top leftmost edge having weight 1 while the bottom leftmost edge having weight $a$). Consider an Aztec diamond with all horizontal edges having weight $a$  and all vertical edges having weight $1$.  By the above notation, its partition function is $Z_n(a,1/a)$ and the partition function when removing two vertices $(2i+1,0)$ and $(0,2j+1)$ is given by $Z(i,j,a,1/a,n)$.  By multiplying all the white vertices by $1/a$, we recover an Aztec diamond with horizontal edges having weight $1$ and vertical edges having weight $1/a$.  Because there is one less white vertex in $Z(i,j,a,1/a,n)$ that $Z_n(a,1/a)$, this gauge transformation leads to
\begin{equation}
	\frac{Z(i,j,a,1/a,n)}{Z_n(a,1/a)}=\frac{1}{a} \frac{Z(i,j,1,a^{-1},n)}{Z_n(1,a^{-1})}
\end{equation}
From the above equation, we compute the boundary generating function for $Z(i,j,a,1/a,n)/Z_n(a,1/)$ (i.e. set $a$ to $1/a$ and multiply by $1/a$
in Lemma~\ref{bgf:oneperiodicanalog}).  From this boundary generating function, we follow the proof in Lemma~\ref{uniform:lem:sign} for the Kasteleyn orientation and change of variables and we recover~\eqref{bfga:h02n}.  A similar computation holds for $(i,j)=(0,1)$.

\end{proof}

We now give the proof of Theorem~\ref{thm:oneperiodic}.
\begin{proof}[Proof of Theorem~\ref{thm:oneperiodic}]

In Lemma~\ref{oneperiodic:lem:sign}, we found the expressions for $H^{i,j,a}_n(\mathtt{w},\mathtt{b})$ for all $i,j \in \{0,1\}$.  We substitute these expressions for $H^{i,j,a}_n(\mathtt{w},\mathtt{b})$ into the formula given in Lemma~\ref{oneperiodic:lem:movevertices} which gives the generating function for $K_a^{-1}$.  

\end{proof}

\section{$q^{\mathrm{vol}}$ weighting} \label{section q-vol for the Aztec Diamond}

As mentioned in the introduction, it is possible to associate a discrete stepped surface to a covering of the Aztec diamond.  This function is called the \emph{height function}.  The most concrete way of viewing the height function is as the surface of a certain stack of blocks, called \emph{Levitov blocks}~\cite{Lev:89,Lev:90}.  One can construct an edge weighting of the Aztec diamond in several ways, so that the contribution from each covering is proportional to $q^{\#\{\text{Levitov Blocks}\}}$ - that is, adding a block to a covering multiplies the weight of the covering by $q$.  
Concretely, this means that for each face, the product of the edges parallel to $e_2$ divided by the product of edges parallel to $e_1$ is equal to $q$ or $q^{-1}$, depending on the parity of the face.  Such a choice of weightings is not unique, and we refer to any such choice as a \emph{$q^{\mathrm{vol}}$ weighting} of the Aztec diamond.  For technical reasons, we make use of two such weightings here.

We call the weighting $q^{\mathrm{col}}$ to be the choice of weights where all edges parallel to $e_1$ have weight 1 and the edges parallel to $e_2$ have edge weights organized in columns given by $ a q^{2n}, a q^{1-2n} , a q^{2n-2},\dots, a q^{-1}$ reading from left to right -- see Figure \ref{fig:qvol weightings}.   We let $K_{\mathrm{col}}$ denote the Kasteleyn matrix with $q^{\mathrm{col}}$ weights and its entries are given by 
\begin{equation}\label{Kcol}
	K_{\mathrm{col}}(x,y)=\left\{ \begin{array}{ll}
	        1 & y-x=\pm e_1 \\
	        a q^{-2n+x_1-1} \mathtt{i} & y-x=e_2 \\
	        a q^{2n-x_1} \mathtt{i} & y-x=-e_2\\
	        0 &\mbox{otherwise}	        
	       \end{array} \right.
\end{equation}
where $x=(x_1,x_2)\in \mathtt{B}$ and $y\in \mathtt{W}$.  

\begin{figure}
\caption{ The two $q^{\mathrm{\mathrm{vol}}}$ weightings used in this paper for Aztec diamond of size 3.  The figure on the left is $q^{\mathrm{col}}$ weighting and the figure on the right is the $q^{\mathrm{diag}}$. }
\includegraphics[height=3in]{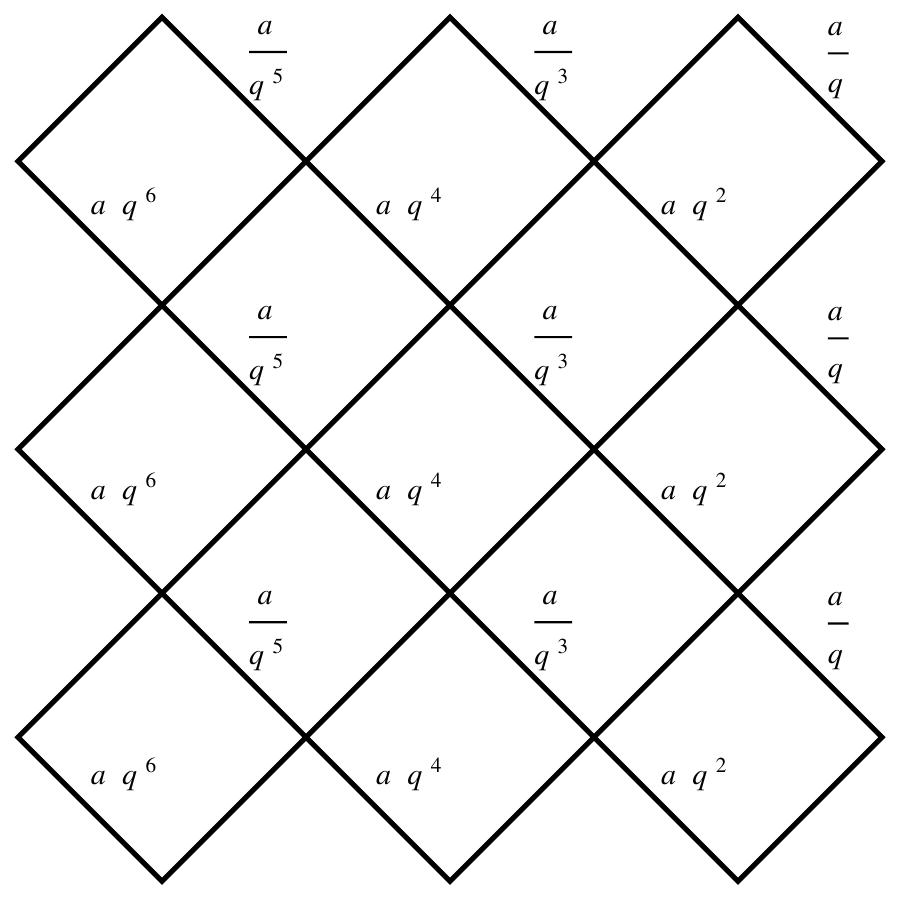}
\includegraphics[height=3in]{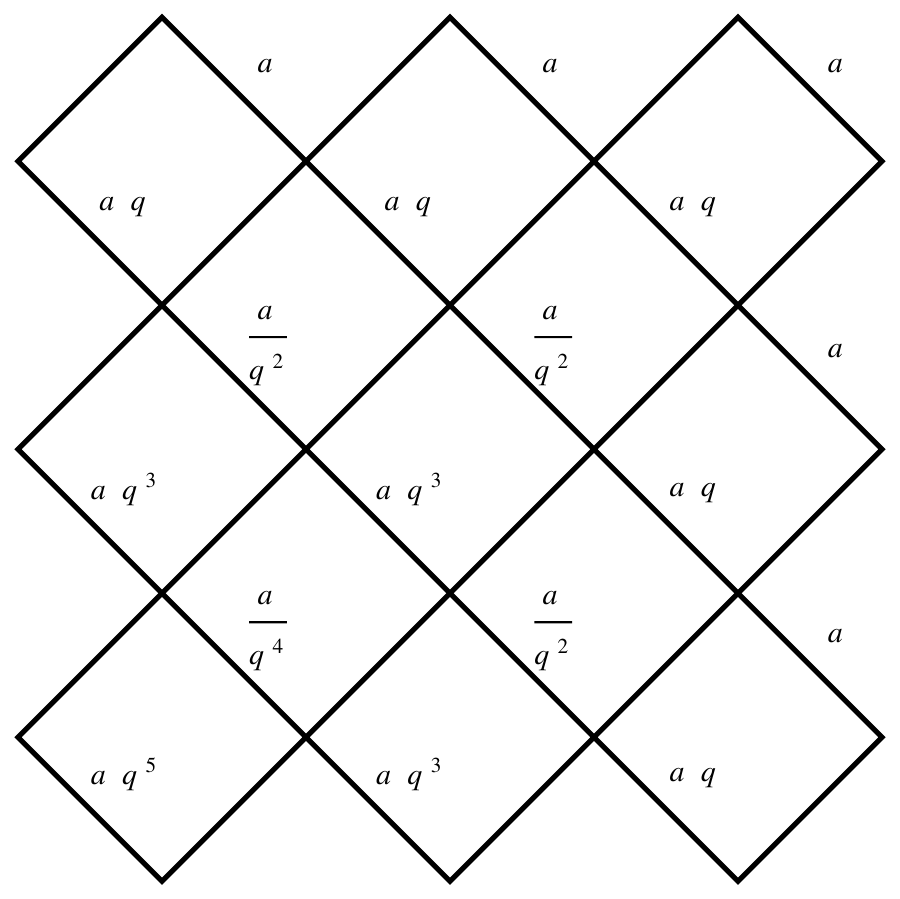}
\label{fig:qvol weightings}
\end{figure}

\begin{thma} \label{thm:qvolinverseK}
The entries of the inverse of $K_{\mathrm{col}}$ are given by 
\begin{equation}
      K_{\mathrm{col}}^{-1}(x,y)= \left\{ \begin{array}{ll} 
	                      f_1(x,y) & \mbox{if } x_1<y_1+1 \\
	                      f_1(x,y) + f_2(x,y) & x_1 \geq y_1+1
	                     \end{array} \right.
\end{equation}
for $x=(x_1,x_2) \in \mathtt{W}$ and $y=(y_1,y_2)\in \mathtt{B}$, where 
\begin{equation}
 \begin{split}
	f_1(x,y)&=\frac{\mathtt{i}^{(-x_1+x_2+y_1-y_2+2)/2} q^{(4 + 4 n - x_1 + x_2 - y_1 - y_2) (x_1 - x_2 - y_1 + y_2)/4}}{(2\pi \mathtt{i})^2} \int_{\Gamma_{1/a,q}^n} \int_{\Gamma_0} dz \, dw \\
	& \frac{w^{y_1/2}}{z^{(x_1+1)/2} (w-z)} \frac{ \prod_{k=0}^{x_2/2-1} (z+a q^{-2k-1+y_2-1}) \prod_{k=0}^{n-x_2/2-1} (azq^{2k+x_2+3-y_2}-1)}{\prod_{k=0}^{(y_2-1)/2} (w+a q^{2k-1}) \prod_{k=0}^{n-(y_2+1)/2} (a w q^{2k+2} -1)}
 \end{split} \label{qvol:f1}
\end{equation}
and
\begin{equation}
 \begin{split} \label{qvol:f2}
	f_2(x,y)&=\frac{ \prod_{k=0}^{(x_1-y_1-3)/2}(- \mathtt{i} q^{-2n+x_1-4-2k}) \prod_{k=0}^{(x_2-y_2-3)/2} \mathtt{i} q^{2n-x_1+4+2k}}{2\pi \mathtt{i}} \int_{\Gamma^n_{1/a,q}} dw \\
	& w^{y_1/2-(x_1+1)/2}   \frac{ \prod_{k=0}^{x_2/2-1} (w+a q^{-2k-1+y_2-1}) \prod_{k=0}^{n-x_2/2-1} (awq^{2k+x_2+3-y_2}-1)}{\prod_{k=0}^{(y_2-1)/2} (w+a q^{2k-1}) \prod_{k=0}^{n-(y_2+1)/2} (a w q^{2k+2} -1)}
\end{split}
\end{equation}
where $\Gamma_0$ is a contour surrounding the origin and $\Gamma^n_{1/a,q}$ is a contour surrounding $1/(aq^2)$, $1/(a q^4)$,$\dots$, $1/(a q^{2n})$ which does not intersect with $\Gamma_0$.

\end{thma}

When $q =1$, we recover the double contour integral form for $K^{-1}$ of a one-periodic weighting of the Aztec diamond with edge weights 1 and $a$ for the horizontal and vertical edges.  Unlike the formulas for the one-periodic case (Theorems~\ref{thm:uniform} and~\ref{thm:oneperiodic}), we were unable to find a generating function form for $K_{\mathrm{col}}^{-1}$.  However, we believe that Theorem~\ref{thm:qvolinverseK} is written in the most suitable form for asymptotic computations --- one can apply the saddle point analysis to compute both local and global behavior of the model.

The rest of this section is organized as follows:
similar to the proof of the one-periodic case, we first derive the boundary generating function for a gauge equivalent weighting of the $q^{\mathrm{col}}$ weighting.  From this boundary generating function, we show how we can obtain coefficients of the boundary generating function of the $q^{\mathrm{col}}$ weighting.  Unlike the one-periodic and two-periodic cases, we do not derive $K^{-1}_{\mathrm{col}}$ directly from the expansion of the boundary generating function.  Instead, we guess a formula for $K^{-1}_{\mathrm{col}}$ which is outlined in Section~\ref{subsection Finding K}.  We then prove that this guess is correct in Section~\ref{Appendixqvol}.

\subsection{Boundary Generating functions on the top and right boundaries}

We focus on extracting the coefficients of the boundary generating function for $q^{\mathrm{col}}$ weights on the top boundary (the white vertices $(2k-1,2n)$ for $0\leq k \leq n-1$) and the right boundary (the black vertices $(2n,2k+1)$ for $0 \leq k \leq n-1$).  For the rest of this section, label the white vertices on the top boundary $0$ to $n-1$ with the label $i$ representing the vertex $(2n-2i-1,2n)$.  Similarly, label the black vertices on the right boundary $0$ to $n-1$ with the label $j$ representing the vertex $(2n,2n-1-2j)$.  

Although we can find a boundary recurrence relation for the $q^{\mathrm{col}}$ weighting, it seems intractable (to us at least) to find an explicit solution.  This inconvenience  can be bypassed by working on another (gauge equivalent) choice of $q^{\mathrm{vol}}$ weighting.  

We let $q^{\mathrm{diag}}$ be the $q^{\mathrm{vol}}$ weighting where all edges parallel to $e_1$ have weight 1 and all edges parallel to $e_2$ on the top and right boundary have weight $a$.  Because the $q^{\mathrm{diag}}$ weighting is a $q^{\mathrm{vol}}$ weighting which has all horizontal edges equal to $1$, this choice of boundary weights uniquely determines the remaining edge weights -- see Figure~\ref{fig:qvol weightings}.   Let $K_{\mathrm{diag}}$ denote the Kasteleyn matrix for the $q^{\mathrm{diag}}$ weighting.  Its entries are given by 
\begin{equation} \label{qvol:Kdiag}
	K_{\mathrm{diag}}(x,y)=\left\{ \begin{array}{ll}
	        1 & y-x=\pm e_1 \\
	        \mathtt{i} a q^{-2\min(n-x_1/2,n-(x_2+1)/2)}  & y-x=e_2 \\
	        \mathtt{i} a q^{2\min(n-x_1/2-1,n-(x_2+1)/2))+1}  & y-x=-e_2\\
	        0 &\mbox{otherwise}	        
	       \end{array} \right.
\end{equation}
where $x=(x_1,x_2) \in \mathtt{B}$ and $y \in \mathtt{W}$.

\begin{lemma} \label{lem:qvolgaugeequivalent}

The gauge transformation from the $q^{\mathrm{col}}$ weighting to the $q^{\mathrm{diag}}$ weighting is given by multiplying the white vertices  $y=(y_1,y_2)$ by $q^{((y_2-y_1+1)/2)^2}$ if $y_2>y_1$ and $q^{(y_2-y_1+1)/2}$ if $y_1>y_2$ and multiplying the black vertices $x=(x_1,x_2)$ by $q^{-((x_2-x_1+1)/2)^2}$ if $x_1<x_2$ and $q^{(x_1-x_2-1)/2}$ if $x_1>x_2$.
\end{lemma}

\begin{proof}
Consider the $q^{\mathrm{col}}$ weighting with the gauge transformation described in the lemma.   Label this gauge transformation $g$ so that $g(v)$ represents the gauge transform which multiplies the vertex $v$ by $g(v)$. 

From the above gauge transformation, for $x=(x_1,x_2) \in \mathtt{B}$, we compute a new Kasteleyn matrix, labeled $\tilde{K}$ which is given entry wise by $K_{\mathrm{col}} (x,y) g(x)g(y)$ where $y \in \mathtt{W}$ is a nearest neighbor vertex to $x$. We find that for $x_1<x_2$, the entries of this new Kasteleyn matrix are  
\begin{itemize}
\item  $\tilde{K}_{\mathrm{col}}(x,x+e_1).q^{-((x_2-x_1+1)/2)^2}.q^{((x_2-x_1+1)/2)^2}=1$,
\item $\tilde{K}_{\mathrm{col}}(x,x-e_1).q^{-((x_2-x_1+1)/2)^2}.q^{((x_2-x_1+1)/2)^2}=1$,
\item $\tilde{K}_{\mathrm{col}}(x,x+e_2).q^{-((x_2-x_1+1)/2)^2}.q^{((x_2-x_1+3)/2)^2}=aq^{x_2-1-2n}\mathtt{i}$ and
\item $\tilde{K}_{\mathrm{col}}(x,x-e_2).q^{-((x_2-x_1+1)/2)^2}.q^{((x_2-x_1-1)/2)^2}=aq^{2n-x_2}\mathtt{i}$.
\end{itemize} 
For $x=(x_1,x_2) \in \mathtt{B}$, we similarly compute $K_{\mathrm{col}}(x,y) g(x) g(y)$ for $x_1>x_2$ and we obtain
\begin{itemize}
\item  $\tilde{K}_{\mathrm{col}}(x,x+e_1).q^{(x_1-x_2+1)/2}.q^{-(x_1-x_2+1)/2}=1$,
\item $\tilde{K}_{\mathrm{col}}(x,x-e_1).q^{(x_1-x_2+1)/2}.q^{-(x_1-x_2+1)/2)}=1$,
\item $\tilde{K}_{\mathrm{col}}(x,x+e_2).q^{(x_1-x_2+1)/2}.q^{-(x_1-x_2-1)/2}=aq^{x_1-2n}\mathtt{i}$ and
\item $\tilde{K}_{\mathrm{col}}(x,x-e_2).q^{(x_1-x_2+1)/2}.q^{-(x_1-x_2+3)/2}=aq^{2n-x_1-1}\mathtt{i}$.
\end{itemize} 
Using~\eqref{qvol:Kdiag}, it follows that $\tilde{K}$ is equal to $K_{\mathrm{diag}}$.
\end{proof}

We now write out the partition function and boundary recurrence for the $q^{\mathrm{vol}}$ weighting of the Aztec diamond. The latter recurrence will be written in terms of the vertices on the top and right boundaries, whereas in Section~\ref{section:general} we considered a recurrence with the bottom  and left  boundaries.  To fix this, we work with the faces with centers having coordinates $(2n-2k-1,2n-2l-1)$ and interchange $w_{0,1}$ and $w_{1,0}$ for an Aztec diamond of size $n$ for the remainder of this section.  That is, the edge weights around a face whose center has co-ordinates $(2n-2k-1,2n-2l-1)$ are given by $w_{0,0}(k,l),w_{1,0}(k,l),w_{0,1}(k,l)$ and $w_{1,1}(k,l)$ for $0\leq k ,l \leq n-1$ where we have used the same labeling procedure as given in Section~\ref{section:general} (i.e. the weights correspond to the edges $((2n-2k-2,2n-2l-1),(2n-2k-1,2n-2l)), ((2n-2k-1,2n-2l),(2n-2k,2n-2l-1)) , ((2n-2k-2,2n-2l-1),(2n-2k-1,2n-2l-2))$ and $((2n-2k-1,2n-2l-2),(2n-2k,2n-2l-1))$).
 With this notation change, let $Z_n^{\mathrm{diag}}(b,a,q)$ denote the partition function of the Aztec diamond with weights
\begin{equation}
\begin{split} \label{qvol:weighting1}
w_{0,0}(k,l)=b, \hspace{8mm} 
w_{1,0}(k,l)=a q^{-2\min(k,l)},  \\
w_{0,1}(k,l)=a q^{2\min(k,l)+1}\hspace{2mm} \mbox{and }
w_{1,1}(k,l)=b
\end{split} 
\end{equation}
for $0 \leq k,l \leq n-1$,
i.e. $Z_n^{\mathrm{diag}}(b,a,q)$ is the partition function of the $q^{\mathrm{diag}}$ weighting of an Aztec diamond of size $n$. Let $\tilde{Z}_{n-1}^{\mathrm{diag}}(b,a,q)$ be the partition function of an Aztec diamond of size $n-1$ whose faces have edge weights
\begin{equation}
\begin{split} \label{qvol:weighting2}
\tilde{w}_{0,0}(k,l)=b, \hspace{8mm} 
\tilde{w}_{1,0}(k,l)=a q^{-2\min(k,l))},  \\
\tilde{w}_{0,1}(k,l)=a q^{2 \min (k,l)+3}\hspace{2mm} \mbox{and }
\tilde{w}_{1,1}(k,l)=b
\end{split} 
\end{equation}
for $0\leq k,l\leq n-2$. We denote $Z^{\mathrm{diag}}(i,j,a,q,n)$ to be the partition function of an Aztec diamond of size $n$ whose edge weights are given by~\eqref{qvol:weighting1} with $b=1$ and where $i$ and $j$ denote removing the vertices $(2n-2i-1,2n)$ and $(2n,2n-2j-1)$ respectively from the graph.  We denote $Z^{\mathrm{diag}}(i,j,b,a,q,n-1)$  to be the partition function of an Aztec diamond of size $n-1$ whose edge weights are given by~\eqref{qvol:weighting2} where $i$ and $j$ denote removing the vertices $(2n-2i-3,2n-2)$ and $(2n-2,2n-2j-3)$ respectively from the graph.  We set $Z^{\mathrm{diag}}(i,j,a,q,n)=Z^{\mathrm{diag}}(i,j,c,a,q,n)=0$ if either $i$ or $j$ (or both) are not contained in $\{0,1,\dots, n-1\}$.

\begin{lemma} \label{lem:qvolrecurrence}
We have that $Z^{\mathrm{diag}}(i,j,a,q,n)/Z^{\mathrm{diag}}_n(1,a,q)$ satisfies the following recurrence

      \begin{equation}
       \frac{Z^{\mathrm{diag}}(i,j,a,q,n)}{Z^{\mathrm{diag}}_n(1,a,q)}=q^{i+j+1} \left(\sum_{\substack{ (k,l) \in \{0,1\}\\ (i-k,j-k) \\ \not =(-1,-1)}} a ^{k+l} \frac{Z^{\mathrm{diag}}(i-k,j-l,aq,q,n-1)}{(1+a^2 q)Z_{n-1}^{\mathrm{diag}}(1,aq,q)} +\frac{a  \mathbbm{I}_{(i,j)=(0,0),n\geq 1}}{1+a^2 q} \right).
      \end{equation}

\end{lemma}
      
\begin{proof}
We first compute the recurrence for $Z_n^{\mathrm{diag}}(1,a,q)$ and then compute the recurrence for $Z^{\mathrm{diag}}(i,j,a,q,n)$.

For $Z_n^{\mathrm{diag}}(1,a,q)$ we compute the urban renewal factors and using the notation from Section~\ref{section:general} we find
\begin{equation}
\Delta(k,l)=w_{0,0}(k,l)w_{1,1}(k,l)+w_{0,1}(k,l) w_{1,0}(k,l)=1+a^2 q
\end{equation}
 for all $0\leq k,l\leq n-1$. Because of the change in notation from Section~\ref{section:general}, the edge weights under the deformation of the Aztec diamond of size $n$ to an Aztec diamond of size $n-1$ as detailed in Section~\ref{section:general} are equal to $w_{0,0}(k+1,l)/\Delta(k,l)$, $w_{1,0}(k,l)/\Delta(k,l)$, $w_{0,1}(k+1,l+1)/\Delta(k,l)$ and $w_{1,1}(k,l+1)/\Delta(k,l)$ around the face whose center has coordinates $(2n-2k-3,2n-2l-3)$. Comparing with~\eqref{qvol:weighting2}, we have
\begin{equation}
w_{r,s}(k+(1-r),l+s) = \tilde{w}_{r,s}(k,l) \Delta(k,l)
\end{equation}
for $r,s \in \{0,1\}$ and so we  use~\eqref{general:eqn:partition} to obtain
\begin{equation} \label{qvol:weightchange0}
	\frac{Z^{\mathrm{diag}}_n(1,a,q)}{(1+a^2 q)^{n^2}}=\tilde{Z}_{n-1}^{\mathrm{diag}} \left( \frac{1}{1+a^2 q}, \frac{a}{1+a^2 q},q \right) 
 \end{equation}
As all the edge weights on the right-hand side are divided by $1+a^2 q$, we  apply a gauge transformation to  $Z_{n-1}^{\mathrm{diag}}(\cdot)$ on the right-hand side of the above equation which multiplies the white vertices by $(1+a^2 q)$. As this term encodes $n(n-1)$ white vertices, we find
 \begin{equation} \label{qvol:weightchange1}
	\frac{Z^{\mathrm{diag}}_n(1,a,q)}{(1+a^2q)^n}= \tilde{Z}_{n-1}^{\mathrm{diag}} (1,a,q) 
 \end{equation}
Setting $a \mapsto a q^{-1}$ gives
\begin{equation} \label{qvol:weightchange2}
\left.      \frac{Z^{\mathrm{diag}}_n(1,a,q)}{(1+a^2 q)^n} \right|_{a \mapsto a q^{-1}} =\tilde{Z}_{n-1}^{\mathrm{diag}}(1,aq^{-1},q)
\end{equation}
To the right-hand side of the above equation, we apply a gauge transformation which multiplies the white vertices at $(2n-2i-3,2n-2)$ by $q^{1+i}$ for all $0 \leq i \leq n-2$, the black vertices at $(2n-2,2n-2j-3)$ by $q^j$ for all $0 \leq j \leq n-2$ and the remaining black and white vertices are multiplied by the appropriate factor so that all the horizontal edges have weight 1 (that is, multiply each white vertex $(x_1,x_2)$ by $q^{(x_2-x_1-1)/2}$ and each black vertex $(x_1,x_2)$ by $q^{(x_1-x_2-1)/2}$). Due to the overall multiplication factor of this gauge transform is $q^{0}$, we find that 
\begin{equation} \label{qvol:weightchange3}
	\tilde{Z}_{n-1}^{\mathrm{diag}} (1,aq^{-1},q)=Z_{n-1}^{\mathrm{diag}}(1,a,q)
\end{equation}
and after substituting~\eqref{qvol:weightchange3} into~\eqref{qvol:weightchange2} and setting $a \mapsto a q$ on both sides in~\eqref{qvol:weightchange2}, we find
\begin{equation} \label{qvol:partitionrecurrence}
      Z^{\mathrm{diag}}_n(1,a,q)=(1+a^2 q)^n Z_{n-1}^{\mathrm{diag}} (1,a q ,q).
\end{equation}

We now find the recurrence for $Z^{\mathrm{diag}}(i,j,a,q,n)$ following similar  transformations used in computing the partition function recurrence.  To make the computations simpler, we will divide through by $(1+a^2 q)^{n-1}$ and set $a \mapsto a q^{-1}$ and so we will  compute a relation for 
\begin{equation}\label{qvol:weightrecurrence1}
\left. \frac{Z^{\mathrm{diag}} (i,j,a,q,n)}{(1+a^2 q)^{n-1}} \right|_{a \mapsto a q^{-1}}
\end{equation}
For the above equation, we  use~\eqref{general:eqn:recurrence} but as the formulas are rather long we shall treat each term of the right-hand side of~\eqref{general:eqn:recurrence} separately. Substituting the first term on the right-hand side of~\eqref{general:eqn:recurrence} with the $q^{\mathrm{diag}}$ weighting into~\eqref{qvol:weightrecurrence1}, because $w_{0,1}(0,0)/\Delta(0,0)=aq/(1+a^2 q)$
\begin{equation} \label{qvol:weightrecurrence2}
\begin{split}
& \left[ a q(1+a^2 q)^{n^2-n}  \tilde{Z}^{\mathrm{diag}}_{n-1}\left( \frac{1}{1+a^2q} , \frac{a}{1+a^2 q} ,q \right) \right]_{a\mapsto a q^{-1}}
=a \tilde{Z}_{n-1}^\mathrm{diag} (1,a q^{-1},q) 
= a Z_{n-1}^{\mathrm{diag}}(1,a,q)
\end{split}
\end{equation}
where the second line follows from the gauge transformation multiplying all $n(n-1)$ white vertices by $(1+a^2 q)$ as given computed in~\eqref{qvol:weightchange0} followed by setting $a\mapsto a q^{-1}$  while the third line follows from~\eqref{qvol:weightchange3}.
Substituting the second term on the right-hand side of~\eqref{general:eqn:recurrence} with the $q^{\mathrm{diag}}$ weighting, we find that because $w_{0,1}(r,s)=a q$ for all $0\leq r\leq n-1$ provided $s=0$ and for all $0\leq s \leq n-1$ provided $r=0$, the
 second term of~\eqref{qvol:weightrecurrence1} is equal to
\begin{equation}
\begin{split} \label{qvol:weightrecurrence3}
&=	\left. \left( (1+a^2 q)^{n^2-n+1} \sum_{ \substack{ k \in \{i-1,i \} \\ l \in \{j-1,j\}}} \frac{(aq)^{i-k+j-l}}{(1+a^2 q)^2}  \tilde{Z}^{\mathrm{diag}}\left(k,l,\frac{1}{1+a^2 q},\frac{a}{1+a^2 q},q,n-1\right) \mathbbm{I}_{0 \leq k,l \leq n-2 }\right) \right|_{a \mapsto a q^{-1}} \\
&= \left.\left( \sum_{\substack{k \in \{i-1,i \} \\ l\in \{j-1,j\}} } (aq)^{i-k+j-l}\tilde{ Z}^{\mathrm{diag}} (k,l,1,a,q,n-1) \mathbbm{I}_{0 \leq k,l \leq n-2 }\right) \right|_{a \mapsto a q^{-1}}\\
&=  \sum_{\substack{k \in \{i-1,i \} \\ l\in \{j-1,j\}} }  a^{i-k+j-l} \left( \left. \tilde{ Z}^{\mathrm{diag}} (k,l,1,a,q,n-1)  \mathbbm{I}_{0 \leq k,l \leq n-2 }\right|_{a \mapsto a q^{-1}} \right)\\
&=    \sum_{\substack{k \in \{i-1,i \} \\ l\in \{j-1,j\}} } q^{k+l+1}  a^{i-k+j-l}    Z^{\mathrm{diag}} (k,l,a,q,n-1)  \mathbbm{I}_{0 \leq k,l \leq n-2 }
\end{split}
\end{equation}
where the first line to the second line follows from using apply a gauge transformation which multiplies the $n(n-1)-1$ white vertices  by $(1+a^2 q)$.  The third to fourth line in the above block of equations is due to the gauge transformation used in~\eqref{qvol:weightchange3} which is be seen by the following: $\tilde{ Z}^{\mathrm{diag}} (k,l,1,a,q,n-1) |_{a \mapsto a q^{-1}}$ has the same weights as $\tilde{Z}_{n-1}^{\mathrm{diag}}(1,a q^{-1},q)$ with the vertices $(2n-2k-3,2n-2)$ and $(2n-2,2n-2l-3)$ removed from the Aztec diamond. We apply the same gauge transformation given in~\eqref{qvol:weightchange3} and the factor $q^{k+l+1}$ compensates for these two removed vertices.  Using~\eqref{qvol:weightrecurrence3} and~\eqref{qvol:weightrecurrence2},~\eqref{qvol:weightrecurrence1} becomes
\begin{equation}
\left. \frac{Z^{\mathrm{diag}} (i,j,a,q,n)}{(1+a^2 q)^{n-1}} \right|_{a\mapsto a q^{-1}}= q\sum_{\substack{k \in \{i-1,i\} \\ l \in \{j-1,j\}}} q^{k+l}a^{i-k+j-l}Z^{\mathrm{diag}} (k,l,a,q,n-1)  \mathbbm{I}_{0 \leq k,l \leq n-2 }+ a Z_{n-1}^{\mathrm{diag}} (1,a,q)
\end{equation}
We  set $a \mapsto aq$ on both sides of the above equation and divide by the partition function recurrence given in~\eqref{qvol:partitionrecurrence}.  A change of summation index gives the result.

\end{proof}

Following the steps outlined in Section~\ref{section Uniform case}, we now find the boundary generating function (ignoring the Kasteleyn orientation) for the $q^{\mathrm{diag}}$ weighting of the Aztec diamond.
Let $G_{NE}^{\mathrm{diag}} (w,b,a,q)$ denote the two variable boundary generating function for the $q^{\mathrm{diag}}$ weighting of an Aztec diamond of size $n$, that is
\begin{equation}
    G_{NE}^{\mathrm{diag}} (w,b,a,q)=\sum_{i=0}^{n-1} \sum_{j=0}^{n-1} \frac{Z^\mathrm{diag}(i,j,a,q,n)}{Z_n^\mathrm{diag}(1,a,q)} w^i b^j
\end{equation}
For this generating function, the variable $w$ marks the white vertices and the variable $b$ marks the black vertices.

 We also let
\begin{equation}
      \overline{G}_{NE}^{\mathrm{diag}} (w,b,a,q,z)=\sum_{n=0}^\infty G_{NE}^{\mathrm{diag}} (w,b,a,q)z^n,
\end{equation}
which denotes the three variable boundary generating function for the $q^{\mathrm{diag}}$ weighting of the Aztec diamond. For this generating function, the variables $w$ and $b$ mark the white and black vertices respectively, while the variable $z$ marks the size of the Aztec diamond.

\begin{lemma} \label{lem:qvolboundarygeneratingfunction}
The boundary generating function for the $q^{\mathrm{diag}}$ weighting of an Aztec diamond of size $n$ is given by
\begin{equation}
      G_{NE}^{\mathrm{diag}} (w,b,a,q)= \sum_{i=0}^{n-1} \frac{a q^{i+1} }{1+q^{2i+1} a^2} \prod_{k=0}^i \frac{q (1+q^{2k+1} a b) (1+q^{2k+1} a w ) }{1+a^2 q^{2k+1}}
\end{equation} 
\end{lemma}

\begin{proof}
      Multiplying the recurrence relation given in Lemma \ref{lem:qvolrecurrence} by $w^i b^j z^n$ and summing over $n \geq 0$, $0\leq i \leq n-1$ and $0 \leq j \leq n-1$ we obtain

      \begin{equation} \label{qvolrecurrence1}
      \overline{G}_{NE}^{\mathrm{diag}} (w,b,a,q,z)=\frac{z q(1+a b q) (1+a w q)}{1+a^2 q} \overline{G}_{NE}^{\mathrm{diag}} (wq,bq,aq,q,z)+\frac{aq}{1+a^2q} \frac{z}{1-z} 
      \end{equation}
The above equation is a recurrence for the $q^{\mathrm{diag}}$ boundary generating function.  To solve this recurrence, 
we let $F(w,b,a)= \overline{G}_{NE}^{\mathrm{diag}} (w,b,a,q,z)$,
\begin{equation}
      R(w,b,a)=  \frac{z q(1+a b q) (1+a w q)}{1+a^2 q},
\end{equation}
and
\begin{equation}
      S(a )=\frac{aq}{1+a^2q} \frac{z}{1-z}. 
\end{equation}
This means, we  rewrite~\eqref{qvolrecurrence1} as
\begin{equation}\label{qvolrecurrence2}
	F(w,b,a)= R(w,b,a) F(qw,qb,qa)+S(a)
\end{equation}
By applying the recurrence relation~\eqref{qvolrecurrence2} iteratively we have
\begin{equation}
      F(w,b,a) =\prod_{i=0}^\infty R(w q^i,bq^i , a q^i) F(0,0,0) +\sum_{i=0}^\infty S(a q^i) \prod_{k=0}^{i-1} R(wq^i,bq^i,a q^i)  
\end{equation}
The first term on the right-hand side of the above equation goes to zero by choosing $z q<1$, and so we obtain
\begin{equation}
      \overline{G}_{NE}^{\mathrm{diag}} (w,b,a,q,z)=\sum_{i=0}^\infty S(a q^i) \prod_{k=0}^{i-1} R(wq^i,bq^i,a q^i).
\end{equation}
To obtain the lemma, we extract the $n^{th}$ coefficient of the above equation which is computed using the fact that
\begin{equation}
	\prod_{k=0}^{i-1} R(w q^i , bq^i,a q^i)= z^i \prod_{k=0}^{i-1} \frac{ q(1+a bq^{2k+1})(1+ a w q^{2k+1})}{(1+a^2 q^{2k+1})}
\end{equation}
and
\begin{equation}
	S(a q^i) = \frac{a q^{i+1}}{ 1+ a^2 q^{2i+1}} \sum_{n=1}^\infty z^n.
\end{equation}

\end{proof}

Let $Z_n^{\mathrm{col}}(a,q)$  denote the partition function of  the Aztec diamond with $q^{\mathrm{col}}$ and $Z^{\mathrm{col}}(i,j,a,q,n)$ denote the partition function of  the Aztec diamond with $q^{\mathrm{col}}$ weights with vertices $(2n-2i-1,2n)$ and $(2n,2n-1-2j)$ removed. 

\begin{lemma} \label{lem:qvolboundarygeneratingfunction2}
For $0 \leq i,j \leq n-1$, $n \in \mathbb{N}$ and $a, q >0$ we have
\begin{equation}
\frac{Z^{\mathrm{col}} (i,j,a,q,n)}{Z_n^{\mathrm{col}}(a,q)} = \frac{1}{(2\pi \mathtt{i})^2} \int_{\Gamma_0} \int_{\Gamma_0} q^{(i+1)^2+j} \frac{G_{NE}^{\mathrm{diag}} (w,b,a,q) }{w^i b^j} \frac{dw}{w} \frac{db}{b}
\end{equation}
\end{lemma}

\begin{proof}
 The gauge transformation from the proof of Lemma~\ref{lem:qvolgaugeequivalent} gives 
\begin{equation} \label{qvol:gaugetransforminlemma}
	q^{(i+1)^2+j} \frac{Z^{\mathrm{diag}}(i,j,a,q,n)}{Z_n^{\mathrm{diag}}(1,a,q)}=\frac{Z^{\mathrm{col}}(i,j,a,q,n)}{Z_n^{\mathrm{col}}(a,q)}.
\end{equation}
From Lemma~\ref{lem:qvolboundarygeneratingfunction}, the coefficient of $w^i b^j$ in  $G_{NE}^{\mathrm{diag}}(w,b,a,q)$ is equal to $Z_n^{\mathrm{diag}}(i,j,a,q,n)/Z_n^{\mathrm{diag}}(1,a,q)$ and this is given by
\begin{equation}
 \frac{1}{(2\pi \mathtt{i})^2} \int_{\Gamma_0} \int_{\Gamma_0}  \frac{G_{NE}^{\mathrm{diag}} (w,b,a,q) }{w^i b^j} \frac{dw}{w} \frac{db}{b}
\end{equation}

\end{proof}

\subsection{Finding $K_{\mathrm{col}}^{-1}$}
\label{subsection Finding K}

In this subsection, we explain how we guessed the formula for $K^{-1}_{\mathrm{col}}$ which appears in Theorem~\ref{thm:qvolinverseK}.  The proof of Theorem~\ref{thm:qvolinverseK} appears in the following subsection.  It is a correct but somewhat unilluminating argument, since we essentially demonstrate that the formula for $K^{-1}_{\mathrm{col}}$ satisfies the equation $K_{\mathrm{col}}\cdot K^{-1}_{\mathrm{col}} = I$;  the purpose of this section is to describe the heuristics behind the guess that we made.

The first step is to rewrite the formula given in Lemma~\ref{lem:qvolboundarygeneratingfunction2} as a double contour integral formula where the contours of integration are given by $\Gamma_0$ and $\Gamma_{1/a,q}^n$.  We found that 
\begin{equation}
\begin{split} \label{qvol:assertion1}
\frac{Z^{\mathrm{col}}(i,j,a,q,n)}{Z_n^{\mathrm{col}}(a,q)}&= \frac{q^{(2+i+j)^2-1}}{(2 \pi i)^2} \int_{\Gamma_0} \int_{\Gamma_{a,q}^n} dw \, dz \\
& \frac{ w^n \prod_{k=0}^{n-1} a q^{-2k-1+2(n-j-1)} +z}{(w-z)z^{n-i} \prod_{k=0}^{n-j-1} (a q^{2k-1} +w ) \prod_{k=0}^j (-1+ a w q^{2k+2}) }
\end{split}
\end{equation}
We guessed~\eqref{qvol:assertion1} using the following: 
\begin{itemize}
\item we knew that such a formula holds in the case when $q=1$ by comparing the formulas for the absolute value of the inverse Kasteleyn matrix on the boundary given in Section~\ref{section One periodic case} and~\cite{CJY:12} which writes the formula for the inverse Kasteleyn matrix as a double contour integral formula. We compared the absolute values due to different Kasteleyn orientations.  We also found a direct computation between these two formulas however, we were unable to extend this computation when $q \not =1$. 
\item From Lemma~\ref{lem:qvolboundarygeneratingfunction2}, we had an approximate structure of the formula: for example, each $q$ appears as $q^{2k+1}$ for some $k$.  
\item We guessed that the poles with respect to $w$ split when $q \not=1$ (because this occurs for the lozenge tiling case-- e.g. see Theorem 2.25 in~\cite{BF:08}).
\item We used  small examples and the above steps to guess formula~\eqref{qvol:assertion1}.
\end{itemize}
To prove~\eqref{qvol:assertion1} is correct, we setup a boundary recurrence relation for $Z^{\mathrm{col}}(i,j,a,q,n)/Z^{\mathrm{col}}_n(a,q)$ similar to Lemma~\ref{lem:qvolrecurrence} and showed that~\eqref{qvol:assertion1} satisfied the boundary recurrence and its initial condition.

We next found  values of $K^{-1}_{\mathrm{col}}$ by multiplying the sign from the Kasteleyn orientation to equation~\eqref{qvol:assertion1} and writing the formula in terms of the Kasteleyn coordinates.  Due to the $q^{\mathrm{col}}$ weighting having the same Kasteleyn orientation as the uniform case, the sign of the boundary values of the inverse Kasteleyn matrix are the same.  In the proof of Lemma~\ref{uniform:lem:sign}, we found that the sign of $K^{-1}((x_1,2n),(2n,y_2))$ is equal to $-\mathtt{i}^{i+j+1}$where $i=(2n-1-x_1)/2$ and $j=(2n-1-y_2)/2$.  Multiplying~\eqref{qvol:assertion1} by $-\mathtt{i}^{i+j+1}$ and setting $i=(2n-1-x_1)/2$ and $j=(2n-1-y_2)/2$, we obtain
\begin{equation}
\begin{split}
      K_{\mathrm{col}}^{-1}((x_1,2n),(2n,y_2)) &=\frac{\mathtt{i}^{(4n-x_1-y_2+2)/2} q^{(4+4n -x_1-y_2)(x_1+y_2-4n)/4}}{( 2\pi \mathtt{i})^2}  \int_{\Gamma_{a,q}^n}\int_{\Gamma_0} dz \, dw \\
      &\frac{w^n}{z^{(x_1+1)/2}(w-z)} \frac{\prod_{k=0}^{n-1} (z+ aq^{-2k+y_2-2})}{\prod_{k=0}^{(y_2-1)/2} (w+aq^{2k-1}) \prod_{k=0}^{n-(y_1+1)/2} (a w q^{2k+2} -1)}.  
      \end{split}
\end{equation}

We found $K^{-1}_{\mathrm{col}} (x,y)$ for all $x \in \mathtt{W}$ and $y \in \mathtt{B}$ by  treating the  entry-wise expansions of the matrix equations $K_{\mathrm{col}}\cdot K^{-1}_{\mathrm{col}} =\mathbbm{I}$ and $K_{\mathrm{col}}^{-1}\cdot K_{\mathrm{col}} =\mathbbm{I}$ as recurrence relations whose initial condition is given by $K^{-1}$ on the top and right boundaries which is given in the above equation.  Details on the reason why these two matrix equations give a recurrence relations can be found in Section~\ref{section Uniform case}.

\subsection{Proof of Theorem~\ref{thm:qvolinverseK}} \label{Appendixqvol}
In this subsection we prove Theorem~\ref{thm:qvolinverseK}.  As the computations are particularly messy, we used computer algebra to help with the simplifications.  We will use the following notations:
let $(a;q)_n=\prod_{i=0}^{n-1} (1-aq^i)$ be the $q$-Pochhammer symbol. For $x=(x_1,x_2)\in \mathtt{W}$ and $y=(y_1,y_2) \in \mathtt{B}$, we set
\begin{equation}
\begin{split}
&g_1(w,z,x,y)=
\mathtt{i} w^{\frac{1}{2} \left(y_1-y_2-1\right)} (-1)^{\frac{1}{4} \left(-x_1-x_2+y_1+y_2\right)} z^{\frac{1}{2} \left(-x_1+x_2-1\right)}\\
&\times q^{-\frac{1}{4} \left(x_1-x_2-y_1+y_2\right) \left(4 n-x_1+x_2-y_1-y_2+4\right)}  
\frac{ \left(-\frac{a q^{y_2-2}}{z};\frac{1}{q^2}\right){}_{\frac{x_2}{2}} \left(a q^{x_2-y_2+3} z;q^2\right){}_{n-\frac{x_2}{2}}}{(w-z) \left(-\frac{a}{q w};q^2\right){}_{\frac{1}{2} \left(y_2+1\right)} \left(a q^2 w;q^2\right){}_{n-\frac{y_2}{2}+\frac{1}{2}}},
\end{split}
\end{equation}
\begin{equation}
\begin{split}
&g_2(w,x,y)=-(-1)^{\frac{1}{4} \left(7 x_1+x_2-7 y_1-y_2\right)} \mathtt{i}^{-x_2+y_2+1} w^{\frac{1}{2} \left(-x_1+x_2+y_1-y_2-2\right)} \\
&\times q^{\frac{1}{4} \left(x_1-x_2-y_1+y_2\right) \left(-4 n+x_1-x_2+y_1+y_2-4\right)}
 \frac{\left(-\frac{a q^{y_2-2}}{w};\frac{1}{q^2}\right){}_{\frac{x_2}{2}} \left(a q^{x_2-y_2+3} w;q^2\right){}_{n-\frac{x_2}{2}}}{\left(-\frac{a}{q w};q^2\right){}_{\frac{1}{2} \left(y_2+1\right)} \left(a q^2 w;q^2\right){}_{n-\frac{y_2}{2}+\frac{1}{2}}}
\end{split}
\end{equation}
and
\begin{equation}
\begin{split} \label{Appendix:qvol:g2tilde}
&\tilde{g}_2(w,x,y)=- (-1)^{\frac{1}{4} \left(7 x_1+x_2-7 y_1-y_2\right)} \mathtt{i}^{-x_2+y_2+1} w^{\frac{1}{2} \left(-x_1+x_2+y_1-y_2-2\right)}\\
&\times q^{-\frac{1}{4} \left(x_1-x_2-y_1+y_2\right) \left(4 n-x_1+x_2-y_1+y_2+4\right)} 
 \frac{ \left(-\frac{a}{q^2 w};\frac{1}{q^2}\right){}_{\frac{x_2}{2}} \left(a q^{x_2+3} w;q^2\right){}_{n-\frac{x_2}{2}}}{\left(-\frac{a q^{-y_2-1}}{w};q^2\right){}_{\frac{1}{2} \left(y_2+1\right)} \left(a q^{y_2+2} w;q^2\right){}_{n-\frac{y_2}{2}+\frac{1}{2}}}
\end{split}
\end{equation}
Note that, by reversing the order of summation, we have  
\begin{equation}
\left(-\frac{a}{q^2 w} ; \frac{1}{q^2} \right)_{(y_2+1)/2} = \left(-\frac{a q^{-y_2-1}}{w} ; q^2\right)_{(y_2+1)/2}
\label{Appendix:qvol:rearrange}
\end{equation}
We have chosen $g_1$ and $g_2$ to be integrands of $f_1$ and $f_2$ and $\tilde{g}_2$ to be the integrand of $f_2$ under the change of variables $w \mapsto w q^{y_2}$.  This means that
\begin{equation}
	\frac{1}{(2\pi \mathtt{i})^2} \int_{\Gamma^n_{1/a,q}} \int_{\Gamma_0} g_1(w,z,x,y) dz \, dw = f_1(x,y),
\end{equation}
and
\begin{equation}
	\frac{1}{2\pi \mathtt{i}} \int_{\Gamma^n_{a^{-1} q^{-y_2},q}} \tilde{g}_2(w,x,y) dw= \frac{1}{2 \pi \mathtt{i}} \int_{\Gamma^n_{1/a,q}} g_2(w,x,y)  \, dw = f_2(x,y).
\end{equation}
 Note that we also have
\begin{equation} \label{Appendix:qvol:g1tog2}
	g_2(w,x,y)=-\lim_{z\to w} (w-z) g_1(w,z,x,y)
\end{equation}
for $x \in \mathtt{W}$ and $y \in \mathtt{B}$.

\begin{proof}
Throughout the proof, we assign $x=(x_1,x_2) \in \mathtt{B}$ and $y=(y_1,y_2)\in\mathtt{B}$.
To prove Theorem~\ref{thm:qvolinverseK}, we have to verify  the equation $K_{\mathrm{col}}\cdot L=\mathbbm{I}$ where we set
\begin{equation}
L(x,y)= \left\{ \begin{array}{ll}
	f_1(x,y) & x_1<y_1+1\\
	f_1(x,y)+f_2(x,y) & x_1 \geq y_1+1
	\end{array} \right.
\end{equation}
where $f_1(x,y)$ and $f_2(x,y)$ are as given in Theorem~\ref{thm:qvolinverseK} and $x=(x_1,x_2) \in \mathtt{W}$ and $y=(y_1,y_2) \in \mathtt{B}$.

We first expand  out $K_{\mathrm{col}}\cdot L $ entry-wise by   using the definition of $K_{\mathrm{col}}$ given in equation~\eqref{Kcol}.  By comparing with the identity matrix, we want to verify the equation
\begin{equation}\label{Appendix:qvol:KK1}
	(L(x+e_1,y)  + a \mathtt{i} q^{x_1-2n-1} L (x+e_2,y) ) \delta_{x_1<2n} +(L(x-e_1,y) + a \mathtt{i} q^{2n-x_1} L(x-e_2,y)) \delta_{x_1>0} =\delta_{x=y}
\end{equation}
where $x, y \in \mathtt{B}$ and the delta functions account for $x$ is on the left or right boundaries of the Aztec diamond.  There are three cases to consider for~\eqref{Appendix:qvol:KK1}, namely $x_1=0$, $0<x_1<2n$ and $x_1=2n$. 

For $0<x_1<2n$, from~\eqref{Appendix:qvol:KK1} we want to verify the equation 
\begin{equation} \label{Appendix:qvol:firstcase}
L(x+e_1,y) +L(x-e_1,y)+ a \mathtt{i} q^{x_1-2n-1} L(x+e_2,y) +a \mathrm{i}q^{2n-x_1} L(x-e_2,y)=\delta_{x=y}
\end{equation}
for $x=(x_1,x_2),y \in \mathtt{B}$. We first consider the term $g_1(w,z,x,y)$ when substituted into the left-hand side of the above equation.  We find that after some simplification
\begin{equation}
\begin{split} \label{Appendix:qvol:g1}
&g_1(w,z,x+e_1,y) + g_1(w,z,x-e_1,y)+a \mathtt{i} q^{x_1-2n-1} g_1(w,z,x+e_2,y) +a \mathrm{i}q^{2n-x_1} g_1(w,z,x-e_2,y)\\
&=\frac{(-1)^{\frac{1}{4} \left(-x_1-x_2+y_1+y_2\right)} z^{\frac{1}{2} \left(-x_1+x_2-3\right)} q^{\frac{1}{4} \left(-2 x_1 \left(2 n+x_2-y_2+2\right)+\left(x_2+y_1-y_2\right) \left(4 n+x_2-y_1-y_2\right)+x_1^2+4 y_1-8 y_2-4\right)}}
{w^{-\frac{1}{2} \left(y_1-y_2-1\right)} (w-z) \left(-\frac{a}{q w};q^2\right){}_{\frac{1}{2} \left(y_2+1\right)} \left(a q^2 w;q^2\right){}_{n-\frac{y_2}{2}+\frac{1}{2}}} \\
&\times
 \left( \frac{ \left(-\frac{a q^{y_2-2}}{z};\frac{1}{q^2}\right){}_{\frac{1}{2} \left(x_2+1\right)} \left(a q^{x_2-y_2+4} z;q^2\right){}_{n-\frac{x_2}{2}-\frac{1}{2}}}{ z^{-1} q^{-x_2-1}\left(q^{y_2}-a z q^{x_2+2}\right)^{-1}}- \frac{\left(-\frac{a q^{y_2-2}}{z};\frac{1}{q^2}\right){}_{\frac{1}{2} \left(x_2-1\right)} \left(a q^{x_2-y_2+2} z;q^2\right){}_{n-\frac{x_2}{2}+\frac{1}{2}}}{q^{-y_2} \left(a q^{y_2}+z q^{x_2+1}\right)^{-1}} \right)\\
&=0
\end{split}
\end{equation} 
where the first equality follows from computer algebra simplification and the second equality follows because we are able to factor the $q$-Pochhammer symbols in the third line, that is, the third line in the above equation is equal to 
\begin{equation}
\begin{split}
  &\left(-\frac{a q^{y_2-2}}{z};\frac{1}{q^2}\right){}_{\frac{1}{2} \left(x_2-1\right)} \left(a q^{x_2-y_2+4} z;q^2\right){}_{n-\frac{x_2}{2}-\frac{1}{2}}
\\&\times \left(
q^{1+x_2} \left(1+\frac{a q^{-1-x_2+y_2}}{z}\right) z \left(q^{y_2}-a q^{2+x_2} z\right)-q^{y_2} \left(a q^{y_2}+q^{1+x_2} z\right) \left(1-a q^{2+x_2-y_2} z\right)\right)=0
\end{split}
\end{equation}
It follows from~\eqref{Appendix:qvol:g1} that
\begin{equation} \label{Appendix:qvol:f1case1}
f_1(x+e_1,y) +f_1(x-e_1,y)+ a \mathtt{i} q^{x_1-2n-1} f_1(x+e_2,y) +a \mathrm{i}q^{2n-x_1} f_1(x-e_2,y)= 0
\end{equation}
for $x=(x_1,x_2), y \in \mathtt{B}$ and $0<x_1<2n$.  Due to the method of computation, we  note that the above equation holds for any black vertices $x,y$, including vertices outside the Aztec diamond.

To substitute the term $g_2$ into the left-hand side of~\eqref{Appendix:qvol:firstcase}, we have to consider $x_1=y_1$ and $x_1 \geq y_1+2$ for $x,y \in \mathtt{B}$ separately due to the split definition of $L$.  For $x_1 \geq y_1+2$, all four terms of $g_2$ are present in the left-hand side of~\eqref{Appendix:qvol:firstcase} and using~\eqref{Appendix:qvol:g1tog2} and~\eqref{Appendix:qvol:g1}, it follows that for $x_1 \geq y_1+2$ and $x,y \in \mathtt{B}$
\begin{equation}\label{Appendix:qvol:f2case1}
f_2(x+e_1,y)+f_2(x-e_1,y)+a \mathtt{i} q^{x_1-2n-1} f_2(x+e_2,y) + a \mathtt{i} q^{2n-x_1}f_2(x-e_2,y)=0,
\end{equation}
whereas, for $x_1=y_1$, we have that the left-hand side of~\eqref{Appendix:qvol:firstcase} reads
\begin{equation}\label{Appendix:qvol:f2first2}
f_2((x_1+1,x_2+1),(y_1,y_2))+a \mathtt{i}q^{2n-x_1} f_2((x_1+1,x_2-1),(y_1,y_2)).
\end{equation}
Using~\eqref{Appendix:qvol:g2tilde}, we  write the integrand of the above equation using $\tilde{g}_2$.  We find that this is given by
\begin{equation}
\begin{split} \label{Appendix:qvol:g2tilde1}
&\frac{(-1)^{\frac{1}{4}  \left(y_2-x_2\right)} w^{\frac{1}{2} \left(x_2-y_2-4\right)} q^{\frac{1}{4} \left(y_2 \left(-4 n+2 x_1-4\right)+x_2 \left(4 n-2 x_1\right)+x_2^2-y_2^2-4\right)}}
{\left(-\frac{a q^{-y_2-1}}{w};q^2\right){}_{\frac{1}{2} \left(y_2+1\right)} \left(a q^{y_2+2} w;q^2\right){}_{n-\frac{y_2}{2}+\frac{1}{2}}}\\
&
 \left.\left(a \left(-\frac{a}{q^2 w};\frac{1}{q^2}\right){}_{ \frac{x_2-1}{2}} \left(a q^{x_2+2} w;q^2\right){}_{\frac{2n-x_2+1}{2}}-w q^{x_2+1} \left(-\frac{a}{q^2 w};\frac{1}{q^2}\right){}_{\frac{x_2+1}{2}} \left(a q^{x_2+4} w;q^2\right){}_{\frac{n-x_2-1}{2}}\right)\right).
\end{split}
\end{equation} 
We  simplify the second line of the above equation by expanding the $q$-Pochhammer symbols and factorizing.  This is given by
\begin{equation}
-q^{x_2+1}(1+a^2 q)w \left(-\frac{a}{q^2 w};\frac{1}{q^2}\right){}_{\frac{1}{2} \left(x_2-1\right)} \left(a q^{x_2+4}w;q^2\right){}_{n-\frac{x_2}{2}-\frac{1}{2}}.
\end{equation}
Using the above simplification of~\eqref{Appendix:qvol:g2tilde1},~\eqref{Appendix:qvol:f2first2} is equal to 
\begin{equation}~\label{Appendix:qvol:f2first3}
\begin{split}
& \frac{1}{2\pi \mathtt{i}} \int_{\Gamma^n_{1/(a q^{y_2}),q}} dw \, \frac{(-1)^{\frac{1}{4}  \left(y_2-x_2\right)} w^{\frac{1}{2} \left(x_2-y_2-2\right)} q^{\frac{1}{4} \left(y_2 \left(-4 n+2 x_1-4\right)+x_2 \left(4 n-2 x_1\right)+x_2^2-y_2^2-4\right)}}
{\left(-\frac{a q^{-y_2-1}}{w};q^2\right){}_{\frac{1}{2} \left(y_2+1\right)} \left(a q^{y_2+2} w;q^2\right){}_{n-\frac{y_2}{2}+\frac{1}{2}}}\\
&\times
\left(-q^{x_2+1}(1+a^2 q) \left(-\frac{a}{q^2 w};\frac{1}{q^2}\right){}_{\frac{1}{2} \left(x_2-1\right)} \left(a q^{x_2+4}w;q^2\right){}_{n-\frac{x_2}{2}-\frac{1}{2}}\right)
\end{split}
\end{equation}
We  now compute the above integral for different values of $x_2$ and $y_2$.  First, notice that for $x_2=y_2$,~\eqref{Appendix:qvol:f2first3} simplifies to
\begin{equation}
 \frac{1}{2\pi \mathtt{i}} \int_{\Gamma^n_{1/(a q^{y_2}),q}} dw \, \frac{-(1+a^2 q)}{(a q^{-1-y_2}+w)(1-a q^{2+y_2} w)}=1
\end{equation}
To compute~\eqref{Appendix:qvol:f2first3} for other choices of $x_2$ and $y_2$, we write out the parts of the integrand which contain $w$.  This is given by
\begin{equation}
\begin{split}
&w^{\frac{1}{2} \left(x_2-y_2-2\right)}\frac{ \left(-\frac{a}{q^2 w};\frac{1}{q^2}\right){}_{\frac{1}{2} \left(x_2-1\right)} \left(a q^{x_2+4}w;q^2\right){}_{n-\frac{x_2}{2}-\frac{1}{2}} }{\left(-\frac{a q^{-y_2-1}}{w};q^2\right){}_{\frac{1}{2} \left(y_2+1\right)} \left(a q^{y_2+2} w;q^2\right){}_{n-\frac{y_2}{2}+\frac{1}{2}}}\\
&=\frac{ \prod_{k=0}^{\frac{x_2-3}{2}} w+a q^{-2k-2} \prod_{k=0}^{n-\frac{x_2}{2}-\frac{3}{2}}1+aq^{x_2+4+2k}w}{ \prod_{k=0}^{\frac{y_2-1}{2}} w+a q^{-2k-2} \prod_{k=0}^{n-\frac{y_2}{2}-\frac{1}{2}}1+aq^{y_2+2+2k}w}
\end{split}
\end{equation}
where we have used~\eqref{Appendix:qvol:rearrange} to rearrange a $q-$Pochhammer symbol.  
For $x_2<y_2$ which means that $x_2 \leq y_2-2$, there are no poles in above equation for $w$ inside the contour $\Gamma^n_{1/(aq^{y_2}),q}$ and so~\eqref{Appendix:qvol:f2first3} is equal to zero.  
For $x_2>y_2$,  the degree of the numerator of the above equation in $w$ is $n-1$ while the degree of the denominator is $n+1$.  Therefore,  we move the contour of integration of the integral in~\eqref{Appendix:qvol:f2first3} through infinity so that it now surrounds $\{-a q^{-2},-a q^{-4},\dots, -a q^{-y_2-1}\}$.  But from the above equation, the integrand of~\eqref{Appendix:qvol:f2first3} has no poles at these points for $x_2>y_2$ which means that ~\eqref{Appendix:qvol:f2first3} is equal to zero.
This  means we  conclude
\begin{equation} \label{Appendix:qvol:f2first4}
f_2((x_1+1,x_2+1),(x_1,y_2))+a \mathtt{i}q^{2n-x_1} f_2((x_1+1,x_2-1),(x_1,y_2))=\delta_{x_2=y_2}
\end{equation}
for $x=(x_1,x_2), y=(y_1,y_2) \in \mathtt{B}$. By our method of computation, the above relation holds for $0\leq x_1 \leq 2n$.
It follows from~\eqref{Appendix:qvol:f1case1},~\eqref{Appendix:qvol:f2case1} and~\eqref{Appendix:qvol:f2first4} that we have verified~\eqref{Appendix:qvol:firstcase}.

For $x_1=0$ from~\eqref{Appendix:qvol:KK1} we want to verify 
\begin{equation}\label{Appendix:qvol:secondcase}
L(x+e_1,y) + a \mathtt{i}q^{2n} L(x-e_2,y)=\delta_{x=y}
\end{equation}
for $x=(0,x_2),y \in \mathtt{B}$.  Due to the split definition of $L$, when we substitute the term $f_2$ into the left-hand side of~\eqref{Appendix:qvol:secondcase}, we are required to have $y=(0,y_2)$.  Using  ~\eqref{Appendix:qvol:f2first4} with $x_1=0$, we obtain
\begin{equation}\label{Appendix:qvol:f2second}
f_2((1,x_2+1),(y_1,y_2))+a \mathtt{i}q^{2n} f_2((1,x_2-1),(y_1,y_2))=\delta_{x=y}.
\end{equation}
When we substitute term $f_1$ into the left-hand side of~\eqref{Appendix:qvol:secondcase}, notice that we have
\begin{equation}
f_1((-1,w_2),(y_1,y_2))=0
\end{equation}
for $w_2 \mod2=0$ because there is no residue at $z=0$  in~\eqref{qvol:f1} and as~\eqref{Appendix:qvol:f1case1} holds for any value of $x_1$, we  conclude that
\begin{equation}\label{Appendix:qvol:f1second}
f_1((x_1+1,x_2+1),(x_1,y_2))+a \mathtt{i}q^{2n} f_1((x_1+1,x_2-1),(x_1,y_2))=0.
\end{equation}
It follows from~\eqref{Appendix:qvol:f2second} and~\eqref{Appendix:qvol:f1second} that we have verified~\eqref{Appendix:qvol:secondcase}.

For $x_1=2n$, from~\eqref{Appendix:qvol:KK1} we want to verify 
\begin{equation}\label{Appendix:qvol:thirdcase}
L(x-e_1,y) + a \mathtt{i}q^{-1} L(x+e_2,y)=\delta_{x=y}
\end{equation}
We first consider substituting in $f_1$ into the left-hand side above equation, which means we consider
\begin{equation}
g_1(w,z,x-e_1,y) + a \mathtt{i}q^{-1} g_1(w,z,x+e_2,y).
\end{equation}
for $x=(2n,x_2)$.  The above expression is equal to  
\begin{equation}
\begin{split}
&\frac{i^{-n} w^{\frac{y_1-y_2-1}{2}} (-1)^{\frac{-x_2+y_1+y_2}{4}} q^{-\left(n-\frac{y_1}{2}+1\right){}^2-1} z^{\frac{1}{2} \left(-2 n+x_2-1\right)} }{(w-z) \left(-\frac{a}{q w};q^2\right){}_{\frac{1}{2} \left(y_2+1\right)} \left(a q^2 w;q^2\right){}_{n-\frac{y_2}{2}+\frac{1}{2}}} \left(-q^{\frac{1}{4} \left(x_2-y_2+2\right){}^2+1} \left(-\frac{a q^{y_2-2}}{z};\frac{1}{q^2}\right){}_{\frac{x_2-1}{2}} 
\right. \\
&\left. \left(a q^{x_2-y_2+2} z;q^2\right){}_{\frac{n-x_2+1}{2}}-a z q^{\frac{1}{4} \left(x_2-y_2+4\right){}^2} \left(-\frac{a q^{y_2-2}}{z};\frac{1}{q^2}\right){}_{\frac{x_2+1}{2}} \left(a q^{x_2-y_2+4} z;q^2\right){}_{\frac{n-x_2-1}{2}}\right)
\end{split}
\end{equation}
where we used computer algebra to make the simplifications.
In the above expression, we select the terms that only involve $z$ and factor them by collecting the appropriate $q$-Pochhammer symbols which gives
\begin{equation}
\begin{split}\label{Appendix:qvol:rightsimplify}
& \frac{z^{\frac{x_2-2n-1}{2}}}{(w-z)}\left(-\frac{a q^{y_2-2}}{z};\frac{1}{q^2}\right){}_{\frac{x_2-1}{2}} \left(a q^{x_2-y_2+4} z;q^2\right){}_{\frac{n-x_2-1}{2}}\\
&\times \left(
-a q^{\frac{1}{4} \left(4+x_2-y_2\right){}^2} \left(1+\frac{a q^{-1-x_2+y_2}}{z}\right) z-q^{1+\frac{1}{4} \left(2+x_2-y_2\right){}^2} \left(1-a q^{2+x_2-y_2} z\right)
\right)\\
&=  -q^{\frac{1}{4} \left(8+x_2^2-2 x_2 \left(-2+y_2\right)-4 y_2+y_2^2\right)} \left(1+a^2 q\right)\frac{
\prod_{k=0}^{\frac{x_2-3}{2}} z+a q^{y_2-2-2k} \prod_{k=0}^{\frac{2n-x_2-3}{2}} 1+a q^{4+x_2-y_2+2k} z}{z^n(w-z)}.
\end{split}
\end{equation}
This means 
\begin{equation}\frac{1}{2 \pi \mathtt{i}} \int_{\Gamma_0} g_1(w,z,x-e_1,y) + a \mathtt{i}q^{-1} g_1(w,z,x+e_2,y) dz=-g_2(w,z,x-e_1,y) - a \mathtt{i}q^{-1} g_2(w,z,x+e_2,y) \label{Appendix:qvol:rightg1tog2}
\end{equation}
which is be seen by pushing the contour through infinity which picks up a residue at $z=w$ because from~\eqref{Appendix:qvol:rightsimplify} (with respect to $z$),  the integrand in~\eqref{Appendix:qvol:rightg1tog2}  is a polynomial of degree $n-1$ divided by a  polynomial of degree $n+1$ for $x=(2n,x_2)$, and using~\eqref{Appendix:qvol:g1tog2} to evaluate the integral.  From~\eqref{Appendix:qvol:rightg1tog2}, we compute the integral with respect to $w$ over the contour $\Gamma^n_{1/a,q}$ to find for $x=(2n,x_2)$
\begin{equation}
\begin{split} \label{Appendix:qvol:f1tof2}
f_1(x-e_1,y)+a \mathtt{i}q^{-1}f_1(x+e_2,y) = - f_2(x-e_1,y)-a \mathtt{i}q^{-1}f_2(x+e_2,y) 
\end{split}
\end{equation}
By the definition of $L$, for $2n=x_1 \geq y_1+2$ (i.e. $y_1 \leq 2n-2$) and $x=(2n,x_2),y=(y_1,y_2)\in \mathtt{B}$ and using the above equation we obtain
\begin{equation} \label{Appendix:qvol:case3eqn1}
\begin{split}
	L(x-e_1,y)+a \mathtt{i}q^{-1} L(x+e_2,y)&=f_1(x-e_1,y)+f_2(x-e_1,y) \\&+a \mathtt{i}q^{-1} (f_1(x+e_2,y)+f_2(x+e_2,y))=0
\end{split}
\end{equation}
For $x_1=y_1=2n$, we have using~\eqref{Appendix:qvol:f1tof2},~\eqref{Appendix:qvol:f2case1} and~\eqref{Appendix:qvol:f2first4}
\begin{equation} \label{Appendix:qvol:case3eqn2}
\begin{split}
	L(x-e_1,y)+a \mathtt{i}q^{-1} L(x+e_2,y)&=f_1(x-e_1,y) +a \mathtt{i} q^{-1}f_1(x+e_2,y)\\
&=-f_2(x-e_1,y) -a \mathtt{i} q^{-1}f_2(x+e_2,y)\\
&=f_2(x+e_2,y)+a \mathtt{i} f_2(x-e_2,y)\\
&=\delta_{x=y}
\end{split}
\end{equation}
for $x=(2n,x_2),y\in \mathtt{B}$.   Equations~\eqref{Appendix:qvol:case3eqn1} and~\eqref{Appendix:qvol:case3eqn2} mean that we have verified~\eqref{Appendix:qvol:thirdcase}.

As we have verified~\eqref{Appendix:qvol:firstcase} for $0\leq x_1 \le 2n$,~\eqref{Appendix:qvol:secondcase} for $x_1=0$ and~\eqref{Appendix:qvol:thirdcase} for $x_1=2n$, we have verified~\eqref{Appendix:qvol:KK1}.

\end{proof}

\section{Two-Periodic Weighting} \label{section Diablo Tilings on the Fortress}

In this section, we compute the generating function of the inverse Kasteleyn matrix for the two-periodic weighting of an Aztec diamond of size $4m$.  This is much akin to the parameter $a$ introduced in Section~\ref{section One periodic case}, except now there are two parameters, $a$ and $b$, which decorate the edges of the Aztec diamond in a checkerboard fashion described below.  As remarked in the introduction, one of the special cases of this model is equivalent to a different, uniform tiling problem: namely, the so-called diabolo tilings of the fortress introduced in~\cite{Pro:03}.  The dual graph of the fortress is the \emph{square-octagon lattice} with certain boundary conditions.  This model was the main motivation for this work. 

Because of this periodicity, the problem becomes complicated in two ways: the recurrence relation increases in \emph{order}, and the generating function becomes a \emph{vector} of generating functions.  Finally, the relation which we solve in order to compute the boundary generating function involves a certain matrix multiplication which must be explicitly diagonalized before the recurrence can be solved.

We begin, as before, by explicitly writing down the Kasteleyn matrix $K$ to invert.  In fact, we will have two Kasteleyn matrices, $K$ (for even-order diamonds) and $\tilde{K}$ (for odd-order diamonds); for brevity, we will only invert $K$, and that only for Aztec diamonds of order $4m$.  But we will compute \emph{boundary generating functions} for both $K$ and $\tilde{K}$.  These are generating functions which give those entries of the inverses of $K$, $\tilde{K}$ corresponding to two white and black vertices on the boundary of the Aztec diamond.  The boundary generating functions appear in Lemma~\ref{fortress tilings:lemma:K1boundary}.

The matrix $K$ has rows indexed by black vertices and columns indexed by white vertices.  Due to the periodicity of the weights, we have two types of white and black vertices.  We denote for $i\in \{0,1\}$
\begin{equation}
    \mathtt{W}_i= \{ (x_1,x_2): (x_1 +x_2) \mod 4 = 2i+1, (x_1,x_2) \in \mathtt{W} \}
\end{equation}
and
\begin{equation}
    \mathtt{B}_i= \{ (x_1,x_2):( x_1 +x_2)\mod4= 2i+1, (x_1,x_2) \in \mathtt{B} \}.
\end{equation}

For an Aztec diamond of size $n$, we give weights $a$ and $b$ to the Aztec diamond in the following way: if the size of the Aztec diamond is even, i.e. $n=2r$, then the edge weights around each face with center $(2i+1,2j+1)$ for $0\leq i,j \leq n-1$ are given weight $a$ if $(i+j) \mod 2=0$ and weight $b$ if $(i+j) \mod2=1$.  Conversely, if $n=2r-1$, then the edge weights are obtained from embedding the diamond in an Aztec diamond of order $2r$.  Figure~\ref{fig:diabolo weights} shows this choice of edge weights.

\begin{figure}
\caption{ The two-periodic weighting of the Aztec diamond with the picture on the left being an even ordered Aztec diamond and the picture on the right being an odd ordered Aztec diamond. The edges surrounding each face are given the same weight, denoted by the parameter in the center of each face with the exception of the boundary of the odd ordered Aztec diamond, where the edge weights  are given by the adjacent parameter.  }
\includegraphics[height=3in]{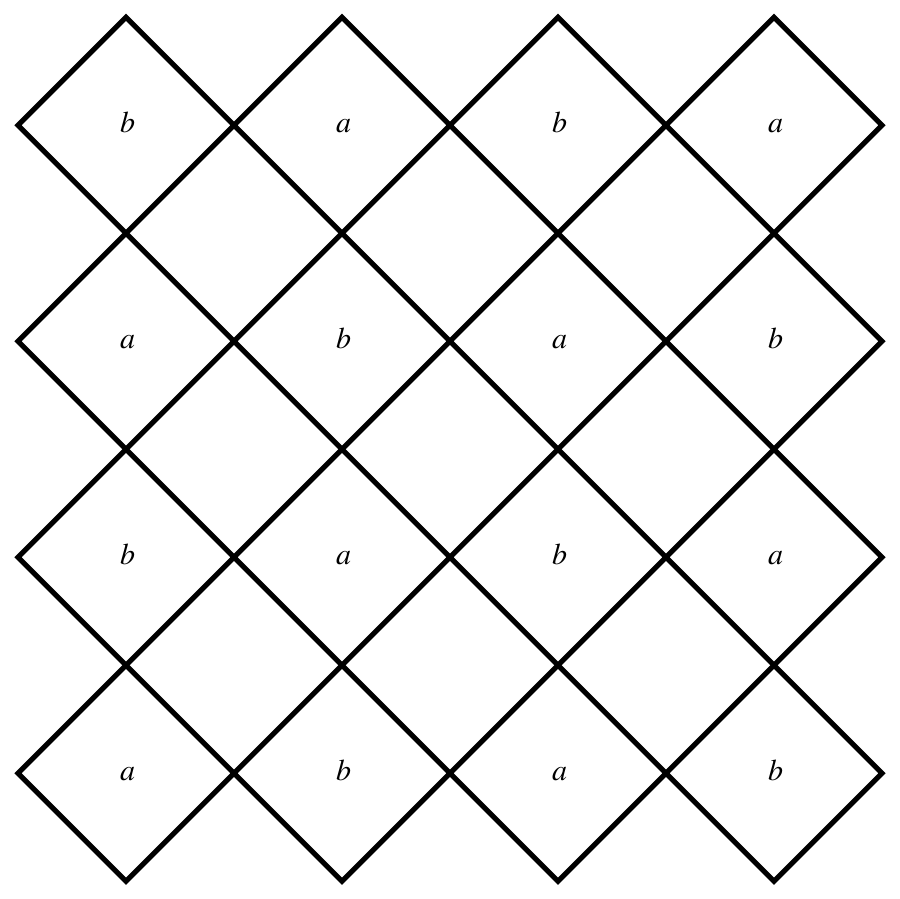}
\includegraphics[height=3in]{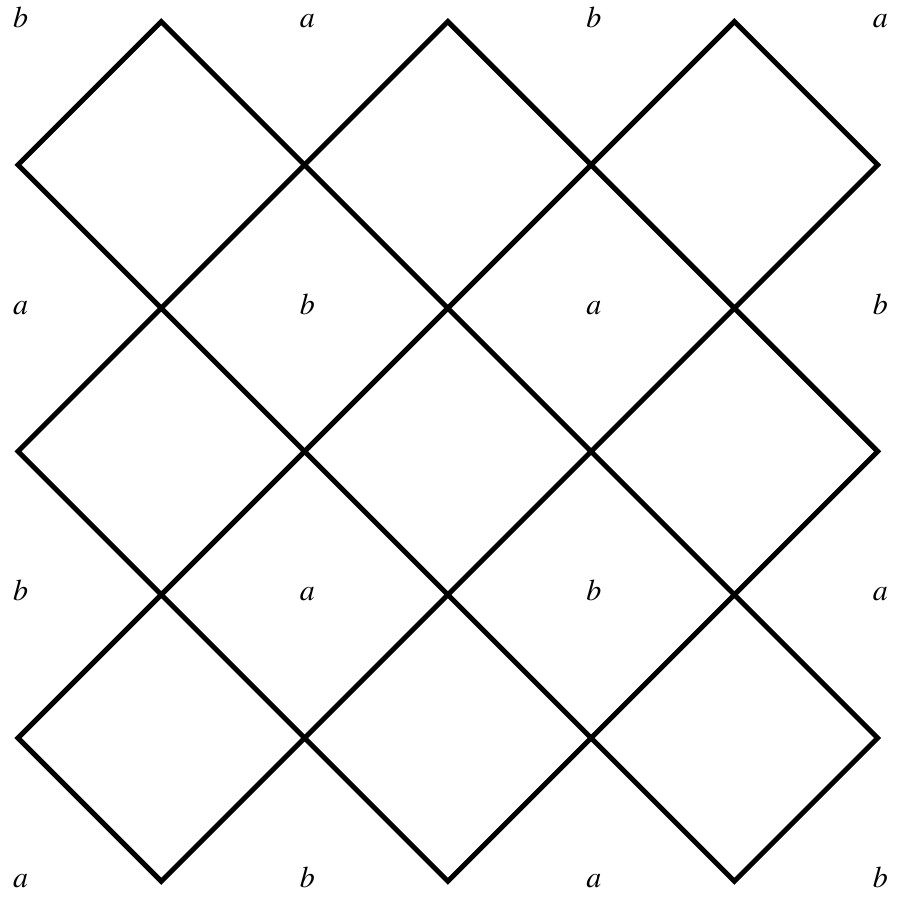}
\label{fig:diabolo weights}
\end{figure}

Let $K$ denote the Kasteleyn matrix for an Aztec diamond when $n$ is even and $\tilde{K}$ denote the Aztec diamond when $n$ is odd.  We have
\begin{equation} \label{fortress tilings:K}
      K(x,y)=\left\{ \begin{array}{ll}
                     a (1-i) + b i  & \mbox{if } y=x+e_1, x \in \mathtt{B}_i \\
                     (a i +b (1-i) ) \mathtt{i} & \mbox{if } y=x+e_2, x \in \mathtt{B}_i\\
                     a i + b (1-i)  & \mbox{if } y=x-e_1, x \in \mathtt{B}_i \\
                     (a (1-i) +b i ) \mathtt{i} & \mbox{if } y=x-e_2, x \in \mathtt{B}_i
                     \end{array} \right.
\end{equation}
and
\begin{equation}  \label{fortress tilings:Ktilde}
      \tilde{K}(x,y)=\left\{ \begin{array}{ll}
                     a i + b (1-i)  & \mbox{if } y=x+e_1, x \in \mathtt{B}_i \\
                     (a i +b (1-i) ) \mathtt{i} & \mbox{if } y=x+e_2, x \in \mathtt{B}_i\\
                     a(1-i) + b i  & \mbox{if } y=x-e_1, x \in \mathtt{B}_i \\
                     (a (1-i) +b i ) \mathtt{i} & \mbox{if } y=x-e_2, x \in \mathtt{B}_i
                     \end{array} \right.
\end{equation}
for $x \in \mathtt{B}$ and $y \in \mathtt{W}$.  

We now give the entries of $K^{-1}$ for white and black vertices on the bottom and left boundaries respectively for $K$  defined in~\eqref{fortress tilings:K} and the size of the Aztec diamond is even. To shorten the length of the formulas,  we will write $[i]_2=i \mod 2$.
%We show that we can give the generating function for $K^{-1}$, where $K$ is defined in~\eqref{fortress tilings:K} with $b=1$.  We first introduce a contour integral which allows us to determine  the boundary values of $K^{-1}$ where $K$ is defined in~\eqref{fortress tilings:K} (with general $b$). 

\begin{lemma} \label{fortress tilings:lemma:K1boundary}
For an Aztec diamond of size $n$, let $n=4m$ and let $K$ denote the Kasteleyn matrix given in~\eqref{fortress tilings:K}.  Let $L(a,b,i,j)$ denote $K^{-1}((2i+1,0),(0,2j+1))$ for the Kasteleyn matrix with parameters $a$ and $b$.  We have 
 \begin{equation}
 \begin{split}
  L(a,b,i,j)&=  \frac{- \mathtt{i}^{i+j+1}}{(2\pi \mathtt{i})^2} \int_{|z|=1} \int_{|w|=1} \\
  &\sum_{r=0}^{m-1} \sum_{k,l\in \{0,1\} } \frac{ g^{2k+l+1}_{2[i]_2+[j]_2+1}(a,b,w,z)  \alpha_k^r(a,b,w) \alpha_l^r(a,b,z)}{ w^{\lfloor i/2 \rfloor +1} z^{\lfloor j/2 \rfloor+1}} dw \, dz 
 \end{split}
 \end{equation}
where for $1 \leq i,j\leq 4$ we have $g_i^j(a,b,w,z)=\mathbf{N}_{i,j} (a,b,w,z)$ which is given in Appendix~\ref{appendix} and
  \begin{equation} \label{fortress tilings: alpha}
      \alpha_k^r(a,b,w) =  \frac{(\beta_0(a,b,w))^{2r} +(-1)^k ( \beta_1(a,b,w))^{2r}}{4 \sqrt{ ab}(a^2+b^2)(\sqrt{w} \sqrt{(b^4 + a^4) w + a^2 b^2 (1 + w^2)})^k} 
  \end{equation}
  for $k \in \{0,1\}$ with 
  \begin{equation} \label{fortress tilings: beta}
  \beta_l(a,b,w) = \frac{  \left( (a^2 + b^2) \sqrt{w} -(-1)^l  \sqrt{(b^4 + a^4) w + a^2 b^2 (1 + w^2)}\right)}{\sqrt{2ab(a^2+b^2)}}
   \end{equation}
\end{lemma}

Note that the expressions $\alpha_0(a,b,w)$ and $\alpha_1(a,b,w)$ are polynomials in $w$ and the $L(a,b,i,j)$ is a rational function in $a$ and $b$.   

It is clearly also possible to find the boundary entries of the inverse Kasteleyn matrix for the size of the Aztec diamond equal to $4m-1$, $4m-2$ and $4m-3$, but we will not do so here.
The proof of the lemma is given in the next subsection.  For the generating function of $K^{-1}$, we need the following terms:
\begin{equation}\begin{split}
        c_{\partial}(w_1,w_2)&=2(1+a^2) +a (w_1^2+w_1^{-2}) (w_2^2+w_2^{-2}),\\
     s_{i,0}(w_1,w_2)&=-a \left(w_1^{-2} w_2^{-2} + w_1^2 w_2^{-2} \right) -a \mathtt{i} w_1^2 +a \mathtt{i} w_1^{-2} -2 a^{2i}, \\
     s_{i,2n}(w_1,w_2)&= w_2^{2n} \left(-a \left( w_1^2 w_2^2 +w_1^{-2} w_2^2 \right) +a \mathtt{i} w_1^{2} - a \mathtt{i} w_1^{-2} -2a^{2(1-i)} \right).
     \end{split} 
\end{equation}
for $i \in \{0,1\}$. We set $f_n(x)=(1-x^{n})/(1-x)$ -- the sum of a geometric series and we also let for $n=2r$ and $\mathtt{w}=(w_1,w_2)$ and $\mathtt{b}=(b_1,b_2)$,
\begin{equation}
\begin{split}
 &d(\mathtt{w},\mathtt{b})=  f_{r}(w_2^4b_2^4)f_{r}(w_1^4b_1^4) b_2 w_1 (a+w_1^2 b_1^2+w_2^2b_2^2(1+a w_1^2b_1^2)) (w_2^2+b_1^2-\mathtt{i}(1+w_2^2b_1^2)) \\
 &+\left(\left( 1-\mathtt{i} w_2^2 \right) w_1^{-1}  b_2(1+a w_2^2 b_2^2)+\left( w_2^2-\mathtt{i} \right)  w_1^{2n+1}  b_1^{2n}b_2 (a+ w_2^2 b_2^2)\right) f_{r}(w_2^4 b_2^4)\\
 & -\frac{d_{\mathrm{sides}}(\mathtt{w},\mathtt{b})}{c_{\partial}(b_1,b_2)} f_{r}(w_1^4 b_1^4)
 \end{split}
 \end{equation}
where
\begin{equation}
 \begin{split} \label{dsides}
d_{\mathrm{sides}} (\mathtt{w},\mathtt{b})&=  s_{0,0}(w_1,w_2)  \left(a(  b_1^2 - \mathtt{i})  w_1 b_2 + w_1  b_2^{-1}(1-\mathtt{i} b_1^2) \right)\\
&+s_{1,0}(w_1,w_2)  \left(  ( b_1^2 - \mathtt{i} ) w_1^3 b_1^2 b_2 + a w_1^3 b_1^2 b_2^{-1} (1-\mathtt{i} b_1^2 ) \right) \\ 
&+ s_{0,2n}(w_1,w_2) \left( (1-\mathtt{i} b_1^2) w_1b_2^{2n-1} +a w_1  b_2^{2n+1} (b_1^2-\mathtt{i})\right)\\
&+s_{1,2n}(w_1,w_2) \left( a(1 -\mathtt{i} b_1^2) w_1^3 b_1^2 b_2^{2n-1} +  w_1^3 b_1^2 b_2^{2n+1} (b_1^2-\mathtt{i}) \right).
 \end{split}
\end{equation}

We denote the generating function of $K^{-1}$ for $K$ defined in~\eqref{fortress tilings:K} as
\begin{equation}
      G(a,b,\mathtt{w},\mathtt{b}) = \sum_{x \in \mathtt{W}} \sum_{y \in \mathtt{B}} K^{-1} (x,y) \mathtt{w}^x \mathtt{b}^y
\end{equation}
where $ \mathtt{w}^x=w_1^{x_1} w_2^{x_2}$ and $\mathtt{b}^y=b_1^{y_1} b_2^{y_2}$ for $x=(x_1,x_2)$ and $y=(y_1,y_2)$.

\begin{thma} \label{fortress tilings:mainthm}
For an Aztec diamond of size $n$, the generating function of $K^{-1}$ for $K$ defined in~\eqref{fortress tilings:K} with $b=1$, $n=4m$ for $m\in\mathbb{N}$, $\mathtt{w}=(w_1,w_2)$ and $\mathtt{b}=(b_1,b_2)$ is given by
\begin{equation}
 \begin{split}\label{fortress tilings:mainthmeqn1}
     & G(a,1,\mathtt{w},\mathtt{b})=\frac{d(\mathtt{w},\mathtt{b})}{c_{\partial}(w_1,w_2)} \\& +       \left( \sum_{i,j\in \{0,1\}} \sum_{k,l \in \{0,2n\}} \sum_{(x_1,k) \in \mathtt{W}_i} \sum_{(l,y_2) \in \mathtt{B}_j}\frac{ s_{i,k}(w_1,w_2) s_{j,l}(b_2,b_1) K^{-1} ((x_1,k),(l,y_2))}{c_{\partial}(w_1,w_2)c_{\partial}(b_1,b_2)} w_1^{x_1} b_2^{y_2}\right) \\
    \end{split}
\end{equation}
where
\begin{equation}
\begin{split} \label{fortress tilings:mainthmeqn2}
      K^{-1}((x_1,0),(0,y_2)) &= L\left(a,1,\frac{x_1-1}{2},\frac{y_2-1}{2} \right), \\
      K^{-1}((x_1,0),(2n,y_2)) &=\mathtt{i}^{2n-1-x_1+y_2} L\left(1,a,n-\frac{x_1+1}{2},\frac{y_2-1}{2} \right), \\
      K^{-1}((x_1,2n),(0,y_2)) &= \mathtt{i}^{2n-1+x_1-y_2} L\left(1,a,\frac{x_1-1}{2},n-\frac{y_2+1}{2} \right)
 \hspace{2mm} \mathrm{and} \\
      K^{-1}((x_1,2n),(2n,y_2)) &= L\left(a,1,n-\frac{x_1+1}{2},n-\frac{y_2+1}{2} \right).
\end{split}
\end{equation}
and $L(a,b,i,j)$ is given in Lemma~\ref{fortress tilings:lemma:K1boundary}.

\end{thma}

For the proof of this theorem, we are not able to follow the approach exactly as given in Theorem~\ref{thm:uniform} because the relations $K.K^{-1}=\mathbbm{I}$ and $K^{-1}. K=\mathbbm{I}$, while technically sufficient, are not of a suitably nice form to allow any progress.  Instead, we need to use two further recurrences $K^*. K.K^{-1}= K^* \mathbbm{I}$ and $K^{-1}.K.K^*=\mathbbm{I}.K^{*}$ where $K^*$ denotes complex conjugate transpose.  These identities have an interpretation in terms of the discrete Laplacian interpretation, see \cite{Ken:00} --- though this interpretation is only heuristically relevant here.

\subsection{Boundary Generating Function}

In this section, we find the boundary recurrence relation and solve the recurrence.  We rely on the computations given in Section~\ref{section:general}. Because we have a difference recurrence for each type of black and white vertex, we obtain a matrix equation explaining the boundary recurrence.  This matrix equation is also periodic.
  Due to the nature of the computations, we had to rely heavily on computer algebra in this subsection. 

Let $Z_P(a,b,n)$ denote the partition function of an Aztec diamond of size $2n$, whose Kasteleyn matrix is given by~\eqref{fortress tilings:K}.  Let $\tilde{Z}_P(a,b,n)$ denote the partition function of an Aztec diamond of size $2n-1$, whose Kasteleyn matrix is given by~\eqref{fortress tilings:Ktilde}.   Let $Z_{kl}(i,j,a,b,n)$ count the number of weighted dimer coverings of an Aztec diamond of size $2n$ with the vertices $(4i+2k+1,0)$  and $(0,4j+2l+1)$ removed, whose Kasteleyn matrix is given by~\eqref{fortress tilings:K} but omitting the removed vertices from the matrix. Let $\tilde{Z}_{kl}(i,j,a,b,n)$ count the number of weighted dimer coverings of an Aztec diamond of size $2n-1$  with the vertices $(4i+2k+1,0)$  and $(0,4j+2l+1)$, whose Kasteleyn matrix is given by~\eqref{fortress tilings:Ktilde} but omitting the removed vertices from the matrix. Note that we have the constraint that $Z_{kl}(i,j,a,b,n)=0$ if either $2i+k$ or $2j+l$ (or both) are not in $\{0,1,\dots,2n-1\}$ and $\tilde{Z}_{kl}(i,j,a,b,n)=0$ if either $2i+k$ or $2j+l$ (or both) are not in $\{0,1,\dots,2n-2\}$. 

We introduce the following generating functions: denote
\begin{equation}
      G_{kl}(a,b,x,y,z)=\sum_{n=0}^\infty \sum_{i=0}^{n-1} \sum_{j=0}^{n-1} \frac{Z_{kl}(i,j,a,b,n)}{Z_P(a,b,n)} x^i y^j z^n
\end{equation}
and
\begin{equation}
      \tilde{G}_{kl}(a,b,x,y,z)=\sum_{n=0}^\infty \sum_{i=0}^{n-1} \sum_{j=0}^{n-1}  \frac{\tilde{Z}_{kl}(i,j,a,b,n)}{\tilde{Z}_P(a,b,n)} x^i y^j z^n.
\end{equation}
for $k,l\in\{0,1\}$.  Let
\begin{equation}
      {\bf G}(a,b,x,y,z)= \left( \begin{array}{c}
                            G_{00}(a,b,x,y,z) \\
                            G_{01}(a,b,x,y,z) \\
                            G_{10}(a,b,x,y,z)\\
                            G_{11}(a,b,x,y,z)
                           \end{array} \right)
\end{equation}
and similarly, let $\tilde{\bf G}$ denote the corresponding vector for $\tilde{G}_{kl}$.  
\begin{lemma} \label{fortress tilings:lemma:bfgequations}
The boundary generating functions $G$ and $\tilde{G}$ satisfy the following recurrences
\begin{equation} \label{twoperiodic:generating1}
      {\bf G}(a,b,x,y,z)= {\bf A}(a,b,x,y). \tilde{\bf G}\left(\frac{1}{2a},\frac{1}{2b},x,y,z \right)+ {\bf B}(a)\frac{z}{1-z}
\end{equation}
and
\begin{equation} \label{twoperiodic:generating2}
      \tilde{\bf G}(c,d,x,y,z) =z {\bf C}(c,d,x,y) .{\bf G}(c,d,x,y,z) + {\bf D}(c,d) \frac{z}{1-z} 
\end{equation}
where
\begin{equation}
      {\bf A} (a,b,x,y)= \left( \begin{array}{cccc}
                          \frac{1}{4a^2} &\frac{y}{4a^2}&\frac{x}{4a^2}&\frac{xy}{4a^2} \\
                          \frac{1}{4a b} &\frac{1}{4a b }&\frac{x}{4a b}&\frac{x}{4a b} \\
                          \frac{1}{4ab} &\frac{y}{4ab}&\frac{1}{4ab}&\frac{y}{4a b} \\
                          \frac{1}{4b^2} &\frac{1}{4b^2}&\frac{1}{4b^2}&\frac{1}{4b^2}
                           \end{array} \right), \hspace{10mm}
       {\bf B}(a)= \left( \begin{array}{c}
                            \frac{1}{2 a} \\
                             0 \\
                             0 \\
                            0
                           \end{array} \right)
\end{equation}
\begin{equation}
      {\bf C} (c,d,x,y)=  \frac{1}{c^2+d^2} \left( \begin{array}{cccc}
                          d^2 & cd y & cd x & c^2 x y \\
                          d^2 & cd  & cd x & c^2  x \\
                          d^2 & cd y & cd  & c^2 y  \\
                          d^2 & cd  & cd  & c^2  \\
                          \end{array} \right) \hspace{5mm}\mbox{and} \hspace{5mm}
       {\bf D}(c,d)= \left( \begin{array}{c}
                            \frac{c}{c^2 +d^2} \\
                             0 \\
                             0 \\
                            0
                           \end{array} \right).
\end{equation}
\end{lemma}

\begin{proof}

To prove the lemma, we will use Lemma~\ref{general:lem:recurrence} and the notation of Section~\ref{section:general} to compute recurrences for $Z_P(a,b,n)$, $\tilde{Z}_P(c,d,n)$, $Z_{st}(i,j,a,b,n)$ and $\tilde{Z}_{st}(i,j,c,d,n)$ for $s,t \in \{0,1\}$.   We first compute the recurrence starting with $Z_P(a,b,n)$ followed by the recurrence starting from $Z_{st}(i,j,a,b,n)$ for $s,t \in \{0,1\}$ which will lead to obtaining~\eqref{twoperiodic:generating1}.  We will then compute the recurrence starting from $\tilde{Z}_P(c,d,n)$ and finally compute the recurrence starting from $\tilde{Z}_{st}(i,j,c,d,n)$ for $s,t \in \{0,1\}$ which will lead to obtaining~\eqref{twoperiodic:generating2}.  

We write the edge weights of the Aztec diamond encoded by $Z_P(a,b,n)$ using the notation from Section~\ref{section:general}.  We find that  the edge weights around the face whose center has coordinates $(2k+1,2l+1)$ are given by 
\begin{equation}\label{fortress:eqn:firstweights}
w_{i,j}(k,l)= \left \{ \begin{array}{ll}
	a & \mbox{if } (k+l)\mod 2=0\\
	b & \mbox{if } (k+l) \mod 2=1 \end{array}
\right.
\end{equation}
for $i,j \in \{0,1\}$ and for all $0\leq k,l \leq 2 n-1$ where $w_{i,j}(k,l)$ is described in Section~\ref{section:general}.  
For the faces with coordinates $(2k+1,2l+1)$ with $0 \leq k, l \leq 2n-1$, the urban renewal factors $\Delta(k,l)$, are given by $2a^2$ if $k+l \mod 2 =0$ and $2b^2$ otherwise for $0\leq k,l \leq 2n-1$ and so $$\prod_{k,l=0}^{2n-1} \Delta(k,l)=(4a b)^{2n^2}.$$
The edge weights under the deformation of an Aztec diamond of size $2n$ to an Aztec diamond of size $2n-1$ as detailed in Section~\ref{section:general} are equal to 
\begin{equation}
\begin{split}\label{fortress:eqn:firstweights1}
&\left\{ \frac{w_{0,0}(k,l+1)}{\Delta(k,l+1)},\frac{w_{0,1}(k+1,l+1)}{\Delta(k+1,l+1)},\frac{w_{0,0}(k,l)}{\Delta(k,l)},\frac{w_{1,1}(k+1,l)}{\Delta(k+1,l)} \right\}\\
&=\left\{ \begin{array}{ll}
\{1/(2a),1/(2b),1/(2a),1/(2b)\} &\mbox{if }k+l\mod 2=0\\
\{1/(2b),1/(2a),1/(2b),1/(2a)\} & \mbox{if }k+ l \mod2=1
\end{array} \right.
\end{split}
\end{equation}
for all $0 \leq k,l \leq 2n-2$.  These edge weights are the same edge weights as encoded by $\tilde{Z}_P(1/(2a),1/(2b),n)$.
  From~\eqref{general:eqn:partition} in Lemma~\ref{general:lem:recurrence}, we conclude
\begin{equation} \label{fortress:eqn:partition1}
Z_P(a,b,n)=(4ab)^{2n^2} \tilde{Z}_P\left(\frac{1}{2a},\frac{1}{2b},n\right)
\end{equation}
We now write the relation starting from $Z_{st}(i,j,a,b,n)$ for $s,t \in \{0,1\}$ using~\eqref{general:eqn:recurrence} and the edge weights given in~\eqref{fortress:eqn:firstweights}.  Recall that $Z_{st}(i,j,a,b,n)$ corresponds to the partition function of the Aztec diamond with weights given by~\eqref{fortress:eqn:firstweights} with the vertices $(4i+2s+1,0)$ and $(0,4j+2t+1)$ removed.  Using the notation from Section~\ref{section:general}, we write $\mathcal{Z}_{2n}(2i+s,2j+t)=Z_{st}(i,j,a,b,n)$. We  now list the contributions for $\tilde{\mathcal{Z}}_{2n-1}(\cdot,\cdot)$.  Under the deformation of an Aztec diamond of size $2n$ to an Aztec diamond of size $2n-1$ as detailed in Section~\ref{section:general}, the edge weights are given by~\eqref{fortress:eqn:firstweights1}.  This means that $\tilde{\mathcal{Z}}_{2n-1}=\tilde{Z}_P(1/(2a),1/(2b),n)$ and $\tilde{\mathcal{Z}}_{2n-1}(\cdot,\cdot)$ can written in terms of $\tilde{Z}_{kl} (\cdot,\cdot, 1/(2a),1/(2b),n)$ where $k,l \in \{0,1\}$ while the removed vertices need to be determined.
By comparing all the different combinations of $\tilde{\mathcal{Z}}_{2n-1}(\cdot,\cdot)$  given in~\eqref{general:eqn:recurrence}, we find that for $k,l \in \{0,1\}$
\begin{equation}
\label{fortress:eqn:recurrencecomp1}
\tilde{\mathcal{Z}}_{2n-1}(2i+s-k,2j+t-l)=\tilde{Z}_{kl}\left(i-k(1-s),j-l(1-t),\frac{1}{2a},\frac{1}{2b},n \right)
\end{equation}
From~\eqref{fortress:eqn:firstweights} we also have
\begin{equation}
\label{fortress:eqn:recurrencecomp2}
	\frac{w_{0,r}(2i+s,0)}{\Delta (2i+s,0)}= \frac{1}{2 a^{1-s} b^s} \hspace{5mm}\mbox{and} \hspace{5mm} \frac{w_{p,1}(0,2j+t)}{\Delta(0,2j+t)} =\frac{1}{2a^{1-t}b^t}
\end{equation} 
for $p,r \in \{0,1\}$. We rewrite part of the sum in~\eqref{general:eqn:recurrence} under a change of summation, that is, we write (i.e. setting $i \mapsto 2i+s$, $j\mapsto 2j+s$ and $n \mapsto 2n$ in the last line of~\eqref{general:eqn:recurrence} and ignore the product of the urban renewal factors)
\begin{equation}
\begin{split}
&\sum_{\substack{k \in \{2i+s-1,2i+s\}\\ l \in \{2j+t-1,2j+t\}}}
\frac{w_{0,2i+s-k}(2i+s,0)}{\Delta(2i+s,0)} \frac{w_{1+2j+t-l,1}(0,2j+t)}{\Delta(0,2j+t)} \tilde{\mathcal{Z}}_{2n-1} (k,l) \mathbbm{I}_{0\leq k,l \leq 2n-2} \\
&= \sum_{ \substack{ k, l \in \{0,1\} \\(i-k(1-s),j-l(1-t)) \\ \not = (-1,-1) }} \frac{ w_{0,k(1-s)+s(1-k)} (2i+s,0) w_{1-(l(1-t)+t(1-l)),1}(0,2j+t) }{ \Delta(2i+s,0) \Delta(0,2j+t) } \\
&\hspace{2mm} \times \tilde{Z}_{kl}\left(i-k(1-s),j-l(1-t),\frac{1}{2a},\frac{1}{2b},n \right)
\end{split}
\end{equation}
where we have used~\eqref{fortress:eqn:recurrencecomp1}.
Using the above equation  and~\eqref{fortress:eqn:recurrencecomp2},  for $s,t \in \{0,1\}$ we  write~\eqref{general:eqn:recurrence} as 
 \begin{equation}
\begin{split}
        Z_{st}(i,j,a,b,n)&= 
\sum_{\substack{k,l\in \{0,1\} \\ (i-k,j-l)\not=(-1,-1)  } } \frac{(4 a b)^{2n^2}}{4 a^{2-s-t} b^{s+t}} \tilde{Z}_{kl}\left(i-k(1-s),j-l(1-t),\frac{1}{2a},\frac{1}{2b},n \right)\\ &+ \frac{(4a b)^{2n^2}}{2a}\tilde{Z}_P\left(\frac{1}{2 a},\frac{1}{2b} ,n \right)  \mathbbm{I}_{(i,j,s,t)=(0,0,0,0)}
\end{split}
\end{equation}
 We divide both sides of the above equation by $Z_P(a,b,n)$ and use \eqref{fortress:eqn:partition1} which gives
\begin{equation}
 \frac{Z_{st}(i,j,a,b,n)}{Z_P(a,b,n)} = \sum_{\substack{k,l\in \{0,1\} \\ (i-k,j-l)\not=(-1,-1)  } } \frac{ \tilde{Z}_{kl}\left(i-k(1-s),j-l(1-t),\frac{1}{2a},\frac{1}{2b},n \right)}{4a^{2-s-t}b^{s+t} \tilde{Z}_P\left(\frac{1}{2 a},\frac{1}{2b} ,n \right)} +\frac{\mathbbm{I}_{(i,j,s,t)=(0,0,0,0)}}{2a} 
\end{equation}
For the recurrence equation given in the equation above, we multiply by $x^iy^jz^n$ and sum over the relevant quantities which gives
\begin{equation}
\begin{split}
G_{st}(a,b,x,y,z)=\sum_{k,l \in \{0,1\}} \tilde{G}_{kl}(a,b,x,y,z) \frac{x^{k(1-s)} y^{l (1-t)} }{4 a^{2-s-t} b^{s+t}}+ \frac{ z}{2a(1-z)} \mathbbm{I}_{(s,t)=(0,0)} 
\end{split}
\end{equation}
for $s,t \in \{0,1\}$ which is exactly equal to the row $2s+t+1$ of~\eqref{twoperiodic:generating1} for $s,t \in \{0,1\}$.

We write the edge weights of the Aztec diamond encoded by $\tilde{Z}_P(a,b,n)$ using the notation from Section~\ref{section:general}.  We find that  the edge weights around the face whose center has coordinates $(2k+1,2l+1)$ are given by 
\begin{equation}\label{fortress:eqn:secondweights}
w_{i,j}(k,l)= \left \{ \begin{array}{ll}
	c & \mbox{if } (i+j) \mod 2=1 \mbox{ and }(k+l)\mod 2=0\\
	d & \mbox{if } (i+j) \mod 2=0 \mbox{ and }(k+l)\mod 2=0\\
	d & \mbox{if } (i+j) \mod 2=1 \mbox{ and }(k+l)\mod 2=1\\
	c & \mbox{if } (i+j) \mod 2=0 \mbox{ and }(k+l)\mod 2=1\\
\end{array}
\right.
\end{equation}
for $i,j \in \{0,1\}$ and for all $0\leq k,l \leq 2 n-2$ where $w_{i,j}(k,l)$ is described in Section~\ref{section:general}.  
For the faces with coordinates $(2k+1,2l+1)$ with $0 \leq k, l \leq 2n-2$, the urban renewal factors $\Delta(k,l)$, are given by $c^2+d^2$ for all $0\leq k,l \leq 2n-2$ and so $$\prod_{k,l=0}^{2n-2} \Delta(k,l)=(c^2+d^2)^{(2n-1)^2}.$$
The edge weights under the deformation of an Aztec diamond of size $2n-1$ to an Aztec diamond of size $2n-2$ as detailed in Section~\ref{section:general} are equal to 
\begin{equation}
\begin{split}\label{fortress:eqn:secondweights1}
&\left\{ \frac{w_{0,0}(k,l+1)}{\Delta(k,l+1)},\frac{w_{0,1}(k+1,l+1)}{\Delta(k+1,l+1)},\frac{w_{0,0}(k,l)}{\Delta(k,l)},\frac{w_{1,1}(k+1,l)}{\Delta(k+1,l)} \right\}\\
&=\left\{ \begin{array}{ll}
\{c/(c^2+d^2),c/(c^2+d^2),c/(c^2+d^2),c/(c^2+d^2)\} &\mbox{if }k+l\mod 2=0\\
\{d/(c^2+d^2),d/(c^2+d^2),d/(c^2+d^2),d/(c^2+d^2)\} & \mbox{if }k+ l \mod2=1
\end{array} \right.
\end{split}
\end{equation}
for all $0 \leq k,l \leq 2n-3$.  These edge weights are the same edge weights as encoded by ${Z}_P(c/(c^2+d^2),d/(c^2+d^2),n-1)$.
  From~\eqref{general:eqn:partition} in Lemma~\ref{general:lem:recurrence}, we conclude
\begin{equation} 
\begin{split}
\label{fortress:eqn:partition2}
\tilde{Z}_P(c,d,n)&=(c^2+d^2)^{(2n-1)^2} {Z}_P\left(\frac{c}{c^2+d^2},\frac{d}{c^2+d^2},n-1\right)\\
&= (c^2+d^2)^{2n-1} Z_P(c,d,n-1)
\end{split}
\end{equation}
 where the last line follows by applying a gauge transformation which multiplies all the white vertices by $c^2+d^2$.

We now write the relation starting from $\tilde{Z}_{st}(i,j,a,b,n)$ for $s,t \in \{0,1\}$ using~\eqref{general:eqn:recurrence} and the edge weights given in~\eqref{fortress:eqn:secondweights}.  Recall that $\tilde{Z}_{st}(i,j,a,b,n)$ corresponds to the partition function of the Aztec diamond with weights given by~\eqref{fortress:eqn:secondweights} with the vertices $(4i+2s+1,0)$ and $(0,4j+2t+1)$ removed.  Using the notation from Section~\ref{section:general}, we write $\mathcal{Z}_{2n-1}(2i+s,2j+t)=\tilde{Z}_{st}(i,j,a,b,n)$. We now list the contributions for $\tilde{\mathcal{Z}}_{2n-2}(\cdot,\cdot)$.  Under the deformation of an Aztec diamond of size $2n-1$ to an Aztec diamond of size $2n-2$ as detailed in Section~\ref{section:general}, the edge weights are given by~\eqref{fortress:eqn:secondweights1}.  This means that $\tilde{\mathcal{Z}}_{2n-2}={Z}_P(c/(c^2+d^2),d/(c^2+d^2),n-1)$ and $\tilde{\mathcal{Z}}_{2n-2}(\cdot,\cdot)$ can written in terms of ${Z}_{kl} (\cdot,\cdot, c/(c^2+d^2),d/(c^2+d^2),n-1)$ where $k,l \in \{0,1\}$ while the removed vertices need to be determined.
By comparing all the different combinations of $\tilde{\mathcal{Z}}_{2n-2}(\cdot,\cdot)$  given in~\eqref{general:eqn:recurrence}, we find that for $k,l \in \{0,1\}$
\begin{equation}
\label{fortress:eqn:recurrencecomp3}
\tilde{\mathcal{Z}}_{2n-2}(2i+s-k,2j+t-l)={Z}_{kl}\left(i-k(1-s),j-l(1-t),\frac{c}{c^2+d^2},\frac{d}{c^2+d^2},n-1 \right)
\end{equation}
From~\eqref{fortress:eqn:secondweights}, we write out the edge weights that are found in the right-hand side of~\eqref{general:eqn:recurrence}.  These are given by
\begin{equation}
\begin{split}
\label{fortress:eqn:recurrencecomp4}
	\frac{w_{0,r}(2i+s,0)}{\Delta (2i+s,0)}&= \frac{c^{(1-r)s+r(1-s)}d^{1-((1-r)s+r(1-s))}}{c^2+d^2} \hspace{5mm} \mbox{and}\\
\frac{w_{p,1}(0,2j+t)}{\Delta(0,2j+t)} &=\frac{c^{1-((1-t)p+(1-p)t)} d^{(1-p)t+(1-t)p)}}{c^2+d^2}
\end{split}
\end{equation} 
for $p,r \in \{0,1\}$.
We rewrite part of the sum in~\eqref{general:eqn:recurrence} under a change of summation, that is, we write (i.e. setting $i \mapsto 2i+s$, $j\mapsto 2j+s$ and $n \mapsto 2n-1$ in the last line of~\eqref{general:eqn:recurrence} and ignoring the product of the urban renewal factors)
\begin{equation}
\begin{split} \label{fortress:eqn:recurrencecomp5}
&\sum_{\substack{k \in \{2i+s-1,2i+s\}\\ l \in \{2j+t-1,2j+t\}}}
\frac{w_{0,2i+s-k}(2i+s,0)}{\Delta(2i+s,0)} \frac{w_{1+2j+t-l,1}(0,2j+t)}{\Delta(0,2j+t)} \tilde{\mathcal{Z}}_{2n-2} (k,l) \mathbbm{I}_{0\leq k,l \leq 2n-3} \\
&= \sum_{ \substack{ k, l \in \{0,1\} \\(i-k(1-s),j-l(1-t)) \\ \not = (-1,-1) }} \frac{ w_{0,k(1-s)+s(1-k)} (2i+s,0) w_{1-(l(1-t)+t(1-l)),1}(0,2j+t) }{ \Delta(2i+s,0) \Delta(0,2j+t) } \\
&\hspace{2mm} \times  {Z}_{kl}\left(i-k(1-s),j-l(1-t),\frac{c}{c^2+d^2},\frac{d}{c^2+d^2},n-1 \right)
\end{split}
\end{equation}
where we use have used~\eqref{fortress:eqn:recurrencecomp3}.
In the above equation we have
\begin{equation} \label{fortress:eqn:recurrencecomp6}
 \frac{w_{0,k(1-s)+s(1-k)} (2i+s,0)}{\Delta(2i+s,0)} = \frac{c^k d^{1-k}}{c^2+d^2}
\end{equation}
which is be seen by evaluating the cases for $s=0$ and $s=1$ separately and using~\eqref{fortress:eqn:recurrencecomp4}.  We also find
\begin{equation} \label{fortress:eqn:recurrencecomp7}
\frac{w_{1-(l(1-t)+t(1-l)),1}(0,2j+t) }{\Delta(0,2j+t)} = \frac{c^l d^{1-l}}{c^2+d^2}.
\end{equation}
Using~\eqref{fortress:eqn:recurrencecomp5},~\eqref{fortress:eqn:recurrencecomp6} and~\eqref{fortress:eqn:recurrencecomp7} for $s,t \in \{0,1\}$ the recurrence in~\eqref{general:eqn:recurrence} is equal to
 \begin{equation}
\begin{split}
& \tilde{Z}_{st}(i,j,c,d,n)
=\frac{c}{c^2+d^2} (c^2+d^2)^{(2n-1)^2} Z_P\left(\frac{c}{c^2+d^2},\frac{d}{c^2+d^2},n-1\right) \mathbbm{I}_{(i,j,s,t)=(0,0,0,0)} \\
&+\sum_{\substack{k,l\in \{0,1\} \\ (i-k,j-l) \\ \not=(-1,-1)  } } \frac{d^{2-k-l} c^{k+l}}{(c^2+d^2)^2} (c^2+d^2)^{(2n-1)^2} Z_{kl}\left(i-k(1-s),j-l(1-t),\frac{c}{c^2+d^2},\frac{d}{c^2+d^2},n-1 \right)\\
&=c (c^2+d^2)^{2n-2} Z_P(c,d,n-1)  \mathbbm{I}_{(i,j,s,t)=(0,0,0,0)} \\
&+\sum_{\substack{k,l\in \{0,1\} \\ (i-k,j-l)\not=(-1,-1)  } } d^{2-k-l} c^{k+l}(c^2+d^2)^{2n-2}Z_{kl}\left(i-k(1-s),j-l(1-t),{c},{d},n-1 \right)
\end{split}
\end{equation}
where the last line follows by applying a gauge transformation which multiplies all the white vertices in each expression by $c^2+d^2$. We divide the above equation by the partition function recurrence given in~\eqref{fortress:eqn:partition2} which gives
\begin{equation}\begin{split}
 \frac{\tilde{Z}_{st}(i,j,c,d,n)}{\tilde{Z}_P(a,b,n)}&=\sum_{\substack{k,l\in \{0,1\} \\ (i-k,j-l)\not=(-1,-1)  } } \frac{d^{2-k-l} c^{k+l}}{c^2+d^2}\frac{Z_{kl}(i-k(1-s),j-l(1-t),c,d,n-1)}{Z_P(c,d,n-1)} \\&+\frac{c}{c^2+d^2}\mathbbm{I}_{(i,j,s,t)=(0,0,0,0)}. 
\end{split}
 \end{equation}
For the recurrence equation given in the equation above, we multiply by $x^iy^jz^n$ and sum over the relevant quantities. This gives
\begin{equation}
\tilde{G}_{st} (a,b,x,y,z)=\sum_{k,l\{0,1\}} G_{kl}(a,b,x,y,z) \frac{c^{k+l}d^{2-k-l} x^{k(1-s)} y^{l(1-t)}}{c^2+d^2} + \frac{cz}{(c^2 +d^2)(1-z)} \mathbbm{I}_{(s,t)=(0,0)}
\end{equation}
which is exactly equal to row $2s+t+1$ of~\eqref{twoperiodic:generating2} for $s,t \in \{0,1\}$.

\end{proof}

Let 
\begin{equation}
\mathbf{M}(a,b,x,y) = \mathbf{A}(a,b,x,y) .\mathbf{C}\left((2a)^{-1},(2b)^{-1},x,y \right) .\mathbf{A}\left((2a)^{-1},(2b)^{-1},x,y\right). \mathbf{C}(a,b,x,y)
\end{equation}
and define the vectors
\begin{equation}
\mathbf{B_1}(a,b,x,y) = \mathbf{A}(a,b,x,y) . \mathbf{D} \left( (2a)^{-1},(2b)^{-1} \right) + \mathbf{B}(a)
\end{equation}
and
\begin{equation}
\mathbf{B_2}(a,b,x,y)= \mathbf{A}(a,b,x,y). \mathbf{C}\left((2a)^{-1},(2b)^{-1},x,y \right) . \mathbf{B_1}\left((2a)^{-1},(2b)^{-1},x,y\right).
\end{equation}

\begin{lemma}\label{fortress tilings:lemma:bfgequations2}
For $n=4m$, the $n^{th}$ coefficient of $z$ of the generating function $\mathbf{G}(a,b,x,y,z)$ is given by
\begin{equation}
      \sum_{i=0}^{m-1} \mathbf{\Gamma}\mathbf{ \Lambda}^i \mathbf{\Gamma}^{-1} . \left( \mathbf{B_1} + \mathbf{B_2} \right)
\end{equation}
where $(\mathbf{\Lambda}, \mathbf{\Gamma})$ is the eigensystem of $\mathbf{M}(a,b,x,y)$.  Explicitly, the eigenvalues of $\mathbf{M}(a,b,x,y)$ are given by
\begin{equation}
      \lambda_{2i+j+1} (a,b,x,y)= \beta_i (a,b,x)^2 \beta_j(a,b,y)^2 
\end{equation}
and the eigenvectors are given by
\begin{equation}
      v_{2i+j+1}= \left( \begin{array}{c}
                           \frac{ \left((b^2 - a^2) x -(-1)^i \sqrt{x} \sqrt{(a^4 + b^4) x + a^2 b^2 (1 + x^2)}\right) \left((b^2 - a^2) y -(-1)^j \sqrt{y}\sqrt{(a^4 + b^4) y + a^2 b^2 (1 + y^2)}\right)}{a^2 b^2 (1+x)(1+y)} \\
			    \frac{ ((b^2 - a^2) x -(-1)^i \sqrt{x} \sqrt{(a^4 + b^4) x + a^2 b^2 (1 + x^2)})}{a b(1+x)} \\
			   \frac{(b^2 - a^2) y -(-1)^j \sqrt{y}\sqrt{(a^4 + b^4) y + a^2 b^2 (1 + y^2)}}{a b(1+y)} \\
			    1
			    \end{array} \right)
\end{equation}
\end{lemma}

Although the above lemma does give the boundary generating function for a two-periodic Aztec diamond, the expression is complicated.  We believe that the expression given in Lemma~\ref{fortress tilings:lemma:K1boundary} is more feasible for potential asymptotic computations.

\begin{proof}
 From Lemma~\ref{fortress tilings:lemma:bfgequations}, we have a generating function equation for $\mathbf{G}$ and $\tilde{\mathbf{G}}$. We write
 \begin{equation}
 \begin{split}
  \mathbf{G}(a,b,x,y,z)& = z\mathbf{A}(a,b,x,y). \mathbf{C}\left((2a)^{-1},(2b)^{-1},x,y\right). \mathbf{G}\left((2a)^{-1},(2b)^{-1},x,y,z \right) \\ 
  &+\left(\mathbf{A}(a,b,x,y). \mathbf{D}\left((2a)^{-1}, (2b)^{-1} \right) +\mathbf{B} (a) \right) \frac{z}{1-z}
  \end{split}
 \end{equation}
Apply the above equation to itself, we write
\begin{equation}\begin{split}
  \mathbf{G}(a,b,x,y,z) &= z^2 \mathbf{M}(a,b,x,y) \mathbf{G}(a,b,x,y,z)+ \frac{z}{1-z} \mathbf{B_1}(a,b,x,y) + \frac{z^2}{1-z} \mathbf{B_2} (a,b,x,y)\\
			 &= \sum_{k=0}^\infty z^{2k} (\mathbf{M}(a,b,x,y))^{k}  .\left(\frac{z}{1-z} \mathbf{B_1}(a,b,x,y) + \frac{z^2}{1-z} \mathbf{B_2} (a,b,x,y)\right)
  \end{split}
 \end{equation}
by using the expansion of a geometric series of matrices. The above equation can be solved but it does not appear to give a tractable answer. As we are interested in the $n^{th}$ coefficient, we use the above expansion and the expansion of $(1-z)^{-1}$ to find the coefficient of $z^n$.  This is given by
\begin{equation}
    \sum_{k=0}^{m-1}  (\mathbf{M}(a,b,x,y))^{k}  .\left(\mathbf{B_1}(a,b,x,y) + \mathbf{B_2} (a,b,x,y)\right).
\end{equation}
The eigenvalues and eigenvectors of $\mathbf{M}$ can be verified by $\mathbf{M}.v=\lambda v$ where $\lambda$ is the eigenvalue for the eigenvector $v$.
\end{proof}

We now prove Lemma~\ref{fortress tilings:lemma:K1boundary}.

\begin{proof}[Proof of Lemma \ref{fortress tilings:lemma:K1boundary}]
From Lemma~\ref{fortress tilings:lemma:bfgequations2}, we write the $n^{th}$ coefficient of $\mathbf{G}(a,b,x,y,z)$ as
\begin{equation} \label{fortress tilings:lemmaproof:K1boundary}
       \sum_{i=0}^{m-1} \sum_{j=1}^4 \mathbf{X_j} \l_j^i
\end{equation}
where $\{\l_i\}_{i=1}^4$ are the eigenvalues of $\mathbf{M}(a,b,x,y)$ and $\mathbf{X_i}$ are four column vectors which are the coefficients of $a_i$ in the following expression 
\begin{equation}
    \mathbf{\Gamma} . \mathrm{diag}\left(a_1,a_2,a_3,a_4 \right) . \mathbf{\Gamma^{-1}}.\left(\mathbf{B_1}(a,b,x,y) + \mathbf{B_2} (a,b,x,y)\right).
\end{equation}
where $\mathrm{diag}(a_1,a_2,a_3,a_4)$ denotes a diagonal matrix with four entries $a_1, \dots, a_4$.  We rewrite~\eqref{fortress tilings:lemmaproof:K1boundary} as
\begin{equation}
    \sum_{k=0}^{m-1} \mathbf{Y} . \left( \begin{array}{rrrr}
                                   1 &  1 &  1  &  1 \\
                                   1 &  -1 & 1  & -1 \\
                                   1 & 1 &  -1  & -1 \\
                                   1 & -1 & -1  &  1
                                  \end{array} \right).   \left(\begin{array}{c}
                                  \l_1^i \\ \l_2^i \\ \l_3^i \\ \l_4^i \end{array}\right)
\end{equation}
where 
\begin{equation}
    \mathbf{Y}=\left[\begin{array}{rrrr}
                \mathbf{X_1} & \mathbf{X_2} & \mathbf{X_3} & \mathbf{X_4}
               \end{array} \right] .\left( \begin{array}{rrrr}
                                   1 &  1 &  1  &  1 \\
                                   1 &  -1 & 1  & -1 \\
                                   1 & 1 &  -1  & -1 \\
                                   1 & -1 & -1  &  1
                                  \end{array} \right) 
\end{equation}
A computation shows that for
\begin{equation}
    \mathbf{D_1}=16ab(a^2+b^2)^2\mathrm{diag}\left( \begin{array}{c}
                                                1\\
                                                \sqrt{y} \sqrt{(b^4 + a^4) y + a^2 b^2 (1 + y^2)}\\
                                                \sqrt{x} \sqrt{(b^4 + a^4) x + a^2 b^2 (1 + x^2)}\\
                                               \sqrt{x} \sqrt{(b^4 + a^4) x + a^2 b^2 (1 + x^2)} \sqrt{y} \sqrt{(b^4 + a^4) y + a^2 b^2 (1 + y^2)}
                                               \end{array}  \right)  
\end{equation}
that

\begin{equation}
\mathbf{N}(a,b,x,y)= \mathbf{Y} . \mathbf{D_1}
\end{equation}
where $\mathbf{N}(a,b,x,y)$ is the matrix defined in~\eqref{fortress tilings:N} and 
\begin{equation}\begin{split}
    \left( \begin{array}{c}
     \alpha_0^i(a,b,x) \alpha_0^i(a,b,y) \\
     \alpha_0^i(a,b,x) \alpha_1^i(a,b,y)\\
     \alpha_1^i(a,b,x) \alpha_0^i(a,b,y)\\
     \alpha_1^i(a,b,x) \alpha_1^i(a,b,y)
    \end{array} \right)
&= \mathbf{D_1}^{-1}.
                                               \left( \begin{array}{cccc}
                                   1 &  1 &  1  &  1 \\
                                   1 & -1 & 1  & -1 \\
                                   1 & 1 &  -1  & -1 \\
                                   1 & -1 & -1  &  1
                                  \end{array} \right).   \left(\begin{array}{c}
                                  \l_1^i \\ \l_2^i \\ \l_3^i \\ \l_4^i \end{array}\right)\end{split}
\end{equation}
where $\alpha_j$ is defined in the statement of Lemma~\ref{fortress tilings:lemma:K1boundary} for $j\in \{0,1\}$. This means that the coefficient of $z^n$ in the expression $\mathbf{G}(a,b,x,y,z)$ is equal to 
\begin{equation}
    \mathbf{N} (a,b,x,y).  \left( \begin{array}{c}
     \alpha_0^i(a,b,x) \alpha_0^i(a,b,y) \\
     \alpha_0^i(a,b,x) \alpha_1^i(a,b,y)\\
     \alpha_1^i(a,b,x) \alpha_0^i(a,b,y)\\
     \alpha_1^i(a,b,x) \alpha_1^i(a,b,y)
    \end{array} \right).
\end{equation}
We extract the relevant coefficient of the above equation to find $|L (a,b,i,j)|$ which means we now only need to compute the sign of $L(a,b,i,j)$.  Because the Kasteleyn orientation is the same as the Kasteleyn orientation as the Aztec diamond with uniform weights,  the same computation from the proof of Lemma~\ref{uniform:lem:sign} holds. We conclude  that the sign of $L(a,b,i,j)$ is given by $-\mathtt{i}^{i+j+1}$.

\end{proof}

\subsection{Generating Function of $K^{-1}$}

For notational purposes in the proof, we write $G=G(a,1,\mathtt{w}, \mathtt{b})$ and let 
\begin{equation} \label{fortress tilings:gdef1}
\left. G \right|_{x=(i,j)} = \sum_{\substack{ x=(x_1,x_2),x_1=i,x_2=j \\ y \in \mathtt{B}}} K^{-1} (x,y) \mathtt{w}^x \mathtt{b}^y = \sum_{y \in \mathtt{B}} K^{-1}((i,j),y) \mathtt{b}^y w_1^i w_2^j
\end{equation}
where $\mathtt{w}^x=w_1^{x_1} w_2^{x_2}$, $\mathtt{b}^y=b_1^{y_1} b_2^{y_2}$,
for $i \in \{1,2n-1\}$ or $j \in \{0,2n \}$. As an abuse of notation, for $j \in \{0,2n\}$, we will also write
\begin{equation}  \label{fortress tilings:gdef2}
\left. G\right|_{x=(x_1,j)}= \sum_{\substack{1 \leq i \leq 2n-1,\\ i\mod2=1}} \left. G \right|_{x=(i,j)}= \sum_{\substack{1 \leq i \leq 2n-1,\\ i\mod2=1}} \sum_{y \in \mathtt{B}} K^{-1}((i,j),y) \mathtt{b}^y  w_1^i w_2^j
\end{equation} 
and for $i \in \{1,2n-1\}$ 
\begin{equation}  \label{fortress tilings:gdef3}
\left. G\right|_{x=(i,x_2)}= \sum_{\substack{0 \leq j \leq 2n,\\ j\mod2=0}} \left. G \right|_{x=(i,j)}=\sum_{\substack{0 \leq j \leq 2n,\\ j\mod2=0}} \sum_{y \in \mathtt{B}} K^{-1}((i,j),y) \mathtt{b}^y w_1^i w_2^j
\end{equation} 
We will also denote
\begin{equation}  \label{fortress tilings:gdef4}
\left. G\right|_{ \substack{x=(x_1,j) \\ x \in \mathtt{W}_i}}= \sum_{\substack{ 0 \leq x_1 \leq 2n-1, x_2 =j \\ x \in \mathtt{W}_i, y \in \mathtt{B} }} K^{-1}(x,y) \mathtt{w} ^x \mathtt{b}^y
\end{equation} 
for $j \in \{0,2n\}$ and $i \in \{0,1 \}$.

The proof of Theorem~\ref{fortress tilings:mainthm} has similar structure to the proof of Theorem~\ref{thm:uniform} albeit it is slightly more involved: we have to use the additional recurrences derived from $K^*\cdot K\cdot K^{-1}=K^*\cdot \mathbbm{I}$ and $K^{-1}\cdot K\cdot K^*=\mathbbm{I}\cdot K^*$, where $K^*$ denotes the complex transpose of $K$.

\begin{proof}[Proof of Theorem~\ref{fortress tilings:mainthm}]
Recall that $K^*$ is the conjugate transpose of $K$ which means that the rows of $K^*$ are indexed by white vertices, the columns are indexed by black vertices and the entries corresponding to vertical edges have sign $-\mathtt{i}$.  We remind the reader that $K$ is a sparse matrix: each row has at most four nonzero, one for each neighbor of the vertex indexing the row.

Applying $K^{*}$ to both sides of the equation  $K\cdot K^{-1}=\mathbbm{I}$, we obtain $K^{*} \cdot K\cdot K^{-1}=K^{*}\cdot \mathbbm{I}$.  To obtain the left-hand side of this equation, we have that $K^{*}\cdot K$ is an operator on white vertices with 
\begin{equation} \label{fortress tilings:KstarKdef}
	(K^{*}\cdot K) f(x)= \sum_{\substack{w \sim b\\ w \in \mathtt{W}}} \sum_{\substack{b \sim x \\ b\in \mathtt{B}}} K^{*} (w,b) K(b,x) f(w)
\end{equation}
where $f:\mathtt{W} \to \mathtt{W}$, $b \sim x, b \in \mathtt{B}$ means that $b$ is a nearest neighbored black vertex to $x$ and $w \sim b,w \in \mathtt{W}$ means that $w$ is a nearest neighbored white vertex to $b$.  In the above sum, if the coefficient of $f(\tilde{w})$ for $\tilde{w} \in \mathtt{W}$ and $\tilde{w} \not = x$, is given by $K^*(\tilde{w},b_1) K(b_1, x)+ K^*(\tilde{w},b_2) K(b_2,x)$ for $b_1 \not= b_2$ with $b_1, b_2 \in \mathtt{B}$, then this coefficient is zero. This follows from the fact that $K^*$ is the conjugate transpose of $K$, that $K$ has the Kasteleyn orientation and the choice of weighting. Because $K$ and $K^*$ are sparse matrices, it is possible to expand the double sum in~\eqref{fortress tilings:KstarKdef} to obtain an expression for $K^*\cdot K f(x)$, taking care to include the boundary of the Aztec diamond.  
To do so, we first list the possible choices for $w$ in~\eqref{fortress tilings:KstarKdef} for the different possibilities of $x=(x_1, x_2) \in \mathtt{W}$, where the first item in the following list corresponds to $x=(x_1,x_2)$ in the interior while the remaining items correspond to $x=(x_1,x_2)$ close to  or on the boundary.
\begin{itemize}
\item If  $3 \leq x_1 \leq 2n-3$ and $2\leq x_2\leq 2n-2$, the possible choices for $w$ are $x$, $ x\pm 2 e_1$ and $x \pm 2e_2$.  
\item If $x_1=1$ and $2 \leq x_2 \leq 2n-2$, the possible choices for $w$ are $x$, $x +2e_1$ and $x -2 e_2$.
\item If $x_1=1$ and $x_2=0$, the possible choices for $w$ are $x$, $x+2 e_1$ and $x+e_1-e_2$.
\item If $x_1=1$ and $x_2=2n$, the possible choices for $w$ are $x$, $x-2e_2$ and $x+e_1-e_2$.
\item If $x_1=2n-1$ and $2 \leq x_2 \leq 2n-2$, the possible choices for $w$ are $x$, $x -2e_1$ and $x +2 e_2$.
\item If $x_1=2n-1$ and $x_2=0$, the possible choices for $w$ are $x$, $x+2 e_2$ and $x-e_1+e_2$.
\item If $x_1=2n-1$ and $x_2=2n$, the possible choices for $w$ are $x$, $x-2e_1$ and $x-e_1+e_2$.
\item If $3 \leq x_1 \leq 2n-3$ and $x_2=0$, the possible choices for $w$ are $x$, $x+2 e_1$,$x+2 e_2$, $x+e_1-e_2$ and $x-e_1+e_2$.
\item If $3 \leq x_1 \leq 2n-3$ and $x_2=2n$, the possible choices for $w$ are $x$, $x-2 e_1$,$x-2 e_2$, $x+e_1-e_2$ and $x-e_1+e_2$.
\end{itemize}
For the last two items of the above list, notice that $x+e_2-e_1$ and $x+e_1-e_2$ changes  the parity of the white vertex, that is $x+e_2-e_1,x+e_2-e_1 \in \mathtt{W}_{1-i}$ for $x \in \mathtt{W}_i$ for $i \in \{0,1\}$.   From the above list, we will evaluate the coefficient of $f(w)$ given in~\eqref{fortress tilings:KstarKdef} for each possibility of $w$.
 For example, if we set $w=x+2e_1$, from the above list we have $x_1 \not = 2n-1$ or $x_2 \not= 2n$ and so the coefficient of $f(w+2e_1)$ is equal to $$ K^*(x+2e_1,x+ e_1) K(x+e_1,x) \delta_{x_1<2n-1} \delta_{x_2 <2n}=a\delta_{x_1<2n-1} \delta_{x_2 <2n}$$
because $K(x+e_1,x)=a^{1-i}$ and $K^*(x+2 e_1,x+e_1)=a^{i}$ for $x \in \mathtt{W}_i$ and $i \in \{0,1\}$.  Continuing for the rest of the choices of $w$ in the above list, we find that the expansion of the double sum in~\eqref{fortress tilings:KstarKdef} is given by
\begin{equation} \label{fortress tilings:KstarK1}
\begin{split}
(K^{*} K) f(x)=& a(f(x+2 e_1) \delta_{x_2<2n} \delta_{x_1<2n-1}+ f (x+2 e_2) \delta_{x_2<2n} \delta_{x_1>1}\\& +f (x-2 e_1) \delta_{x_2>0} \delta_{x_1>1}+f (x-2 e_2) \delta_{x_2>0} \delta_{x_1<2n-1})\\&+2\left(1+a^2 -\delta_{x_2=0}( \delta_{x\in \mathtt{W}_0}+a^2 \delta_{x \in \mathtt{W}_1}) -\delta_{x_2=2n} (a^2\delta_{x \in \mathtt{W}_0}+\delta_{x \in \mathtt{W}_1}) \right)f(x)\\&+  a \mathtt{i}\delta_{x_2=0} \left(- f (x+e_2-e_1) \delta_{x_1>1} + f(x+e_1-e_2) \delta_{x_1<2n-1} \right) \\&+  a \mathtt{i}\delta_{x_2=2n} \left( f (x+e_2-e_1) \delta_{x_1>1} -f(x+e_1-e_2) \delta_{x_1<2n-1} \right)
 \end{split}
\end{equation}
 Using~\eqref{fortress tilings:KstarK1} we find that an entry of the matrix equation $K^*\cdot K\cdot K^{-1}=K^*\cdot \mathbbm{I}$ is given by
\begin{equation} \label{fortress tilings:KstarK}
\begin{split}
 & a(K^{-1} (x+2 e_1,y) \delta_{x_2<2n} \delta_{x_1<2n-1}+ K^{-1} (x+2 e_2,y) \delta_{x_2<2n} \delta_{x_1>1}\\& +K^{-1} (x-2 e_1,y) \delta_{x_2>0} \delta_{x_1>1}+K^{-1} (x-2 e_2,y) \delta_{x_2>0} \delta_{x_1<2n-1})\\&+2\left(1+a^2 -\delta_{x_2=0}( \delta_{x\in \mathtt{W}_0}+a^2 \delta_{x \in \mathtt{W}_1}) -\delta_{x_2=2n} (a^2\delta_{x \in \mathtt{W}_0}+\delta_{x \in \mathtt{W}_1}) \right)K^{-1}(x,y)\\&+  a \mathtt{i}\delta_{x_2=0} \left(- K^{-1} (x+e_2-e_1,y) \delta_{x_1>1} + K^{-1}(x+e_1-e_2,y) \delta_{x_1<2n-1} \right) \\&+  a \mathtt{i}\delta_{x_2=2n} \left( K^{-1} (x+e_2-e_1,y) \delta_{x_1>1} - K^{-1}(x+e_1-e_2,y) \delta_{x_1<2n-1} \right) =K^{*}(\delta_{x+ \cdot=y}(x))
 \end{split}
  \end{equation}
for $x=(x_1,x_2)\in\mathtt{W}$ and $y\in \mathtt{B}$ and where we denote 
\begin{equation}
	K^*( \delta_{x+\cdot=y} (x)) = \sum_{ \tilde{x} \sim x, \tilde{x} \in \mathtt{B}} K^*(x,\tilde{x}) \delta_{\tilde{x}=y}
\end{equation}
where $\tilde{x} \sim x$ means that $\tilde{x}$ is a nearest neighbored black vertex to $x$.
A consequence of~\eqref{fortress tilings:KstarK} is that it `moves' the white vertices, that is, there is no change in the black vertices in the above equation.

We multiply both sides of equation~\eqref{fortress tilings:KstarK} by $\mathtt{w}^x\mathtt{b}^y=w_1^{x_1} w_2^{x_2} b_1^{y_1} b_2^{y_2}$ and sum over all the white and black vertices of the Aztec diamond.  We then rewrite each term of~\eqref{fortress tilings:KstarK} using $G(a,1,\mathtt{w},\mathtt{b})$, an expression with white vertices on the top or bottom boundary and an expression with $x_1=1$ or $x_1=2n-1$. These terms can both be written using~\eqref{fortress tilings:gdef1}, ~\eqref{fortress tilings:gdef2} and ~\eqref{fortress tilings:gdef3}.  We give an example of this computation for the first term in~\eqref{fortress tilings:KstarK}
\begin{equation}
\begin{split}
&      \sum_{\substack{x \in \mathtt{W} \\ y \in \mathtt{B}}} \mathtt{w}^x \mathtt{b}^y K^{-1}(x+2e_1,y)\delta_{x_2<2n} \delta_{x_1<2n-1}
\\&= \frac{ \left( \sum_{x \in\mathtt{W},y\in \mathtt{B}}-\sum_{\substack{x_1=1,x_2\not = 0 \\ x \in \mathtt{W}, y\in \mathtt{B}}} - \sum_{\substack{x_2=0 \\ x\in \mathtt{W}, y \in \mathtt{B}}} \right) \mathtt{w}^x \mathtt{b}^y K^{-1}(x,y) }{w_1^2 w_2^2}\\
&= \frac{G - \left. G \right|_{x=(1,x_2)} +\left. G \right|_{x=(1,0)}-  \left. G \right|_{x=(x_1,0)}}{w_1^2 w_2^2}
\end{split}
\end{equation}
The computations for the remaining terms in~\eqref{fortress tilings:KstarK} are given in Appendix~\ref{AppendixA2}.  Using the above equation and the computations in Appendix~\ref{AppendixA2}, we collect terms and we find that
\begin{equation}
\begin{split}
&(2(1+a^2)+a(w_1^2+w_1^{-2})(w_2^2+w_2^{-2}))G \\
&-a \left. G \right|_{x=(1,x_2)}\left( \frac{ 1+w_2^4}{ w_1^2 w_2^2} \right) -a \left. G \right|_{x=(2n-1,x_2)} \left( \frac{w_1^2(1+w_2^4)}{w_2^2} \right) \\
&+a \left. G\right|_{x=(1,0)} \left( \frac{1+w_2^4}{w_1^2 w_2^2} -\frac{w_2^2}{w_1^2} -w_1^{-2} \mathtt{i} \right)
+a \left. G \right|_{x=(1,2n)} \left( \frac{1+w_2^4}{w_1^2 w_2^2} -\frac{1}{w_1^2 w_2^2} + w_1^{-2} \mathtt{i} \right)\\
 &+a \left. G\right|_{x=(2n-1,0)} \left( \frac{w_1^2(1+w_2^4)}{ w_2^2} -w_1^2 w_2^2 +w_1^2 \mathtt{i} \right)
+a \left. G \right|_{x=(2n-1,2n)} \left( \frac{w_1^2(1+w_2^4)}{ w_2^2} -\frac{w_1^2}{ w_2^2} - w_1^{2} \mathtt{i} \right)\\
&+\left. G \right|_{\substack{x=(x_1,0) \\ x\in \mathtt{W}_0}} \left( -a(w_1^{-2} w_2^{-2} +w_1^2 w_2^{-2} ) -a w_1^2 \mathtt{i} +a w_1^{-2} \mathtt{i} -2 \right) \\
 &+\left. G \right|_{\substack{x=(x_1,0)\\ x \in \mathtt{W}_1}} \left( -a(w_1^{-2} w_2^{-2} +w_1^2 w_2^{-2} ) -a w_1^2 \mathtt{i} +a w_1^{-2} \mathtt{i} -2a^2 \right) \\
&+\left. G \right|_{\substack{x=(x_1,2n) \\ x\in \mathtt{W}_0}} \left( -a(w_1^{2} w_2^{2} +w_1^{-2} w_2^{2} ) +a w_1^2 \mathtt{i} -a w_1^{-2} \mathtt{i} -2a^2 \right) \\
 &+\left. G \right|_{\substack{x=(x_1,2n)\\ x \in \mathtt{W}_1}} \left( -a(w_1^{2} w_2^{2} +w_1^{-2} w_2^{2} ) +a w_1^2 \mathtt{i} -a w_1^{-2} \mathtt{i} -2 \right)= \sum_{\substack{ x\in \mathtt{W} \\ y \in \mathtt{B}}} K^* (\delta_{x+\cdot=y}(x) )\mathtt{w}^x \mathtt{b}^y \\
\end{split} \label{fortress tilings:whiterecurrence}
\end{equation}

We remark that the coefficient of $G$ in the above equation is exactly $c_{\partial}(w_1,w_2)$.  We also remark that the coefficients of $\left. G \right|_{x=(x_1,0),x \in \mathtt{W}_i}$ and  $\left. G \right|_{x=(x_1,2n),x \in \mathtt{W}_i}$ in the above equation are given by  $s_{i,0}(w_1,w_2)$ and  $w_2^{-2n}s_{i,2n}(w_1,w_2)$ respectively for $i \in \{0,1\}$, where the latter we divide by $w_2^{2n}$ because there is already a factor of $w_2^{2n}$ in $\left.G\right|_{x=(x_1,2n),x \in \mathtt{W}_i}$.

The terms  $\left. G \right|_{x=(1,x_2)}$ and $ \left. G \right|_{x=(2n-1,x_2)}$ involve white vertices away from the top and bottom boundary. For both terms, we use the recurrence relation obtained $K\cdot K^{-1} = \mathbbm{I}$ to write these expressions in terms of the boundary vertices.  For $\left. G \right|_{x=(1,x_2)}$, we extract the entry $(x+e_2,y)$ of $K^{-1}\cdot K$ and compare the entry $(x+e_2,y)$ of $\mathbbm{I}$ where $x \in \mathtt{W}$ with $x=(1,x_2)$ and $y \in \mathtt{B}$.  This gives
\begin{equation}\label{fortress tilings:side recurrence}
K^{-1} (x+e_2+e_1,y) = \frac{1}{a^{1-i}} \delta_{x+e_2=y} -\mathtt{i} K^{-1}(x,y)
\end{equation}
for $x=(1,x_2) \in \mathtt{W}_i$ for $i \in \{0,1\}$.  We multiply the above equation by $w_1 w_2^{x_2} \mathtt{b}^y$ and sum over all the black vertices and the white vertices $x=(1,x_2)$. This gives
\begin{equation}
w_2^{-2}  \left. G\right|_{x=(1,x_2)}-w_2^2 \left. G\right|_{x=(1,0)} =
\sum_{i \in \{0,1\}} \sum_{\substack{x_1=1 \\x_2 \not =2n, x\in\mathtt{W}_i \\ y \in\mathtt{B}}}\frac{1}{a^{1-i}} \delta_{x+e_2=y}\mathtt{w}^x\mathtt{b}^y -\mathtt{i} \left. G\right|_{x=(1,0)} +\mathtt{i}  \left. G\right|_{x=(1,2n)} .\label{fortress tilings:side recurrence2} 
\end{equation}
Rearranging~\eqref{fortress tilings:side recurrence2}, we find that 
\begin{equation}
\left. G \right|_{x=(1,x_2)} (w_2^{-2} + \mathtt{i} ) = \sum_{i \in \{0,1\}} \sum_{\substack{x_1=1 \\x_2 \not =2n, x\in\mathtt{W}_i \\ y \in\mathtt{B}}}\frac{1}{a^{1-i}} \delta_{x+e_2=y}\mathtt{w}^x\mathtt{b}^y +w_2^{-2}  \left. G \right|_{x=(1,0)}+ \mathtt{i} \left. G \right|_{x=(1,2n)}.\label{fortress tilings:side recurrence3}
\end{equation}

Similar to the computation of $\left. G \right|_{x=(1,x_2)}$ given above, to compute $\left. G \right|_{x=(2n-1,x_2)}$ we use the recurrence relation
\begin{equation}
K^{-1}(x+e_2+e_1,y) \mathtt{i} = \frac{1}{a^i} \delta_{x+e_1=y}-K^{-1}(x,y)  
\end{equation}
for $x=(2n-1,x_2) \in \mathtt{W}_i$ for $i\in \{0,1\}$, which is derived from the entry $(x+e_1,y)$ of the matrix equation $K\cdot K^{-1}=\mathbbm{I}$ with $x=(2n-1,x_2)$.  We now repeat the steps that we used in computing $\left. G \right|_{x=(1,x_2)}$ for the computation of $\left. G \right|_{x=(2n-1,x_2)}$.  This gives 
\begin{equation}
\left. G \right|_{x=(2n-1,x_2)} (w_2^{-2}  \mathtt{i}+1 ) = \sum_{i \in \{0,1\}} \sum_{\substack{x_1=2n-1 \\x_2 \not =2n, x\in\mathtt{W}_i \\ y \in\mathtt{B}}}\frac{1}{a^{i}} \delta_{x+e_1=y}\mathtt{w}^x\mathtt{b}^y+w_2^{-2} \mathtt{i}  \left. G \right|_{x=(2n-1,0)}+  \left. G \right|_{x=(2n-1,2n)}\label{fortress tilings:side recurrence4}
\end{equation}

From~\eqref{fortress tilings:side recurrence3} and~\eqref{fortress tilings:side recurrence4}, we now have expressions for $G|_{x=(1,x_2)}$ and $G|_{x=(2n-1,x_2)}$ which can then be substituted into~\eqref{fortress tilings:whiterecurrence}.  After doing this, we find that the coefficient of $a\left. G \right|_{x=(1,0)}$ in~\eqref{fortress tilings:whiterecurrence} is given by
\begin{equation}
	\frac{1+w_2^4}{w_1^2 w_2^2} -\frac{w_2^2}{w_1^2} - w_1^{-2} \mathtt{i} - \frac{w_2^2 (1+w_2^4)}{(w_2^{-2}+\mathtt{i}) w_1^2w_2^2} =0.
\end{equation}
Similarly, we also find that the coefficients of $a\left. G \right|_{x=(1,2n)}$, $a\left. G \right|_{x=(2n-1,0)}$ and $a\left. G \right|_{x=(2n-1,2n)}$ are also zero.  This means we reduce~\eqref{fortress tilings:whiterecurrence} to 
\begin{equation}
\begin{split} \label{fortress tilings:whiteGF}
&c_{\partial}(w_1,w_2)G +\sum_{i \in \{0,1\}} s_{i,0}(w_1,w_2) \left. G \right|_{\substack{ x=(x_1,0) \\ x \in \mathtt{W}_i}} +w_2^{-2n} \sum_{i \in \{0,1\}} s_{i,2n}(w_1,w_2) \left. G \right|_{\substack{ x=(x_1,2n) \\ x \in \mathtt{W}_i}}=d_w(\mathtt{w},\mathtt{b})
\end{split}
\end{equation} 
where we have defined
\begin{equation}
\begin{split} \label{fortress tilings:eqn:dw}
d_w(\mathtt{w},\mathtt{b})&= \frac{a(1+w_2^4)}{w_1^2 w_2^2 (w_2^{-2}+\mathtt{i})} \sum_{i \in \{0,1\}} \sum_{\substack{x_1=1 \\x_2 \not =2n, x\in\mathtt{W}_i \\ y \in\mathtt{B}}}\frac{1}{a^{1-i}} \delta_{x+e_2=y}\mathtt{w}^x\mathtt{b}^y \\
&+\frac{a (1+w_2^4) w_1^{2}}{w_2^2(w_2^{-2}\mathtt{i} +1)}
\sum_{i \in \{0,1\}} \sum_{\substack{x_1=2n-1 \\x_2 \not =2n, x\in\mathtt{W}_i \\ y \in\mathtt{B}}}\frac{1}{a^{i}} \delta_{x+e_1=y}\mathtt{w}^x\mathtt{b}^y
+\sum_{\substack{ x \in \mathtt{W} \\ y\in \mathtt{B}}} K^* (\delta_{x+\cdot=y)}(x)) \mathtt{w}^x \mathtt{b}^y
\end{split}
\end{equation}
We evaluate $d_w(\mathtt{w},\mathtt{b})$ in Appendix~\ref{AppendixA3} and we find that
\begin{equation}
	d_w(\mathtt{w},\mathtt{b})=d(\mathtt{w},\mathtt{b})+\frac{d_{\mathrm{sides}}(\mathtt{w},\mathtt{b})}{c_{\partial}(b_1,b_2)} f_r(w_1^4 b_1^4)
\end{equation}
and we refer the reader there for the details of the computation.

To find $G(a,1,\mathtt{w},\mathtt{b})$ in the above equation~\eqref{fortress tilings:whiteGF}, we need to find expressions for $G|_{x=(x_1,0),x \in \mathtt{W}_i}$ and $G|_{x=(x_1,2n),x \in \mathtt{W}_i}$ for $i\in \{0,1\}$.  This is a very similar to the computation used to find~\eqref{fortress tilings:whiteGF} and so we outline the main steps.  First, we  derive a relation in terms of entries of $K^{-1}$ from an entry-wise expansion of the matrix equation $K^{-1}\cdot K\cdot K^*=\mathbbm{I}\cdot K^*$.  This relation is analogous to the relation given in~\eqref{fortress tilings:KstarK} but acts on the black vertices of $K^{-1}$, keeping the white vertices of $K^{-1}$ fixed.  To this new relation, we multiply both sides of the equation by $\mathtt{b}^y=b_1^{y_1}b_2^{y_2}$ and sum over all $y \in \mathtt{B}$. We then apply the long simplification procedure that is detailed above to find~\eqref{fortress tilings:whiteGF}.  We find that for a fixed $x \in \mathtt{W}$ we have
\begin{equation}
\begin{split} \label{fortress tilings:blackGF}
&c_{\partial}(b_1,b_2) \sum_{y \in \mathtt{B}} K^{-1}(x,y)\mathtt{b}^y+\sum_{j \in \{0,1\}} s_{j,0}(b_2,b_1) \sum_{ y=(0,y_2) \in \mathtt{B}_j} K^{-1}(x,(0,y_2)) b_2^{y_2} \\
&+ \sum_{j \in \{0,1\}} s_{j,2n}(b_2,b_1) \sum_{ y=(2n,y_2) \in \mathtt{B}_j} K^{-1}(x,(2n,y_2)) b_2^{y_2} = \tilde{d}_b(x,\mathtt{b})
\end{split}
\end{equation}
where
\begin{equation}
\begin{split} \label{fortress tilings:dtildeb}
\tilde{d}_b(x,\mathtt{b}) &= \frac{a(1+b_1^4)}{b_1^2 b_2^2 (b_1^{-2}+\mathtt{i})} \sum_{j \in \{0,1\}} \sum_{\substack{y_2=1 \\y_1 \not =2n, y\in\mathtt{B}_j }}\frac{1}{a^{1-j}} \delta_{y-e_2=x}\mathtt{b}^y \\
&+\frac{a (1+b_1^4) b_2^{2}}{b_1^2(b_1^{-2}\mathtt{i} +1)}
\sum_{j \in \{0,1\}} \sum_{\substack{y_2=2n-1 \\y_1 \not =2n, y\in\mathtt{B}_j}}\frac{1}{a^{j}} \delta_{y+e_1=x}\mathtt{b}^y
+\sum_{\substack{  y\in \mathtt{B}}} K^* (\delta_{y+\cdot=x}(y))  \mathtt{b}^y
\end{split}
\end{equation}
where 
\begin{equation}
K^*(\delta_{y+\cdot =x}(y)) =\sum_{ \tilde{y}\sim y,\tilde{y} \in \mathtt{W}} K^*(\tilde{y},y) \delta_{\tilde{y}=x}
\end{equation}
where $\tilde{y} \sim y$ means that $\tilde{y}$ is a nearest neighbored vertex to $y$.  Note that~\eqref{fortress tilings:blackGF} can also be obtained by symmetry using~\eqref{fortress tilings:whiteGF} because the model is symmetric under the map $w_1 \mapsto b_2$ and $w_2 \mapsto b_1$. 
  We now choose $x$ in~\eqref{fortress tilings:blackGF} to be either $x=(x_1,0)\in\mathtt{W}_0$, $x=(x_1,0) \in \mathtt{W}_1$, $x=(x_1,2n)\in\mathtt{W}_0$ or $x=(x_1,2n) \in \mathtt{W}_1$ and for each case, we multiply~\eqref{fortress tilings:blackGF} by $w_1^{x_1} w_2^{2nk}$ and sum over $(x_1,2nk) \in \mathtt{W}_i$ for $i,k \in \{0,1\}$.  This gives an expression for $G|_{x=(x_1,2nk),x\in \mathtt{W}_i}$ for $i,k \in \{0,1\}$ because by applying the above operations to the first term in~\eqref{fortress tilings:blackGF} gives
\begin{equation}
	c_{\partial}(b_1,b_2) \sum_{(x_1,2nk) \in \mathtt{W}_i}w_1^{x_1} w_2^{2nk} \sum_{y\in \mathtt{B}} K^{-1}((x_1,2nk),y) \mathtt{b}^y= c_{\partial}(b_1,b_2)\left. G \right|_{x=(x_1,2nk), x \in \mathtt{W}_i}. 
\end{equation}
We substitute these expressions for $G|_{x=(x_1,2nk),x\in \mathtt{W}_i}$ for $i,k\in \{0,1\}$ into~\eqref{fortress tilings:whiteGF}.  We obtain
\begin{equation}
\begin{split} \label{fortress tilings:penultimateeqn}
&c_{\partial}(w_1,w_2)G -
\sum_{i,j\in \{0,1\}} \sum_{k,l \in \{0,2n\}} \sum_{(x_1,k) \in \mathtt{W}_i} \sum_{(l,y_2) \in \mathtt{B}_j} s_{i,k}(w_1,w_2) s_{j,l}(b_2,b_1) \frac{K^{-1} ((x_1,k),(l,y_2))}{c_{\partial}(b_1,b_2)} w_1^{x_1} b_2^{y_2} \\
&+ \sum_{i,k \in \{0,1\}} \sum_{x=(x_1,2nk)\in \mathtt{W}_i}s_{i,2kn}(w_1,w_2)\frac{\tilde{d}_b((x_1,2n k),\mathtt{b})}{c_{\partial}(b_1,b_2)} w_1^{x_1}  = d_w(\mathtt{w},\mathtt{b})\end{split}
\end{equation} 
A computation in Appendix~\ref{AppendixA4} shows that
\begin{equation}
\sum_{i,k \in \{0,1\}} \sum_{x=(x_1,2nk)\in \mathtt{W}_i} s_{i,2kn}(w_1,w_2){\tilde{d}_b((x_1,2n k),\mathtt{b})}w_1^{x_1} =d_{\mathrm{sides}}(\mathtt{w},\mathtt{b}) f_{n/2}(w_1^4 b_1^4) 
 \label{fortress tilings:dtilde}
\end{equation}
and hence~\eqref{fortress tilings:penultimateeqn} is equal to~\eqref{fortress tilings:mainthmeqn1}.

We have computed $K^{-1}( (x_1,0) , (0,y_2))$ in Lemma~\ref{fortress tilings:lemma:K1boundary}.  By symmetry, we have that
 $$K^{-1}((x_1,0),(0,y_2))=K^{-1}((2n-x_1,2n),(2n,2n-y_2))$$
which gives the last equation in~\eqref{fortress tilings:mainthmeqn2}. 
To compute $|K^{-1}((x_1,2n),(0,2n-y_2))|$, we use the expression from $|K^{-1}((x_1,0),(0,y_2))|$ provided we interchange between $a$ and $b$.  To obtain the sign of $K^{-1}((x_1,2n),(0,2n-y_2))$ as given in the third equation of
~\eqref{fortress tilings:mainthmeqn2}, we use the same sign found in the proof of  Lemma~\ref{uniform:lem:sign} for the vertices removed from the top and left boundaries of the Aztec diamond because the Kasteleyn orientation are the same and account for the fact that we assigned $L(a,b,i,j,n)$ a sign (equal  to $-\mathtt{i}^{i+j+1}$).   From $K^{-1}((x_1,2n),(0,2n-y_2))$, we find $K^{-1}((2n-x_1,0),(2n,y_2))$ by symmetry. 
\end{proof}

\begin{appendix}
\section{The Matrix $\mathbf{N}$} \label{appendix}

In this subsection, we give the matrix $\mathbf{N}$.  We have

\begin{equation}
 \begin{split}\label{fortress tilings:N}
      \mathbf{N}(a,b,w,z)_{1,1} &= 4 b^5 (1 + w z)+ 
	  a^2 b^3 (7 + 3 w + 3 z + 5 w z) + 2 a^4 b (2 + w + z + w z)  \\
      \mathbf{N}(a,b,w,z)_{1,2} &=-4 b^7 z (1 + w z)- a^2 b^5 z (5 + 5 w + 5 z + 5 w z + 2 w z^2)\\
	    &-  a^4 b^3 z (2 + 6 w + 5 z + 5 w z + z^2 + w z^2) -  2 a^6 b z (w + z + w z)  \\ 
	\mathbf{N}(a,b,w,z)_{1,3} &=-4 b^7 w (1 + w z) - a^2 b^5 w (5 + 5 w + 5 z + 5 w z + 2 w^2 z) \\ 
  &- a^4 b^3 w (2 + 5 w + w^2 + 6 z + 5 w z + w^2 z)  - 2 a^6 b w (w + z + w z) \\
  \mathbf{N}(a,b,w,z)_{1,3} &= 4 b^9 w z (1 + w z) + 
  a^2 b^7 w z (3 + 7 w + 7 z + 5 w z + 2 w^2 z + 2 w z^2)\\ 
  & +  a^4 b^5 w z (3 + 9 w + 2 w^2 + 9 z + 10 w z + w^2 z + 2 z^2 + 
     w z^2 + w^2 z^2) \\ 
     &  + a^6 b^3 w z (7 + 6 w + w^2 + 6 z + 7 w z + w^2 z + z^2 + w z^2) 
 \\& + 2 a^8 b w z (2 + w + z + w z)
\end{split}
\end{equation}
\begin{equation}
 \begin{split}
  \mathbf{N}(a,b,w,z)_{2,1} &=  2 a b^4 (1 + w + w z) + 
  a^3 b^2 (5 + 7 w + z + w z)+4 a^5 (1 + w)  \\
  \mathbf{N}(a,b,w,z)_{2,2} &=
	-2 a b^6 (2 + z + w z + w z^2) 
	- a^3 b^4 (7 + 3 w + 7 z + 3 w z + 2 z^2 + 6 w z^2) \\
	&- a^5 b^2 (4 + 2 w + 7 z + 7 w z + 3 z^2 + 3 w z^2) 
	-4 a^7 (1 + w) z\\
   \mathbf{N}(a,b,w,z)_{2,3} &= 
	  - 2 a b^6 w (1 + w + w z)
	- a^3 b^4 w (6 + 9 w + w^2 + 2 z + w z + w^2 z)\\
	&-   a^5 b^2 w (9 + 9 w + 2 w^2 + z + w z)
	-4 a^7 w (1 + w)    \\
  \mathbf{N}(a,b,w,z)_{2,4} &= 2 a b^8 w (2 + z + w z + w z^2)  \\
	 & +  a^3 b^6 w (5 + 5 w + 6 z + 3 w z + w^2 z + 3 z^2 + 6 w z^2 + 
     w^2 z^2) \\ 
	   &+  a^5 b^4 w (2 + 5 w + w^2 + 5 z + 10 w z + w^2 z + 7 z^2 + 5 w z^2 +  2 w^2 z^2) \\
	 & +  a^7 b^2 w (2 w + 7 z + 9 w z + 2 w^2 z + 3 z^2 + 3 w z^2) 
	  +4 a^9 w (1 + w) z 
 \end{split}
\end{equation}
\begin{equation}
 \begin{split}
    \mathbf{N}(a,b,w,z)_{3,1} &=  2 a b^4 (1 + z + w z) 
	      + a^3 b^2 (5 + w + 7 z + w z)+
	      4 a^5 (1 + z)  \\
  \mathbf{N}(a,b,w,z)_{3,2} &=
    - 2 a b^6 z (1 + z + w z) 
  - a^3 b^4 z (6 + 2 w + 9 z + w z + z^2 + w z^2)\\
  &- a^5 b^2 z (9 + w + 9 z + w z + 2 z^2)
   -4 a^7 z (1 + z) \\
  \mathbf{N}(a,b,w,z)_{3,3} &=
    -   2 a b^6 (2 + w + w z + w^2 z)
    -   a^3 b^4 (7 + 7 w + 2 w^2 + 3 z + 3 w z + 6 w^2 z)\\
      &- a^5 b^2 (4 + 7 w + 3 w^2 + 2 z + 7 w z + 3 w^2 z) 
      -4 a^7 w (1 + z) \\
  \mathbf{N}(a,b,w,z)_{3,4} &=  2 a b^8 z (2 + w + w z + w^2 z) \\
&  +  a^3 b^6 z (5 + 6 w + 3 w^2 + 5 z + 3 w z + 6 w^2 z + w z^2 + 
     w^2 z^2) \\
   & + a^5 b^4 z (2 + 5 w + 7 w^2 + 5 z + 10 w z + 5 w^2 z + z^2 + w z^2 + 
     2 w^2 z^2)\\
 & +  a^7 b^2 z (7 w + 3 w^2 + 2 z + 9 w z + 3 w^2 z + 2 w z^2) 
  + 4 a^9 w z (1 + z) 
     \end{split}
\end{equation}
\begin{equation}
 \begin{split}
  \mathbf{N}(a,b,w,z)_{4,1} &= 
  a^2 b^3 (3 + w + z + w z)+ 2 a^4 b (5 + w + z)
 + 8 a^6 b^{-1}   \\
  \mathbf{N}(a,b,w,z)_{4,2} &= -a^2 b^5 (2 + z + w z + z^2 + w z^2)
	  - a^4 b^3 (5 + w + 5 z + w z + 6 z^2 + 2 w z^2) \\ &
	  -2 a^6 b (2 + 5 z + w z + 3 z^2)  -8 a^8 b^{-1} z   \\
  \mathbf{N}(a,b,w,z)_{4,3} &= - a^2 b^5 (2 + w + w^2 + w z + w^2 z)
   - a^4 b^3 (5 + 5 w + 6 w^2 + z + w z + 2 w^2 z) \\
  &-   2 a^6 b (2 + 5 w + 3 w^2 + w z)  -8 a^8 b^{-1} w  \\
  \mathbf{N}(a,b,w,z)_{4,4} &=a^2 b^7 (4 + 2 w + 2 z + 3 w z + w^2 z + w z^2 + w^2 z^2) \\
  & +   a^4 b^5 (7 + 7 w + 2 w^2 + 7 z + 6 w z + 3 w^2 z + 2 z^2 + 
     3 w z^2 + 5 w^2 z^2)\\
  &  + 
  a^6 b^3 (4 + 7 w + 3 w^2 + 7 z + 9 w z + 6 w^2 z + 3 z^2 + 
     6 w z^2 + 5 w^2 z^2)\\
  &+   2 a^8 b (2 w + 2 z + 5 w z + 3 w^2 z + 3 w z^2) +8 a^{10} b^{-1} w z 
 \end{split}
\end{equation}

\section{Appendix to Proof of Theorem~\ref{fortress tilings:mainthm}} 

\subsection{Generating function for the white vertices}\label{AppendixA2}

The following computations are the simplifications from multiplying~\eqref{fortress tilings:KstarK} by $\mathtt{w}^x \mathtt{b}^y$ and summing over the white and black vertices of the Aztec diamond.  We computed the first term in the proof of the theorem.  The next three terms can computed similarly:
\begin{equation}
\begin{split}
&      \sum_{\substack{x \in \mathtt{W} \\ y \in \mathtt{B}}} \mathtt{w}^x \mathtt{b}^y K^{-1}(x+2e_2,y)\delta_{x_2<2n} \delta_{x_1>1}
\\&= \frac{ \left( \sum_{x \in\mathtt{W},y\in \mathtt{B}}-\sum_{\substack{x_1=2n-1,x_2\not = 0 \\ x \in \mathtt{W}, y\in \mathtt{B}}} - \sum_{\substack{x_2=0 \\ x\in \mathtt{W}, y \in \mathtt{B}}} \right) \mathtt{w}^x \mathtt{b}^y K^{-1}(x,y) }{w_1^{-2} w_2^2}\\
&= \frac{G - \left. G \right|_{x=(2n-1,x_2)} +\left. G \right|_{x=(2n-1,0)}-  \left. G \right|_{x=(x_1,0)}}{w_1^{-2} w_2^2},
\end{split}
\end{equation}
\begin{equation}
\begin{split}
&      \sum_{\substack{x \in \mathtt{W} \\ y \in \mathtt{B}}} \mathtt{w}^x \mathtt{b}^y K^{-1}(x-2e_1,y)\delta_{x_2>0} \delta_{x_1>1}
\\&= \frac{ \left( \sum_{x \in\mathtt{W},y\in \mathtt{B}}-\sum_{\substack{x_1=2n-1,x_2\not = 2n \\ x \in \mathtt{W}, y\in \mathtt{B}}} - \sum_{\substack{x_2=2n \\ x\in \mathtt{W}, y \in \mathtt{B}}} \right) \mathtt{w}^x \mathtt{b}^y K^{-1}(x,y) }{w_1^{-2} w_2^{-2}}\\
&= \frac{G - \left. G \right|_{x=(2n-1,x_2)}+\left. G \right|_{x=(2n-1,2n)} -  \left. G \right|_{x=(x_1,2n)}}{w_1^{-2} w_2^{-2}}
\end{split}
\end{equation}
and
\begin{equation}
\begin{split}
&      \sum_{\substack{x \in \mathtt{W} \\ y \in \mathtt{B}}} \mathtt{w}^x \mathtt{b}^y K^{-1}(x-2e_2,y)\delta_{x_2>0} \delta_{x_1<2n-1}
\\&= \frac{ \left( \sum_{x \in\mathtt{W},y\in \mathtt{B}}-\sum_{\substack{x_1=1,x_2\not = 2n \\ x \in \mathtt{W}, y\in \mathtt{B}}} - \sum_{\substack{x_2=2n \\ x\in \mathtt{W}, y \in \mathtt{B}}} \right) \mathtt{w}^x \mathtt{b}^y K^{-1}(x,y) }{w_1^{2} w_2^{-2}}\\
&= \frac{G - \left. G \right|_{x=(1,x_2)}+\left. G \right|_{x=(1,2n)} -  \left. G \right|_{x=(x_1,2n)}}{w_1^{2} w_2^{-2}}.
\end{split}
\end{equation}
For $i \in \{0,1\}$, we have
\begin{equation}
      \sum_{\substack{x \in \mathtt{W} \\ y \in \mathtt{B}}} \mathtt{w}^x \mathtt{b}^y K^{-1}(x,y)\delta_{x_2=0} \delta_{x \in \mathtt{W}_i}= \left. G \right|_{\substack{x=(x_1,0)\\ x \in \mathtt{W}_i}}
\end{equation}
and
\begin{equation}
      \sum_{\substack{x \in \mathtt{W} \\ y \in \mathtt{B}}} \mathtt{w}^x \mathtt{b}^y K^{-1}(x,y)\delta_{x_2=2n} \delta_{x \in \mathtt{W}_i}= \left. G \right|_{\substack{x=(x_1,2n)\\ x \in \mathtt{W}_i}}.
\end{equation}
For the computation involving the  last four terms on the left-hand side of~\eqref{fortress tilings:KstarK}, we have
\begin{equation}
\begin{split}
      \sum_{\substack{x \in \mathtt{W} \\ y \in \mathtt{B}}} \mathtt{w}^x \mathtt{b}^y K^{-1}(x+e_2-e_1,y)\delta_{x_2=0} \delta_{x_1>1}
 &= w_1^2 \sum_{ \substack{x_1 \not = 2n-1, x_2=0\\ x \in\mathtt{W},y\in \mathtt{B}}} \mathtt{w}^x \mathtt{b}^y K^{-1}(x,y)
\\&= w_1^2 \left( \left. G \right|_{x=(x_1,0)}-\left. G \right|_{x=(2n-1,0)} \right),
\end{split}
\end{equation}
\begin{equation}
\begin{split}
      \sum_{\substack{x \in \mathtt{W} \\ y \in \mathtt{B}}} \mathtt{w}^x \mathtt{b}^y K^{-1}(x+e_1-e_2,y)\delta_{x_2=0} \delta_{x_1<2n-1}
 &= w_1^{-2} \sum_{ \substack{x_1 \not = 1, x_2=0\\ x \in\mathtt{W},y\in \mathtt{B}}} \mathtt{w}^x \mathtt{b}^y K^{-1}(x,y)
\\&= w_1^{-2} \left( \left. G \right|_{x=(x_1,0)}-\left. G \right|_{x=(1,0)} \right),
\end{split}
\end{equation}
\begin{equation}
\begin{split}
      \sum_{\substack{x \in \mathtt{W} \\ y \in \mathtt{B}}} \mathtt{w}^x \mathtt{b}^y K^{-1}(x+e_2-e_1,y)\delta_{x_2=2n} \delta_{x_1>1}
 &= w_1^{2} \sum_{ \substack{x_1 \not = 2n-1, x_2=2n\\ x \in\mathtt{W},y\in \mathtt{B}}} \mathtt{w}^x \mathtt{b}^y K^{-1}(x,y)
\\&= w_1^{2} \left( \left. G \right|_{x=(x_1,2n)}-\left. G \right|_{x=(2n-1,2n)} \right)
\end{split}
\end{equation}
and
\begin{equation}
\begin{split}
      \sum_{\substack{x \in \mathtt{W} \\ y \in \mathtt{B}}} \mathtt{w}^x \mathtt{b}^y K^{-1}(x+e_1-e_2,y)\delta_{x_2=2n} \delta_{x_1<2n-1}
 &= w_1^{-2} \sum_{ \substack{x_1 \not = 1, x_2=2n\\ x \in\mathtt{W},y\in \mathtt{B}}} \mathtt{w}^x \mathtt{b}^y K^{-1}(x,y)
\\&= w_1^{-2} \left( \left. G \right|_{x=(x_1,2n)}-\left. G \right|_{x=(1,2n)} \right)
\end{split}
\end{equation}

\subsection{Computation of $d_w(\mathtt{w},\mathtt{b})$} \label{AppendixA3}

In this subsection, we evaluate $d_w(\mathtt{w},\mathtt{b})$ which is defined in~\eqref{fortress tilings:eqn:dw}.

We first evaluate the sum in the first term of $d_w(\mathtt{w},\mathtt{b})$.  We first split up the sum, then we sum over $y \in\mathtt{B}$ and finally we sum over $x \in \mathtt{W}$ in the following way:
\begin{equation}
\begin{split}
\sum_{i \in \{0,1\}} \sum_{\substack{x_1=1 \\x_2 \not =2n, x\in\mathtt{W}_i \\ y \in\mathtt{B}}}\frac{1}{a^{1-i}} \delta_{x+e_2=y}\mathtt{w}^x\mathtt{b}^y &=
 \sum_{\substack{x_1=1 \\x_2 \not =2n, x\in\mathtt{W}_0 \\ y \in\mathtt{B}}}\frac{1}{a} \delta_{x+e_2=y}\mathtt{w}^x\mathtt{b}^y 
+ \sum_{\substack{x_1=1 \\x_2 \not =2n, x\in\mathtt{W}_1 \\ y \in\mathtt{B}}} \delta_{x+e_2=y}\mathtt{w}^x\mathtt{b}^y \\
&=\sum_{\substack{x=(1,x_2) \in \mathtt{W}_0 \\ x_2\not=2n}} \frac{1}{a} w_1 w_2^{x_2}  b_2^{x_2+1}+  \sum_{\substack{x=(1,x_2) \in \mathtt{W}_1 \\ x_2\not=2n}} w_1 w_2^{x_2} b_2^{x_2+1}\\
&=\frac{1}{a} w_1 b_2  \sum_{j=0}^{n/2-1} b_2^{4j} w_2^{4j} + w_1 w_2^2 b_2^3 \sum_{j=0}^{n/2-1} b_2^{4j} w_2^{4j} \\
&= f_{n/2} (w_2^4 b_2^4) w_1 b_2(a^{-1} +w_2^2 b_2^2) 
\end{split}
\end{equation}
This means that the first term in $d_w(\mathtt{w},\mathtt{b})$ reads
\begin{equation}\label{AppendixA3eqn1}
 \frac{a(1+w_2^4)}{w_1^2 w_2^2 (w_2^{-2}+\mathtt{i})} f_{n/2} (w_2^4 b_2^4) w_1 b_2(a^{-1} +w_2^2 b_2^2) 
\end{equation}

In the same way, as computed above, we evaluate the sum in the second term of $d_w(\mathtt{w},\mathtt{b})$ which gives
\begin{equation}
\begin{split}
\sum_{i \in \{0,1\}} \sum_{\substack{x_1=2n-1 \\x_2 \not =2n, x\in\mathtt{W}_i \\ y \in\mathtt{B}}}\frac{1}{a^{i}} \delta_{x+e_1=y}\mathtt{w}^x\mathtt{b}^y &=
 \sum_{\substack{x_1=2n-1 \\x_2 \not =2n, x\in\mathtt{W}_0 \\ y \in\mathtt{B}}}\delta_{x+e_1=y}\mathtt{w}^x\mathtt{b}^y 
+ \sum_{\substack{x_1=2n-1 \\x_2 \not =2n, x\in\mathtt{W}_1 \\ y \in\mathtt{B}}} \frac{1}{a} \delta_{x+e_1=y}\mathtt{w}^x\mathtt{b}^y \\
&=\sum_{\substack{x=(2n-1,x_2) \in \mathtt{W}_0 \\ x_2\not=2n}}  w_1^{2n-1} w_2^{x_2} b_1^{2n} b_2^{x_2+1}+  \sum_{\substack{x=(2n-1,x_2) \in \mathtt{W}_1 \\ x_2\not=2n}} \frac{1}{a}w_1^{2n-1} w_2^{x_2} b_1^{2n} b_2^{x_2+1}\\
&=w_1^{2n-1} b_1^{2n} b_2  \sum_{j=0}^{n/2-1} b_2^{4j} w_2^{4j} + \frac{1}{a} w_1^{2n-1}b_1^{2n} w_2^2 b_2^3 \sum_{j=0}^{n/2-1} b_2^{4j} w_2^{4j} \\
&= f_{n/2}(w_2^4 b_2^4) w_1^{2n-1} b_1^{2n}   b_2(1 + a^{-1}w_2^2 b_2^2) 
\end{split}
\end{equation}
Using the above evaluation, the second term in $d_w(\mathtt{w},\mathtt{b})$ reads
\begin{equation}
\frac{a (1+w_2^4) w_1^{2}}{w_2^2(w_2^{-2}\mathtt{i} +1)} f_{n/2}(w_2^4 b_2^4) w_1^{2n-1} b_1^{2n}   b_2(1 + a^{-1}w_2^2 b_2^2) \label{AppendixA3eqn2}
\end{equation}
We have that adding and simplifying~\eqref{AppendixA3eqn1} and~\eqref{AppendixA3eqn2} gives the second term in $d(\mathtt{w},\mathtt{b})$.

We now have to evaluate the last term in $d_w(\mathtt{w},\mathtt{b})$ and it remains to show that this term is equal to the first term of $d(\mathtt{w},\mathtt{b})$.  We first expand out the definition of $K^* (\delta_{\tilde{x}=y}(x))$ which gives
\begin{equation}
\begin{split}
      K^* (\delta_{x+\cdot=y}(x))&=\sum_{\tilde{x} \sim x, \tilde{x} \in \mathtt{B}} K^*(x, \tilde{x}) \delta_{\tilde{x}=y} \\
&=\delta_{x_2>0} \left(a^{i} \delta_{x-e_1=y} -\mathtt{i} a^i \delta_{x-e_2=y}  \right)+\delta_{x_2<2n} \left( a^{1-i} \delta_{x+e_1=y}-\mathtt{i} a^{1-i} \delta_{x+e_2=y}  \right).
\end{split}
\end{equation}
where the delta functions $\delta_{x_2>0}$ and $\delta_{x_2<2n}$ account for the boundary of the Aztec diamond.
We multiply both sides by  $\mathtt{w}^x \mathtt{b}^y$ and we obtain
\begin{equation}
 \begin{split}
      &\sum_{\substack{x\in\mathtt{W}\\ y \in \mathtt{B}}} \mathtt{w}^x \mathtt{b}^y K^* (\delta_{x+\cdot=y}(x)) =\\
      &\sum_{i\in\{0,1\}} \sum_{x \in \mathtt{W}_i} (w_1b_1)^{x_1} (w_2 b_2)^{x_2} \left( -a^{i}\mathtt{i} \frac{b_1}{b_2} \delta_{x_2>0} + \frac{a^{i}}{b_1b_2} \delta_{x_2>0} + a^{1-i} b_1 b_2 \delta_{x_2<2n} - a^{1-i}\mathtt{i} \frac{b_2}{b_1}\delta_{x_2<2n} \right).
 \end{split}\label{fortress tilings:finalwhiterecurrence}
\end{equation}
In~\eqref{fortress tilings:finalwhiterecurrence}, we have the following four sums which can be evaluated directly
\begin{equation}
\begin{split}
\sum_{x \in \mathtt{W}_0, x_2>0} (w_1b_1)^{x_1} (w_2 b_2)^{x_2} &= w_2^4 b_2^4 w_1b_1 f_{r}(w_2^4 b_2^4) f_{r}(w_1^4 b_1^4)+w_1^3 b_1^3 w_2^2 b_2^2 f_r(w_1^4 b_1^4)f_r(w_2^4 b_2^4)\\
&=w_1b_1 w_2^2 b_2^2 (w_2^2 b_2^2+w_1^2 b_1^2) f_r(w_1^4 b_1^4)f_r(w_2^4b_2^4)
\end{split}
\end{equation} 
\begin{equation}
\begin{split}
\sum_{x \in \mathtt{W}_0, x_2<2n} (w_1b_1)^{x_1} (w_2 b_2)^{x_2} &=  w_1b_1 f_{r}(w_2^4 b_2^4) f_{r}(w_1^4 b_1^4)+w_1^3 b_1^3 w_2^2 b_2^2 f_r(w_1^4 b_1^4)f_r(w_2^4 b_2^4)\\
&=w_1 b_1(1+w_1^2 b_1^2 w_2^2 b_2^2)f_r(w_1^4 b_1^4)f_r(w_2^4 b_2^4)
\end{split}
\end{equation} 
\begin{equation}
\begin{split}
\sum_{x \in \mathtt{W}_1, x_2>0} (w_1b_1)^{x_1} (w_2 b_2)^{x_2} &= w_2^2 b_2^2 w_1b_1 f_{r}(w_2^4 b_2^4) f_{r}(w_1^4 b_1^4)+w_1^3 b_1^3 w_2^4 b_2^4 f_r(w_1^4 b_1^4)f_r(w_2^4 b_2^4)\\
&=w_1 b_1 w_2^2 b_2^2 (1+w_1^2 b_1^2 w_2^2 b_2^2)f_r(w_1^4 b_1^4 ) f_r(w_2^4 b_2^4)
\end{split}
\end{equation} 
\begin{equation}
\begin{split}
\sum_{x \in \mathtt{W}_1, x_2<2n} (w_1b_1)^{x_1} (w_2 b_2)^{x_2}&=  w_1^3b_1^3 f_{r}(w_2^4 b_2^4) f_{r}(w_1^4 b_1^4)+w_1 b_1 w_2^2 b_2^2 f_r(w_1^4 b_1^4)f_r(w_2^4 b_2^4)\\
&=w_1 b_1(w_1^2 b_1^2 +w_2^2 b_2^2) f_r(w_1^4 b_1^4) f_r(w_2^4b_2^4)
\end{split}
\end{equation} 
where $r=n/2$.  We substitute the above four sums back into~\eqref{fortress tilings:finalwhiterecurrence} and after some simplification, we obtain the first term of $d(\mathtt{w},\mathtt{b})$.

\subsection{Computation of~(\ref{fortress tilings:dtilde})} \label{AppendixA4}

In this subsection, we show~\eqref{fortress tilings:dtilde} which is another computation.  We first expand out the third term of $\tilde{d}_b(x,\mathtt{b})$ by expanding out the definition of $K^*$.  This gives
\begin{equation} \label{AppendixA4Kstar}
K^{*} (\delta_{y+\cdot=x}(y)) = \delta_{y_1>0}\left( - \mathtt{i} a^{1-j} \delta_{x=y-e_2} +a^{1-j} \delta_{x=y+e_1}  \right) +\delta_{y_1<2n} \left( a^j \delta_{x=y-e_1} - \mathtt{i} a^j \delta_{x=y+e_2} \right)
\end{equation}
for $y \in \mathtt{B}_j$ and $j \in \{0,1\}$ and the delta functions $\delta_{y_1>0}$ and $\delta_{y_1<2n}$ account for the boundary of the Aztec diamond.

We now split up the left-hand side in~\eqref{fortress tilings:dtilde} into four different cases $(x_1,0) \in \mathtt{W}_0$, $(x_1,0) \in \mathtt{W}_1$, $(x_1,2n) \in \mathtt{W}_0$ and $(x_1,2n) \in \mathtt{W}_1$.  

The first case we consider is
\begin{equation} 
\sum_{x=(x_1,0) \in \mathtt{W}_0} s_{0,0} (w_1,w_2) \tilde{d}_b(x,\mathtt{b}) w_1^{x_1}=s_{0,0}(w_1,w_2)\sum_{x=(x_1,0) \in \mathtt{W}_0} \tilde{d}_b(x,\mathtt{b}) w_1^{x_1}
\label{AppendixA4:dbtildebbW0}
\end{equation}
and we shall show that this is equal to the first term of $d_{\mathrm{sides}}(\mathtt{w},\mathtt{b})$ defined in~\eqref{dsides}.
We now expand out the three sums of $\tilde{d}_b(x,\mathtt{w})$ which are given in~\eqref{fortress tilings:dtildeb}.  When we insert the first term of $\tilde{d}_b(x,\mathtt{w})$ into~\eqref{AppendixA4:dbtildebbW0} because $(x_1,0)+e_2 \in \mathtt{B}_0$ for $(x_1,0) \in \mathtt{W}_0$, we have the following
\begin{equation}
\begin{split}
\frac{a(1+b_1^4)}{b_1^2b_2^2(b_1^{-2}+\mathtt{i})} \sum_{j \in \{0,1\}}\sum_{\substack{y_2=1,y_1\not = 2n\\ y=(y_1,y_2)\in\mathtt{B}\\ (x_1,0)\in \mathtt{W}_0}} \frac{1}{a^{1-j}} \delta_{y-e_2=x} \mathtt{b}^yw_1^{x_1}
&=\frac{a(1+b_1^4)}{b_1^2b_2^2(b_1^{-2}+\mathtt{i})} \sum_{\substack{y_2=1,y_1\not = 2n\\ (x_1,0)\in \mathtt{W}_0}} \frac{1}{a} \delta_{y-e_2=x} b_1^{y_1}b_2w_1^{x_1}\\
&= b_2^{-2}(1-\mathtt{i} b_1^2) \sum_{(x_1,0)\in \mathtt{W}_0} w_1^{x_1} b_1^{x_1-1} b_2\\
&=f_{n/2}  (w_1^4 b_1^4) w_1 b_2^{-1} (1-\mathtt{i} b_1^2). \label{AppendixA4:bbW0sum1}
\end{split}
\end{equation} 
When we insert the second term of $\tilde{d}_b(x,\mathtt{w})$ into~\eqref{AppendixA4:dbtildebbW0} we obtain zero because $\delta_{y+e_1=x}=0$ for $y=(y_1,2n-1)$ and $x=(x_1,0) \in \mathtt{W}_0$.  When we insert the third term of $\tilde{d}_b(x,\mathtt{w})$ into~\eqref{AppendixA4:dbtildebbW0} because $(x_1,0)+e_i \in \mathtt{B}_{2-i}$ for $(x_1,0)\in \mathtt{W}_i$ and $i \in \{1,2\}$, we find that
\begin{equation}
\begin{split}
\sum_{\substack{(x_1,0) \in \mathtt{W}_0 \\ y \in \mathtt{B}}} K^*(\delta_{y+\cdot=x}(y)) \mathtt{b}^y w_1^{x_1} &= -\mathtt{i} \sum_{(x_1,0) \in \mathtt{W}_0}a w_1^{x_1} b_1^{x_1-1} b_2 +  \sum_{(x_1,0) \in \mathtt{W}_0} a w_1^{x_1} b_1^{x_1+1} b_2\\
&=f_{n/2} (w_1^4 b_1^4) ( -\mathtt{i} a w_1 b_2 +a w_1 b_1^2 b_2) \label{AppendixA4:bbW0sum2}
\end{split}
\end{equation}
where we use the expansion of $K^*(\delta_{x=\tilde{y}}(y))$ given in~\eqref{AppendixA4Kstar} and use the fact that $\delta_{x=y+e_1}=0$ and $\delta_{x=y+e_2}=0$ for $x=(x_1,0) \in \mathtt{W}_0$.  We sum up all the contributions of $\tilde{d}_b(x,\mathtt{w})$ when inserted into~\eqref{AppendixA4:dbtildebbW0}, i.e. summing~\eqref{AppendixA4:bbW0sum1} and~\eqref{AppendixA4:bbW0sum2}, which gives the first term in $d_{\mathrm{sides}}(\mathtt{w},\mathtt{b})$ defined in~\eqref{dsides}.

The second case we consider is
\begin{equation} 
\sum_{x=(x_1,0) \in \mathtt{W}_1} s_{1,0} (w_1,w_2) \tilde{d}_b(x,\mathtt{b}) w_1^{x_1}=s_{1,0}(w_1,w_2)\sum_{x=(x_1,0) \in \mathtt{W}_1} \tilde{d}_b(x,\mathtt{b}) w_1^{x_1}
\label{AppendixA4:dbtildebbW1}
\end{equation}
and we shall show that this is equal to the second term of $d_{\mathrm{sides}}(\mathtt{w},\mathtt{b})$ defined in~\eqref{dsides}.
We now expand out the three sums of $\tilde{d}_b(x,\mathtt{w})$ which are given in~\eqref{fortress tilings:dtildeb}.  When we insert the first term of $\tilde{d}_b(x,\mathtt{w})$ into~\eqref{AppendixA4:dbtildebbW1}, because $(x_1,0)+e_2 \in \mathtt{B}_1$ for $(x_1,0) \in \mathtt{W}_1$, we have the following
\begin{equation}
\begin{split}
\frac{a(1+b_1^4)}{b_1^2b_2^2(b_1^{-2}+\mathtt{i})} \sum_{j \in \{0,1\}}\sum_{\substack{y_2=1,y_1\not = 2n\\ y=(y_1,y_2)\in\mathtt{B}\\ (x_1,0)\in \mathtt{W}_1}} \frac{1}{a^{1-j}} \delta_{y-e_2=x} \mathtt{b}^yw_1^{x_1}
&=\frac{a(1+b_1^4)}{b_1^2b_2^2(b_1^{-2}+\mathtt{i})} \sum_{\substack{y_2=1,y_1\not = 2n\\ (x_1,0)\in \mathtt{W}_1}}  \delta_{y-e_2=x} b_1^{y_1}b_2w_1^{x_1}\\
&=a b_2^{-2}(1-\mathtt{i} b_1^2) \sum_{(x_1,0)\in \mathtt{W}_1} w_1^{x_1} b_1^{x_1-1} b_2\\
&=f_{n/2}  (w_1^4 b_1^4) w_1^3 b_1^2 b_2^{-1} (1-\mathtt{i} b_1^2). \label{AppendixA4:bbW1sum1}
\end{split}
\end{equation} 
When we insert the second term of $\tilde{d}_b(x,\mathtt{w})$ into~\eqref{AppendixA4:dbtildebbW1} we obtain zero because $\delta_{y+e_1=x}=0$ for $y=(y_1,2n-1)$ and $x=(x_1,0) \in \mathtt{W}_0$.  When we insert the third term of $\tilde{d}_b(x,\mathtt{w})$ into~\eqref{AppendixA4:dbtildebbW1}, because $(x_1,0)+e_i \in \mathtt{B}_{i-1}$ for $(x_1,0) \in \mathtt{W}_1$ and $i \in \{1,2\}$, we find that
\begin{equation}
\begin{split}
\sum_{\substack{(x_1,0) \in \mathtt{W}_1 \\ y \in \mathtt{B}}} K^*(\delta_{y+\cdot=x}(y)) \mathtt{b}^y w_1^{x_1} &= -\mathtt{i} \sum_{(x_1,0) \in \mathtt{W}_1} w_1^{x_1} b_1^{x_1-1} b_2 +  \sum_{(x_1,0) \in \mathtt{W}_1}  w_1^{x_1} b_1^{x_1+1} b_2\\
&=f_{n/2} (w_1^4 b_1^4) ( -\mathtt{i}  w_1^3 b_1^2 b_2 + w_1^3 b_1^4 b_2) \label{AppendixA4:bbW1sum2}
\end{split}
\end{equation}
where we use the expansion of $K^*(\delta_{x=\tilde{y}}(y))$ given in~\eqref{AppendixA4Kstar} and use the fact that $\delta_{x=y+e_1}=0$ and $\delta_{x=y+e_2}=0$ for $x=(x_1,0) \in \mathtt{W}_1$.  We sum up all the contributions of $\tilde{d}_b(x,\mathtt{w})$ when inserted into~\eqref{AppendixA4:dbtildebbW1}, i.e. summing~\eqref{AppendixA4:bbW1sum1} and~\eqref{AppendixA4:bbW1sum2}, which gives the second term in $d_{\mathrm{sides}}(\mathtt{w},\mathtt{b})$ defined in~\eqref{dsides}.

The third case we consider is
\begin{equation} 
\sum_{x=(x_1,2n) \in \mathtt{W}_0} s_{0,2n} (w_1,w_2) \tilde{d}_b(x,\mathtt{b}) w_1^{x_1} =s_{0,2n}(w_1,w_2)\sum_{x=(x_1,2n) \in \mathtt{W}_0} \tilde{d}_b(x,\mathtt{b}) w_1^{x_1}
\label{AppendixA4:dbtildettW0}
\end{equation}
and we shall show that this is equal to the third term of $d_{\mathrm{sides}}(\mathtt{w},\mathtt{b})$ defined in~\eqref{dsides}
We now expand out the three sums of $\tilde{d}_b(x,\mathtt{w})$ which are given in~\eqref{fortress tilings:dtildeb}.  When we insert the first term of $\tilde{d}_b(x,\mathtt{w})$ into~\eqref{AppendixA4:dbtildettW0}, we find that it is equal to zero because $\delta_{y-e_2=x}=0$ for $x=(x_1,2n) \in \mathtt{W}_0$ for all $y=(1,y_2) \in \mathtt{B}$. 
When we insert the second term of $\tilde{d}_b(x,\mathtt{w})$ into~\eqref{AppendixA4:dbtildebbW0}, because $(x_1,2n)-e_1 \in \mathtt{B}_1$ for $(x_1,2n) \in \mathtt{W}_0$, we obtain 
\begin{equation}
\begin{split}
\frac{a(1+b_1^4)b_2^2}{b_1^2(b_1^{-2}\mathtt{i}+1)} \sum_{j \in \{0,1\}}\sum_{\substack{y_2=2n-1,y_1\not = 2n\\ y=(y_1,y_2)\in\mathtt{B}\\ (x_1,2n)\in \mathtt{W}_0}} \frac{1}{a^{j}} \delta_{y+e_1=x} \mathtt{b}^yw_1^{x_1}
&=\frac{a(1+b_1^4)b_2^2}{b_1^2(b_1^{-2}\mathtt{i}+1)} \sum_{\substack{y_2=2n-1,y_1\not = 2n\\ (x_1,2n)\in \mathtt{W}_0}} \delta_{y+e_1=x} b_1^{y_1}b_2^{y_2}w_1^{x_1}\\
&= a b_2^{2}(b_1^2-\mathtt{i} ) \sum_{(x_1,2n)\in \mathtt{W}_0} w_1^{x_1} b_1^{x_1-1} b_2^{2n-1}\\
&=af_{n/2}  (w_1^4 b_1^4) w_1  b_2^{2n+1} (b_1^2-\mathtt{i}). \label{AppendixA4:ttW0sum1}
\end{split}
\end{equation} 
When we insert the third term of $\tilde{d}_b(x,\mathtt{w})$ into~\eqref{AppendixA4:dbtildebbW0}, because $(x_1,2n) -e_i \in \mathtt{B}_{2-i}$ for $(x_1,2n) \in \mathtt{W}_0$ and $i \in \{1,2\}$, we find that
\begin{equation}
\begin{split}
\sum_{\substack{(x_1,2n) \in \mathtt{W}_0 \\ y \in \mathtt{B}}} K^*(\delta_{y+\cdot=x}(y)) \mathtt{b}^y w_1^{x_1} &= -\mathtt{i} \sum_{(x_1,2n) \in \mathtt{W}_0} w_1^{x_1} b_1^{x_1+1} b_2^{2n-1} +  \sum_{(x_1,2n) \in \mathtt{W}_0}  w_1^{x_1} b_1^{x_1-1} b_2^{2n-1}\\
&=f_{n/2} (w_1^4 b_1^4) ( -\mathtt{i}  w_1 b_2^{2n-1}b_1^2 + w_1  b_2^{2n-1}) \label{AppendixA4:ttW0sum2}
\end{split}
\end{equation}
where we use the expansion of $K^*(\delta_{x=\tilde{y}}(y))$ given in~\eqref{AppendixA4Kstar} and use the fact that $\delta_{x=y-e_1}=0$ and $\delta_{x=y-e_2}=0$ for $x=(x_1,2n) \in \mathtt{W}_0$.  We sum up all the contributions of $\tilde{d}_b(x,\mathtt{w})$ when inserted into~\eqref{AppendixA4:dbtildettW0}, i.e. summing~\eqref{AppendixA4:ttW0sum1} and~\eqref{AppendixA4:ttW0sum2}, which gives the third term in $d_{\mathrm{sides}}(\mathtt{w},\mathtt{b})$ defined in~\eqref{dsides}.

The final case we need to consider is
\begin{equation} 
\sum_{x=(x_1,2n) \in \mathtt{W}_1} s_{1,2n} (w_1,w_2) \tilde{d}_b(x,\mathtt{b}) w_1^{x_1} =s_{1,2n}(w_1,w_2)\sum_{x=(x_1,2n) \in \mathtt{W}_1} \tilde{d}_b(x,\mathtt{b}) w_1^{x_1}
\label{AppendixA4:dbtildettW1}
\end{equation}
and we shall show that this is equal to the final term of $d_{\mathrm{sides}}(\mathtt{w},\mathtt{b})$ defined in~\eqref{dsides}.
We now expand out the three sums of $\tilde{d}_b(x,\mathtt{w})$ which are given in~\eqref{fortress tilings:dtildeb}.  When we insert the first term of $\tilde{d}_b(x,\mathtt{w})$ into~\eqref{AppendixA4:dbtildettW1}, we find that it is equal to zero because $\delta_{y-e_2=x}=0$ for $x=(x_1,2n) \in \mathtt{W}_1$ for all $y \in \mathtt{B}$. 
When we insert the second term of $\tilde{d}_b(x,\mathtt{w})$ into~\eqref{AppendixA4:dbtildebbW1}, because $(x_1,2n) -e_1 \in \mathtt{B}_0$ for $(x_1,2n)\in \mathtt{W}_1$, we obtain  
\begin{equation}
\begin{split}
\frac{a(1+b_1^4)b_2^2}{b_1^2(b_1^{-2}\mathtt{i}+1)} \sum_{j \in \{0,1\}}\sum_{\substack{y_2=2n-1,y_1\not = 2n\\ y=(y_1,y_2)\in\mathtt{B}\\ (x_1,2n)\in \mathtt{W}_1}} \frac{1}{a^{j}} \delta_{y+e_1=x} \mathtt{b}^yw_1^{x_1}
&=\frac{a(1+b_1^4)b_2^2}{b_1^2(b_1^{-2}\mathtt{i}+1)}\frac{1}{a} \sum_{\substack{y_2=2n-1,y_1\not = 2n\\ (x_1,2n)\in \mathtt{W}_1}} \delta_{y+e_1=x} b_1^{y_1}b_2^{y_2}w_1^{x_1}\\
&= b_2^{2}(b_1^2-\mathtt{i} ) \sum_{(x_1,2n)\in \mathtt{W}_1} w_1^{x_1} b_1^{x_1-1} b_2^{2n-1}\\
&=f_{n/2}  (w_1^4 b_1^4) w_1^3 b_1^2  b_2^{2n+1} (b_1^2-\mathtt{i}). \label{AppendixA4:ttW1sum1}
\end{split}
\end{equation} 
When we insert the third term of $\tilde{d}_b(x,\mathtt{w})$ into~\eqref{AppendixA4:dbtildettW1}, because $(x_1,2n) -e_i \in \mathtt{B}_{i-1}$ for $(x_1,2n) \in \mathtt{W}_1$ and $i \in \{1,2\}$, we find that
\begin{equation}
\begin{split}
\label{AppendixA4:ttW1sum2}
\sum_{\substack{(x_1,2n) \in \mathtt{W}_1 \\ y \in \mathtt{B}}} K^*(\delta_{y+\cdot=x}(y)) \mathtt{b}^y w_1^{x_1} &= -a\mathtt{i} \sum_{(x_1,2n) \in \mathtt{W}_1} w_1^{x_1} b_1^{x_1+1} b_2^{2n-1} +a  \sum_{(x_1,2n) \in \mathtt{W}_1}  w_1^{x_1} b_1^{x_1-1} b_2^{2n-1}\\
&=f_{n/2} (w_1^4 b_1^4) ( -a\mathtt{i}  w_1^3b_1^4 b_2^{2n-1} +a w_3 b_1^2 b_2^{2n-1}) 
\end{split}
\end{equation}
where we use the expansion of $K^*(\delta_{x=\tilde{y}}(y))$ given in~\eqref{AppendixA4Kstar} and use the fact that $\delta_{x=y-e_1}=0$ and $\delta_{x=y-e_2}=0$ for $x=(x_1,2n) \in \mathtt{W}_1$.  We sum up all the contributions of $\tilde{d}_b(x,\mathtt{w})$ when inserted into~\eqref{AppendixA4:dbtildettW1}, i.e. summing~\eqref{AppendixA4:ttW1sum1} and~\eqref{AppendixA4:ttW1sum2}, which gives the final term in $d_{\mathrm{sides}}(\mathtt{w},\mathtt{b})$ defined in~\eqref{dsides}.

\end{appendix}

\bibliographystyle{plain}
\bibliography{ref}

\end{document}